\documentclass[11pt,reqno]{amsart}

\usepackage{amsthm,amsfonts,amsmath,oldgerm,mathrsfs}
\usepackage{amsmath,amsfonts,amscd,amssymb,graphicx, xcolor}
\usepackage{epstopdf}
\usepackage{cancel}
\usepackage{enumitem}

\usepackage{hyperref}

\numberwithin{equation}{section}
\usepackage{fullpage}
\usepackage[normalem]{ulem}
\usepackage{wrapfig}
\usepackage{tikz}
\usepackage{graphicx}
\usepackage{caption}
\usepackage{subcaption}
\usepackage{sidecap}

\newcommand{\constant}{\textrm{constant}}

\renewcommand\d{\partial}

\def\eps{\varepsilon }

\renewcommand\d{\partial}

\newcommand\R{\mathbb R}
\newcommand\C{\mathbb C}
\newcommand{\brh}{\overline{\rho}}
\newcommand{\bu}{\overline{u}}
\newcommand{\bh}{\overline{h}}

\def\eps{\varepsilon}

\newcommand\coker{\hbox{\rm Coker}}

 \newcommand\br{\begin{remark}}
\newcommand\er{\end{remark}}
\newcommand\bp{\begin{pmatrix}}
\newcommand\ep{\end{pmatrix}}
\newcommand{\be}{\begin{equation}}
\newcommand{\ee}{\end{equation}}
\newcommand\ba{\begin{equation}\begin{aligned}}
\newcommand\ea{\end{aligned}\end{equation}}

\usepackage{color,soul}

\newcommand{\bap}{\begin{app}}
\newcommand{\eap}{\end{app}}
\newcommand{\begs}{\begin{exams}}
\newcommand{\eegs}{\end{exams}}
\newcommand{\beg}{\begin{example}}
\newcommand{\eeg}{\end{exaplem}}
\newcommand{\bpr}{\begin{proposition}}
\newcommand{\epr}{\end{proposition}}
\newcommand{\bt}{\begin{theorem}}
\newcommand{\et}{\end{theorem}}
\newcommand{\bc}{\begin{corollary}}
\newcommand{\ec}{\end{corollary}}
\newcommand{\bl}{\begin{lemma}}
\newcommand{\el}{\end{lemma}}
\newcommand{\bd}{\begin{definition}}
\newcommand{\ed}{\end{definition}}
\newcommand{\brs}{\begin{remarks}}
\newcommand{\ers}{\end{remarks}}

\newtheorem{theo}{Theorem}[section]

\newtheorem{exams}[theo]{Examples}

\numberwithin{equation}{section}

\newcommand{\ZZ}{{\mathbb Z}}

\newcommand{\const}{\text{\rm constant}}
\newcommand{\Id}{{\rm Id }}
\newcommand{\Range}{{\rm Range }}
\newcommand{\diag}{{\rm diag }}
\newcommand{\blockdiag}{{\rm blockdiag }}
\newcommand{\Span}{{\rm Span }}

\newtheorem{theorem}{Theorem}[section]
\newtheorem{proposition}[theorem]{Proposition}
\newtheorem{corollary}[theorem]{Corollary}
\newtheorem{lemma}[theorem]{Lemma}

\theoremstyle{remark}
\newtheorem{remark}[theorem]{Remark}
\theoremstyle{definition}
\newtheorem{definition}[theorem]{Definition}

\newtheorem{claim}[theorem]{Claim}

\newtheorem{example}[theorem]{Example}

\newtheorem{obs}[theorem]{Observation}

\newcommand{\beq}{\begin{equation}}
\newcommand{\eeq}{\end{equation}}
\newcommand{\vp}{\varphi}

\title{
Transverse bifurcation of viscous slow MHD shocks
}

\author{Blake Barker}
\address{Brigham Young University, Provo, UT 84602}
\email{blake@math.byu.edu} 
\thanks{Research of B.B. was partially supported
under NSF grants no. DMS-1700279, DMS-1400555, and NSF Postdoctoral Fellowship No. DMS-1400872.} 
\author{Rafael Monteiro}
\address{Mathematics for Advanced Materials-OIL, AIST-Tohoku University, Sendai 980-8577, Japan.}
\email{monteirodasilva-rafael@aist.go.jp} 
\thanks{Research of R.M. was partially supported
under NSF grants no. DMS-0300487 and DMS-0801745, and by an IU COAS Dissertation Year Fellowship (2014-2015).}
\author{Kevin Zumbrun}
\address{Indiana University, Bloomington, IN 47405}
\email{kzumbrun@indiana.edu} 
\thanks{Research of K.Z. was partially supported
under NSF grants no. DMS-1700279 and DMS-1400555.}
\begin{document}

\begin{abstract}
We study by a combination of analytical and numerical Evans function techniques
multi-D viscous and inviscid stability and associated transverse bifurcation of planar 
slow Lax MHD shocks in a channel with periodic boundary conditions.
Notably, this includes the first multi-D numerical Evans function study for viscous MHD.
Our results suggest that, rather than a planar shock, a nonplanar traveling wave with the same normal velocity
is the typical mode of propagation in the slow Lax mode.
Moreover, viscous and inviscid stability transitions appear to agree, answering (for this particular model and setting)
an open question of Zumbrun and Serre.
\end{abstract}
\date{\today}
\maketitle

\tableofcontents

\section{Introduction}
In this paper, continuing and extending investigations of \cite{BFZ,BHZ1,FT,Mo,TZ1,TZ2},
we study by a combination of analytical and numerical Evans function techniques
multi-D viscous and inviscid stability
and associated transverse bifurcation of planar viscous 
slow Lax magnetohydrodynamic (MHD) shocks in a channel with periodic boundary conditions.
Notably, this includes the first multi-D numerical Evans function study for 
viscous MHD, a computationally intensive problem representing the current state of the art,
and, together with the treatment of gas-dynamical shocks in \cite{HLyZ2}, the
first such study for viscous shock profiles of any physical system in multi-D.

We obtain also new detail on the inviscid stability problem, while at the same time unifying and somewhat
simplifying previously obtained results.
In particular, we give a general framework for the treatment of constraints, or involutions \cite{Dafermos-involution},
such as arise in multi-D MHD or elasticity, recovering and extending to the viscous case the fundamental results
obtained by Blokhin et al (see \cite{BT1,FT} and references therein) for inviscid MHD in a way apparently special to that case.
The latter result answers in the affirmative the fundamental open problem posed in \cite{MeZ2} whether Evans function 
stability, in the sense defined there, is necessary as well as sufficient for nonlinear multi-dimensional
stability of viscous shocks in the presence of a constraint.

Our main physical conclusions are two.
First, we make a mathematical connection between the spectral instability 
observed for slow inviscid MHD shocks \cite{BT1,FT} and ``corrugation instabilities'' observed in the astrophysics 
community \cite{IEHDH,SE}, via a viscous bifurcation analysis as in \cite{Mo,TZ1,TZ2}.
Namely, we demonstrate transitions from stability to instability satisfying the bifurcation hypotheses
proposed in \cite{Mo} (parallel magnetic field case) and \cite{TZ2} (nonparallel case), 
implying bifurcation in a mode transverse to the direction of shock propagation, i.e., lying in the direction 
parallel to the front, the first examples for which these scenarios have been shown to occur.
Our results suggest that, {\it rather than a planar shock, a nonplanar ``wrinkled'' or ``corrugated''
traveling wave with nearby normal velocity is the typical mode of propagation in the slow Lax mode.}

Second, continuing 1-D investigations of \cite{BFZ},
we show numerically that the transition to instability for viscous multi-D slow Lax MHD shocks 
coincides with the transition to instability observed in the inviscid case.  As shown in \cite{ZS,Z1,Z2}, 
for a rather general class of physical systems generalizing the ``Kawashima class'' of \cite{Kaw}, 
viscous stability implies inviscid stability: that is, viscous effects may destabilize, but never stabilize a 
planar shock wave.  
The question posed in \cite{ZS,Z1,Z2} whether and under what circumstances these two conditions coincide
is a fundamental open problem in the theory of shock waves.  On the one hand, it is much simpler to determine 
inviscid as opposed to viscous stability, 
so that pre-knowledge that they coincide would be a great help in applications; on the other, destabilization due to
viscous effects would be extremely interesting physically.
Our results here give the first information in this direction for multi-D, and (along with \cite{HLyZ2}) 
a first set of data for multi-D viscous systems.
They are obtained by the introduction of an algorithm for
numerical determination of the ``refined stability condition'' of \cite{ZS,Z2,Z3,MR2448741},
detecting concavity at transition of the associated ``critical'' spectral curve through the origin,
a new tool of general use.
Our results, as in the 1-D study \cite{BFZ}, suggest that {\it viscous and inviscid
stability transitions coincide}.

For simplicity, our investigations here are confined to the 2-D case. 
However, the methods used apply in general dimensions. In the parallel case that is our main object of study, by rotational symmetry about the longitudinal axis, or direction of propagation of the
front, multi-D stability- more generally, multi-D spectrum-  reduces to the 2-D
case, so there is no loss of generality in restricting to 2-D  (cf.
discussion of \cite{FT}). For the nonparallel case, discussed briefly here, 2-D stability is
necessary but not sufficient for multi-D stability in dimensions $d\geq 3$. Likewise, there could conceivably be an earlier bifurcation involving spectral modes associated with the neglected third dimension.  This, and the nonparallel case in general, would be interesting for further study.

\subsection{Problem and background}
The Navier--Stokes, or viscous, equations for isentropic  2-D MHD  are given, in vectorial notation, by
\begin{subequations}\label{eq1}
\begin{align}
\rho_t + \mathrm{div}(\rho u) =& 0\\
(\rho u)_t + \mathrm{div}(\rho u \otimes u-h\otimes h)+\nabla q =& \mu \Delta u + (\eta + \mu) \nabla \mathrm{div}(u) \label{eq1-b}\\
h_t - \nabla \times(u\times h)=& \nu \Delta h \label{eq1-c},
\end{align}
\end{subequations}
where $u = (u_1, u_2)$  is  the   velocity field,   $h = (h_1, h_2)$ is the magnetic field, 
and 
$q = p + \frac{|h|^2}{2}$,
where $p=p(\rho)$ is gas-dynamical pressure \cite{Ba,J,C,Kaw,Da}. 
Here,  $(x_{1},x_2)$ is spatial location and  $t$ is time.  We take $p(\rho)=a\rho^{\gamma}$ corresponding to a $\gamma$-law, or polytropic equation of state. 
The corresponding Euler, or inviscid, equations are given by Eqs. \eqref{eq1} with $\mu=\eta=\nu=0$.
In either (viscous or inviscid) case, the magnetic field must satisfy in addition the constraint
\begin{equation} \label{constraint}
 \mathrm{div}(h)\Big|_{t=0} =0,
\end{equation}
which if satisfied at initial time $t=0$, may be seen to persist for all $t>0$. 

Our aim is to study the spectral stability of both viscous and inviscid planar 
(without loss of generality standing) shock waves $u(x,t)\equiv \overline{u}(x_1)$ in dimension $2$,
either as solutions on the whole space, or - which amounts to restricting Fourier modes to a (discrete) lattice -
as solutions on a two-dimensional channel, $x_1\in \R, x_2\in [0,1]$, with periodic boundary conditions in $x_2$.
In the {\it parallel} case $u=(u_1, 0)$, $h=(h_1, 0)$, Eqs. \eqref{eq1} decouple into the equations of nonmagnetic isentropic gas dynamics
in $(\rho, u)$ and a heat equation for $h$, from which we may readily deduce that the set of parallel planar MHD shocks consists precisely of the
set of nonmagnetic gas-dynamical shocks in $(\rho,u)$, adjoined with $h_1\equiv \const$ (for details, see for instance \cite{BHZ1,MeZ2, FT}).
As existence/transversality of traveling wave profiles for viscous polytropic gas dynamics is well known both in the isentropic \cite{BaLZ2} and nonisentropic
\cite{Gi,We} case, 
one obtains thereby immediately existence/transversality of parallel MHD profiles, and, by perturbation, 
of near-parallel profiles as well; see \cite{BHZ1} for details. 

Shocks in MHD have different types defined by the number of characteristics at plus and minus infinity moving inward 
toward the shock.\footnote{A standard detail suppressed here is that the equations must first be recast in
a form that is hyperbolic and noncharacteristic with respect to shock speed; see \cite{MeZ2} or 
Sections \ref{mathematical_settings} below.}
For parallel shocks, 
this is determined by the strength $\vert h_1\vert$ of the normal magnetic field $(h_1,0)^T$, being ``fast Lax'' type for $0\leq |h_1|< H_*$, ``intermediate'' type for $H_*< |h_1|<H^*$, and ``slow Lax'' type for $H^*<|h_1|$ \cite{FT,MeZ2},
where 
\begin{equation}\label{h_lower_upper_star}
H_* = u_1^+\sqrt{\rho^+}, \qquad H^* = u_1^-\sqrt{\rho^-}
\end{equation}
(for $0 <u_1^+ < u_1^-$; see Lemma \ref{parametrization_lemma:transition_parameter}). 
Fast shocks are somewhat analogous to gas-dynamical shocks, and indeed reduce to this case in the zero-magnetic field limit $\vert h\vert \equiv 0$.
Intermediate shocks are of nonclassical ``overcompressive'' type not appearing in gas dynamics \cite{FL,Z1,Z2,BHZ1}.
Slow shocks are of classical Lax type, but separated in parameter space from the fast type and exhibiting somewhat
different properties.  The patterns of incoming
characteristics at plus and minus infinity for each type are displayed
in Table \ref{table:shock_types} for the case of parallel shocks.

Inviscid numerical studies \cite{F,T} indicate that fast parallel
shocks are typically stable, while slow parallel shocks are typically unstable.
Intermediate shocks, since overcompressive, are always inviscid unstable, and will not be discussed here.
(Nonetheless, they appear to play an important role in viscous behavior \cite{FL,ZS,Z3} where they have
been seen numerically to be at least 1-D stable \cite{BHZ1}).
Indeed, it has been shown analytically \cite{BT2,GaK}
that fast parallel MHD shocks are stable under the gas-dynamical stability condition of Majda \cite{Majda}, notably for a polytropic
equation of state; likewise, it has been shown analytically \cite{BD2,BT1,FT} that slow MHD shocks are unstable in the infinite-magnetic
field limit $|h|\to \infty$.
In particular, in the brief but suggestive paper \cite{FT} Freist\"uhler and Trakhinin, among other things,
showed analytically the inviscid instability of slow Lax shocks for parallel MHD for sufficiently large magnetic field,
extending to the parallel case (degenerate in this context \cite{FT}) the fundamental results of Blokhin et 
al \cite{BD2,BT1}.

The result of Freist\"uhler and Trakhinin \cite{FT} corroborates and puts on more solid mathematical ground an 
earlier study on instability of slow planar shocks in MHD performed by Stone and Edelman in 
the setting of astrophysics \cite{SE}, 
where it is thought to play a role for example in the dynamics of accretion disks of binary dwarf stars.
According to \cite{SE}, the loss of stability through oscillations in the slow magnetosonic shock front is known
 as \textit{corrugation instability} \cite{SE}.
Though most of Stone and Edelman's results rely on formal linear stability analysis, they also study the phenomena numerically through a time evolution code. 
The latter experiments suggest that the observed oscillatory linear instabilities
result at nonlinear level in ``fingers''
that end up destroying the planar structure of the shock front.
That is, the numerical results \cite[Section 3.1 and 3.3]{SE} of Stone and Edelman indicate, further,
that onset of instability 
is associated with loss of planar structure of the viscous profile, i.e., appearance of the above-mentioned corrugations. 

The possibility of this latter phenomenon has been verified rigorously in the form of a steady transverse
bifurcation in a general $\mathcal{O}$(2)-symmetric strictly parabolic system of conservation laws \cite{Mo} relevant to the parallel MHD case, 
under appropriate spectral bifurcation conditions,
namely, that transition to instability occurs through a pair of real eigenvalues corresponding to nonzero 
transverse Fourier modes passing through the origin.
In the non-$\mathcal{O}$(2) symmetric case, corresponding to nonparallel magnetic field, a similar Hopf bifurcation result was shown in \cite{TZ2}, 
under the assumption that loss of stability occurs through the passage of a complex conjugate pair of eigenvalues 
associated with nonzero transverse modes.  Our ultimate goal 
is to verify these spectral scenarios by a detailed numerical study of the eigenmodes of the linearized operator about the shock.

\subsection{Main results and outline of the paper} 
The first logical step in this work consists of combining the analytical conclusions of \cite{FT} 
of inviscid instability in the infinite-magnetic field limit with numerical Evans function results showing 
that slow Lax shocks can be stable for smaller magnetic fields. 
Putting these observations together, one may conclude the existence of a stability transition, 
associated with which one might hope to observe bifurcation in wave structure.  
This is far from obvious at the inviscid level, where such transitions are associated with infinitely many Fourier modes simultaneously entering the right half of the complex plane (see \cite{Benzoni}); nor is it clear a priori that there is a corresponding stability transition at
the viscous level, since viscous and inviscid spectra can differ greatly at mid and high frequencies.
Nevertheless, by  the result of Zumbrun and Serre \cite[Proposition 5.3]{ZS} connecting  viscous and inviscid spectra 
in the low frequency regime, one may conjecture the associated appearance of more standard bifurcations in the better-behaved viscous case, involving finitely many {\it low-frequency} 
modes; see Section \ref{discussion} or \cite{Z4} for further discussion.
In the simplest situation that the single (necessarily real) double eigenvalue pair (double by $\mathcal{O}$(2) symmetry) for large magnetic field moves into the stable half plane
as magnetic field is decreased, without meeting any other 
eigenvalues along the way, this would necessarily be a ``steady'' spectral bifurcation, passing through $\lambda=0$.  
This simple scenario is likewise not a priori guaranteed, but our numerical investigations confirm that it is indeed what occurs. 

To carry out these numerical investigations requires some interesting extensions of the 
standard Evans function framework \cite{GZ,Z1,Z2,Z3, ZS}
to handle the presence of constraints such as \eqref{constraint}, similar to what was done for inviscid MHD 
by Blokhin et al \cite{BT1,FT}, and (partially) for viscous MHD by M\'etivier et al \cite{MeZ2}.
In the process, we unify and simplify these previous analyses, at the same time
obtaining a new formulation of the MHD equations- the ``$\beta$-model-'' 
that is particularly convenient for numerics, combining in one model the desirable properties
of noncharacteristicity, hyperbolicity, and conservation form;  this is developed in Section \ref{mathematical_settings}.

\begin{SCfigure}
  \centering
    \caption{Log-log plot of the zero of the Lopatinsky determinant $\lambda$ against $\eps:= \frac{1}{h_1}$ when $\gamma = 5/3$ and $u_1^+ = 0.6$. Solid dots correspond to our numerical approximation of the root, open circles to the prediction given by the asymptotic expansion, and asterisks to the description given in  \cite[Equation (61)]{FT}. The value of $\lambda_2$, approximated via $\lambda_2 = \lambda^{num}/\eps^2$, is approximately 0.0836. 
}
\label{fig157}
\includegraphics[scale=0.5]{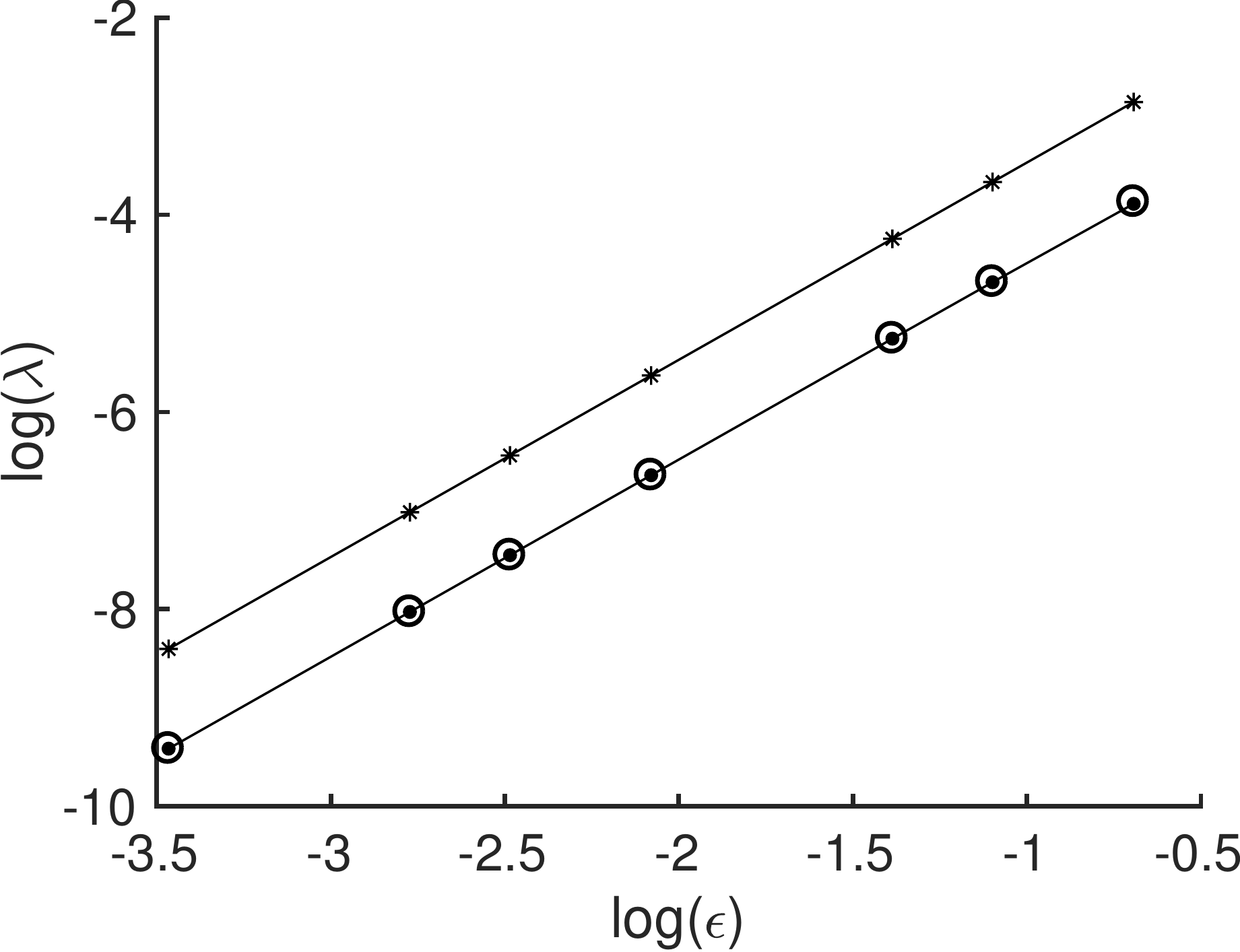}
\end{SCfigure}

To describe the main issues, a starting point is the observation that, in the presence of constraint \eqref{constraint},
the equations of MHD are not prescribed uniquely, but only up to the addition of multiples of the constraint.
Indeed, in the study of spectral stability, one could adjoin the constraint as an additional equation if desired.
Thus, one must take care to choose a form of the equations possessing properties under which the standard Evans function
and Lopatinsky determinants used to study viscous and inviscid stability are well defined, namely noncharacteristicity,
hyperbolicity, and conservative form.
There is a standard reformulation of the equations in which they become symmetric hyperbolic-parabolic and noncharacteristic
\cite{BT1,C,MeZ2} and another, different, formulation in which they become conservative; the standard approach has been to use
ad hoc combinations of these in the analysis, depending on the need at hand.
Here, we introduce for our single formulation a different analytical framework encompassing both viscous and inviscid cases.
This gives   necessity and sufficiency  of the Evans condition for viscous MHD
stability in the presence of constraint \eqref{constraint}, answering an open problem posed in \cite{MeZ2}, where sufficiency but not necessity was established. In passing, we rederive and further illuminate the inviscid results of \cite{BT1,FT}.
These issues are discussed in Section \ref{mathematical_settings}, where the $\beta$-model and basic
analytic framework are introduced.

\begin{figure}
        \centering
        \begin{subfigure}[b]{0.45\textwidth}
		 (a) \includegraphics[scale=0.4]{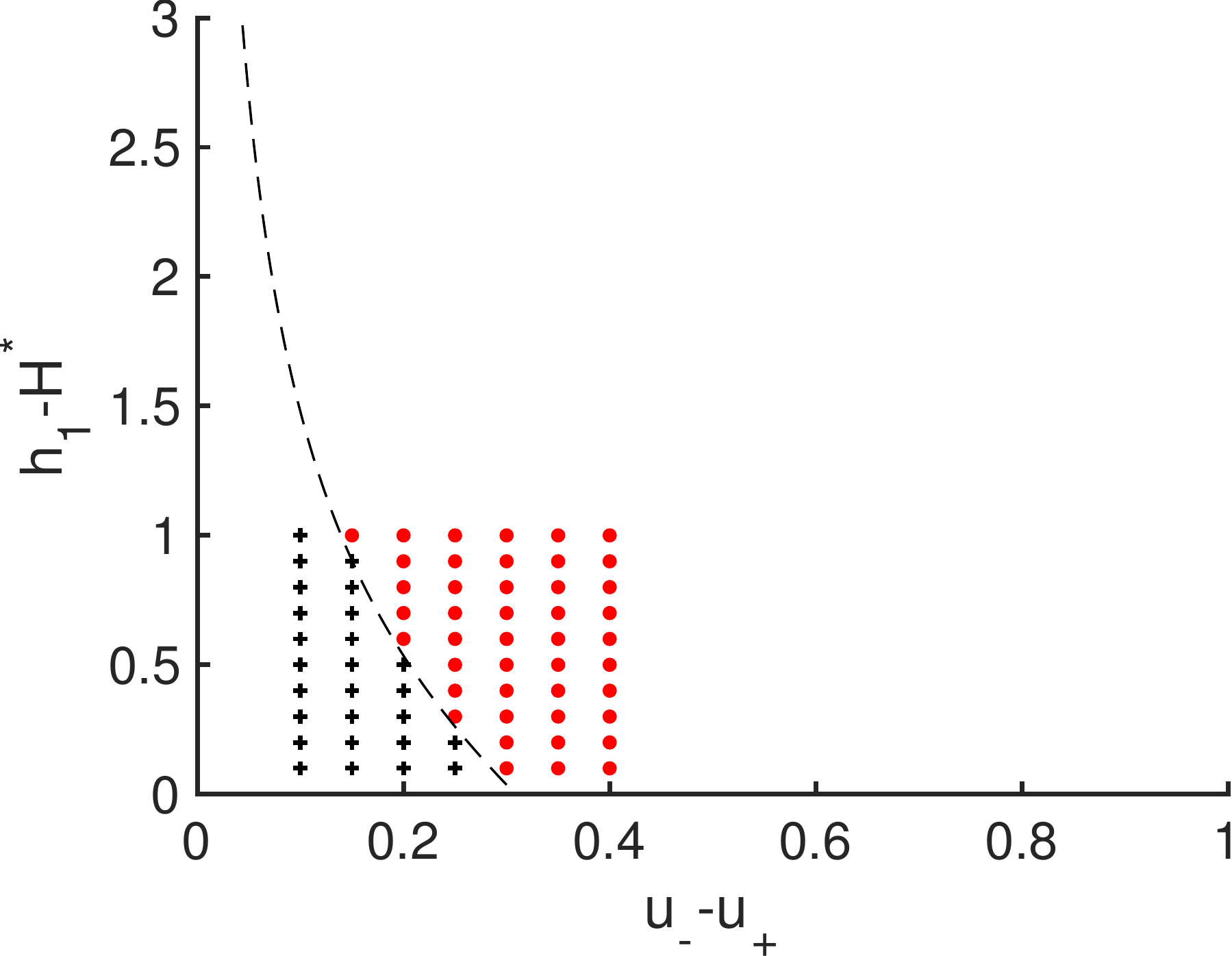}
		 \end{subfigure} 
		 \quad
		 \begin{subfigure}[b]{0.45\textwidth}
		(b) \includegraphics[scale=0.4]{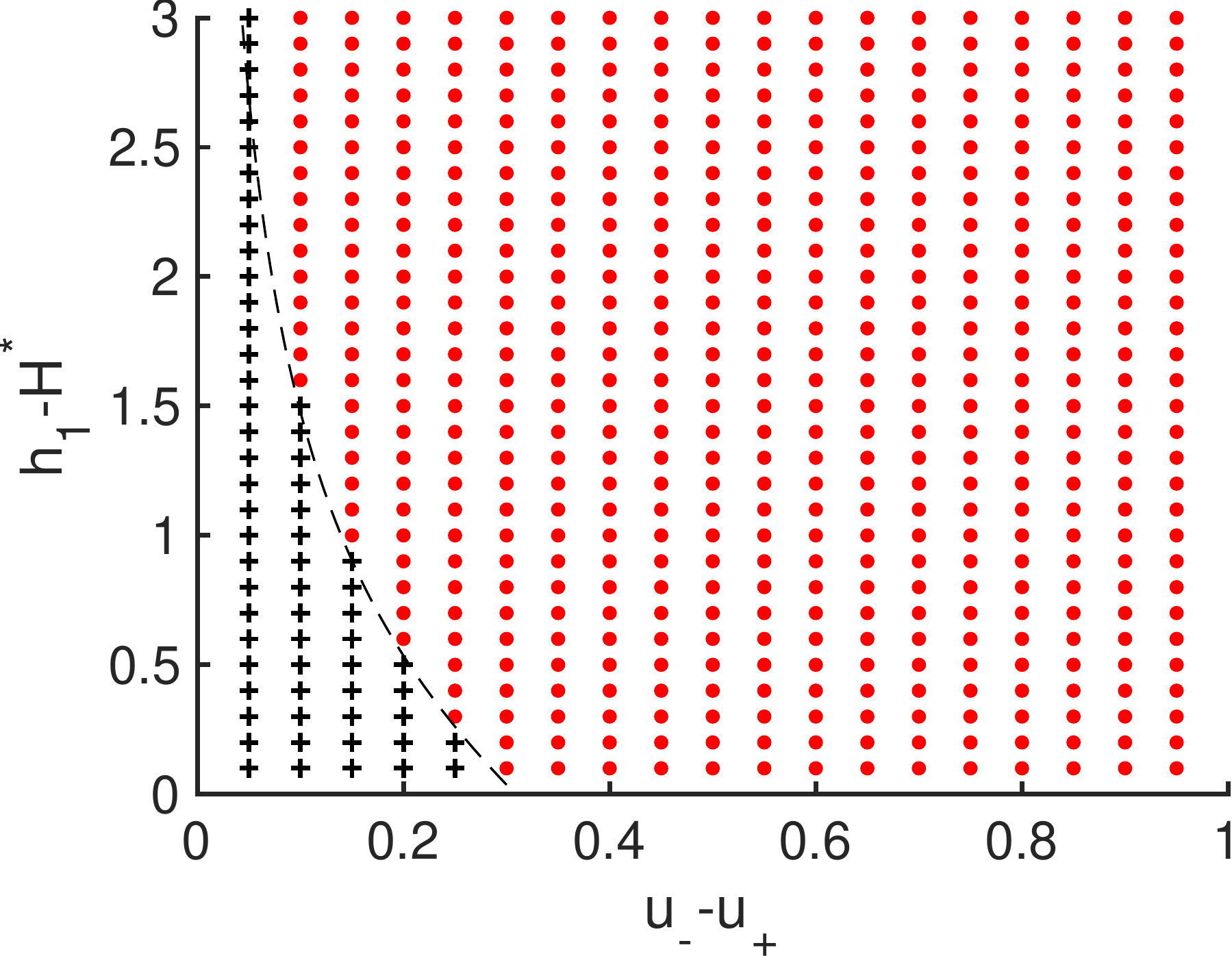}  
		 \end{subfigure}
		 	\caption{
			Bifurcation diagram plotting $h_1-H^*$ against $u_1^--u_1^+$ indicating the boundary between stable and unstable waves in the (a) viscous model and (b) the inviscid model when $\gamma = 5/3$. A red dot corresponds to instability while a black plus sign indicates stability. 
In the viscous case, to examine stability, we considered $\xi \in [0.001, 0.004, 0.007, 0.01, 0.04, 0.07, 0.1, 0.14, 0.17, 0.2]$. The dashed line in both figures indicates the critical transition parameter for the Lopatinsky determinant.		}
	\label{fig264}
\end{figure}

In Section \ref{Lopatinski_analysis}, we provide a description of the Rankine-Hugoniot conditions and the Lopatinsky determinant condition,
giving the foundations for  a careful study of inviscid instabilities. 
In particular, we (i) recapitulate in the more convenient $\beta$-model framework the large-magnetic field asymptotics of \cite{FT} showing instability, at the same time correcting certain computation errors 
in \cite{FT}; 
and (ii) carry out a numerical Lopatinsky study both verifying our 
asymptotics and extending the analysis to the small-magnetic field regime (see Figs. \ref{fig158} and \ref{fig264}(b), respectively).  
\begin{SCfigure}
  \centering
  \caption{Log-log plot of the relative error between the numerical and asymptotic descriptions of the zero of the Lopatinsky determinant against $\eps = \frac{1}{h_1}$ when $\gamma = 5/3$ and $u_1^+ = 0.6$.  }
\label{fig158}
\includegraphics[scale=0.5]{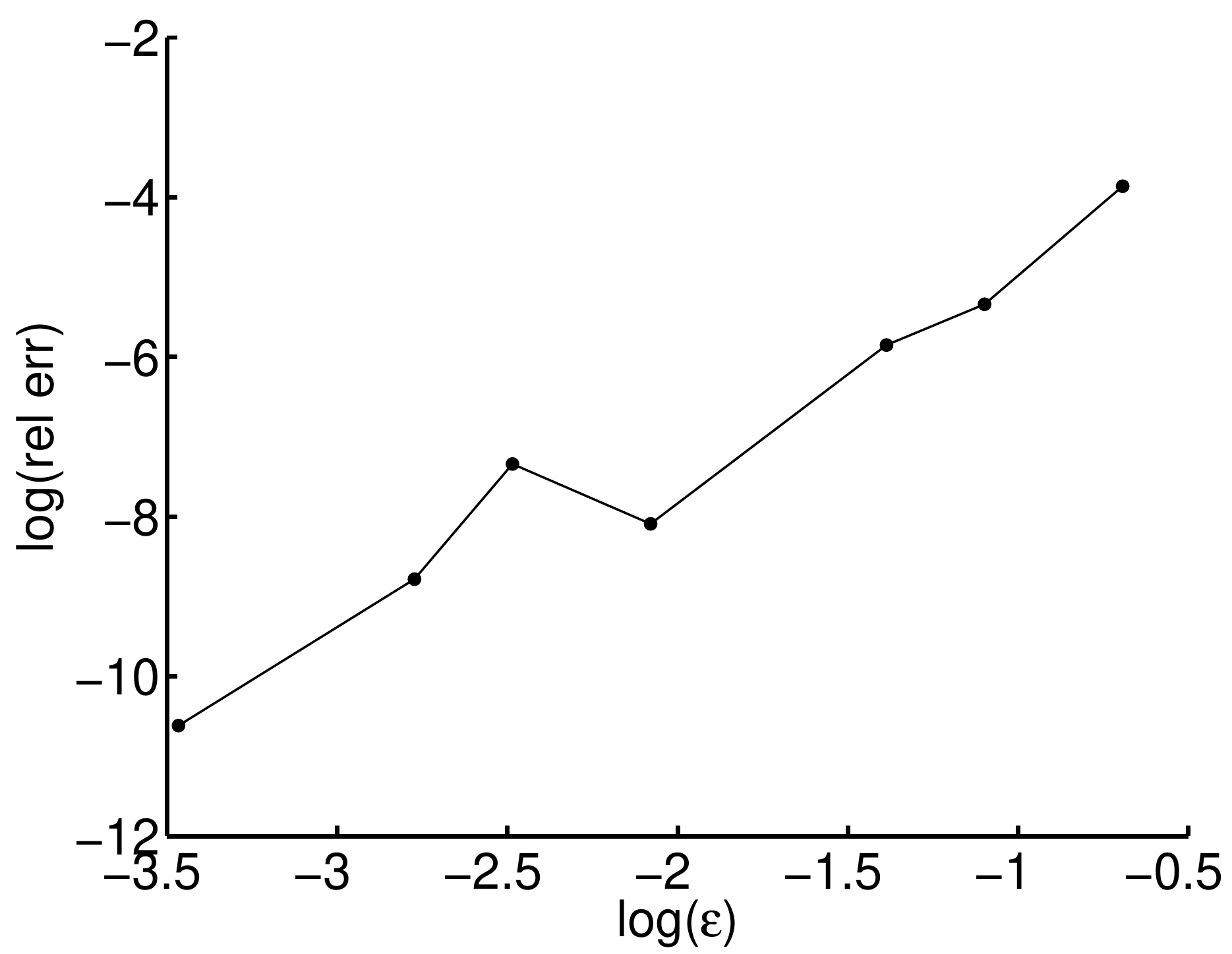}
\end{SCfigure}
Two notable conclusions are that:
 \begin{enumerate}[label=(\roman*), ref=\thetheorem(\roman*)]
\hangindent\leftmargin
	\item[(i)] {\it the large-magnetic field asymptotics
		are quite accurate}, extending even down to rather small magnetic field strengths (see Fig. \ref{fig157});
 \item[(ii)] {\it there do exist inviscid stable slow parallel
MHD shocks} for sufficiently small magnetic field and sufficiently small amplitude of the shock.
\end{enumerate}
The latter conclusion has the important implication that a stability transition, with potential for bifurcation, occurs. Finally, in Section \ref{comparison},
we compare our analytical and numerical multi-D results with those in \cite{FT} and \cite{SE}. 
Notably, we find that our large-magnetic field asymptotics improve by as much as $20\%$ 
on the accuracy of previous analyses. Specifically, as shown in 
Fig. \ref{fig157}, when compared
with results in \cite[ Equation (61)]{FT}, the calculation
\eqref{compare_with_FT} of  $\lambda_2$ shows a better agreement with numerical
results; see also Section \ref{Lopatinski_analysis}.
This discrepancy is partly explained by the computations of Appendix \ref{Rankine},
where, among other things, we show that the dynamic Rankine-Hugoniot
conditions of \cite[Equation (44)]{FT}  are incorrectly stated.
This error, acknowledged by the authors in \cite{Trak_pers}, leads to
quantitative but not qualitative changes in the asymptotic results
obtained in \cite{FT}. More important, the corrected version yields a result that is valid for
all $\gamma\geq   1$ for isentropic-law,  extending the result
of  \cite[\S 3.4]{FT} constrained to $\gamma \in [1,2]$.

\begin{SCfigure}
  \centering
  \caption{
  Plot of $\mathrm{Re}(u_1(x))$ in the first order approximation of  the perturbation of the nonplanar bifurcating wave of \eqref{eq1} as constructed by solving for the eigenvalue-eigenfunction pair using the Evans function and then imposing periodicity. The associated parameters are $\gamma = 5/3$, $u_1^+ = 0.6$, and $h_1 = 3$.
}
\label{fig104}
\includegraphics[scale=0.5]{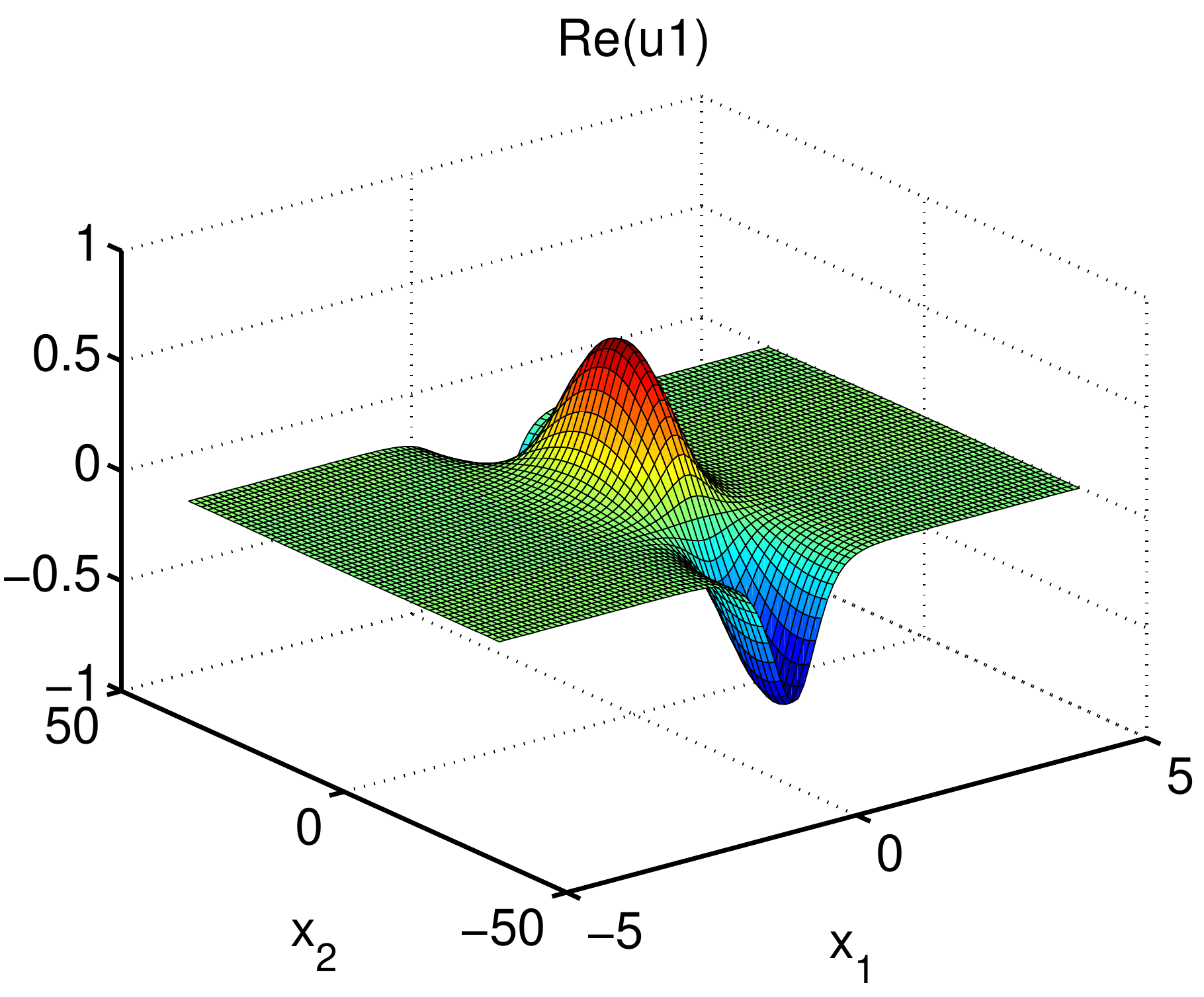}
  \end{SCfigure}

In Section \ref{viscous_evans_function}, we connect the inviscid analysis with the viscous theory of stability of 
planar shocks based on the work of \cite{ZS}, through a study of the low-frequency limit 
(Table \ref{viscous_table} and Fig. \ref{fig101});
this gives an additional check on our Lopatinsky computations through their asymptotic agreement with our numerical computations of the Evans function,
an object defined in a completely different way.

 \begin{SCfigure}
  \centering
  \caption{
  Plot of $\lambda(\xi)$ against $\xi$, where $D(\lambda(\xi),\xi) = 0$, when $\gamma = 5/3$, $u_1^+ =0.6$ and $h_1 = 3$.
  }
\includegraphics[scale=0.4]{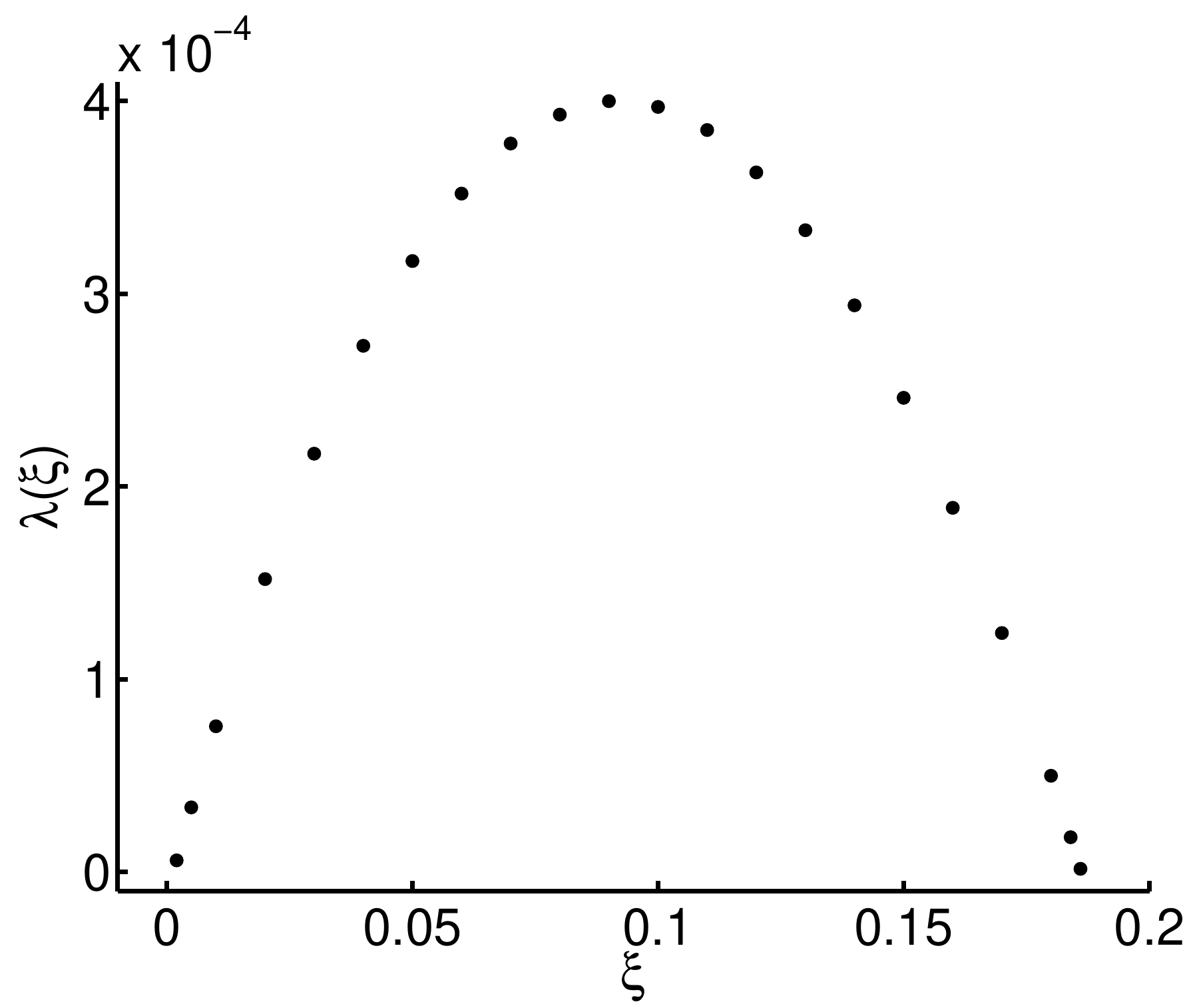}
\label{fig101}
\end{SCfigure}

We go on to carry out a complete numerical Evans function study of viscous stability over all parameters and frequencies (see Fig. \ref{fig264}(a)),
verifying that there occurs the same stability transition, at the same parameter values, that were seen in the inviscid case.
We present, further, numerical results verifying the bifurcation conditions assumed in the abstract results in \cite{TZ2} and \cite{Mo} (see also the related \cite{JYZ}), consisting of the absence of other neutrally stable eigenvalues 
(see Fig. \ref{fig258}) and nonzero speed of crossing of the imaginary axis as the magnetic field is varied (see Fig. \ref{figs237and238}).
We compute also approximate zero-eigenfunctions at the bifurcation point,
yielding the approximate shape of the bifurcating nonplanar wave (see Fig. \ref{fig104}).

We finalize the paper with an appendix: in \ref{winding_number_appendix} we briefly explain how numerical winding number computations were carried out; 
in \ref{appendix_FT_proof}, following \cite{BT1,FT}, 
we present an alternative proof by direct computation of Proposition 
\ref{persistence}
in the special case of inviscid MHD, 
showing that the constraint $\mathrm{div}(h)\Big|_{t=0} =0$ persists throughout the dynamics, namely, $\mathrm{div}\left(h(t)\right) \equiv 0$, for all $t\geq0$;  last,  in \ref{Rankine} we give another perspective on 
the \textit{dynamic Rankine-Hugoniot condition} of \cite{BT1,FT}, at the same time correcting details 
of some related calculations in 
\cite{FT}.
  \begin{SCfigure}
  \centering
  \caption{
  Viscous stability bifurcation diagram when $\gamma = 7/5$, with red open circles corresponding to instability and black plus signs indicating stability of the Fourier transformed operator at frequency $\xi$.
    }
	\label{fig258}
		\includegraphics[scale=0.5]{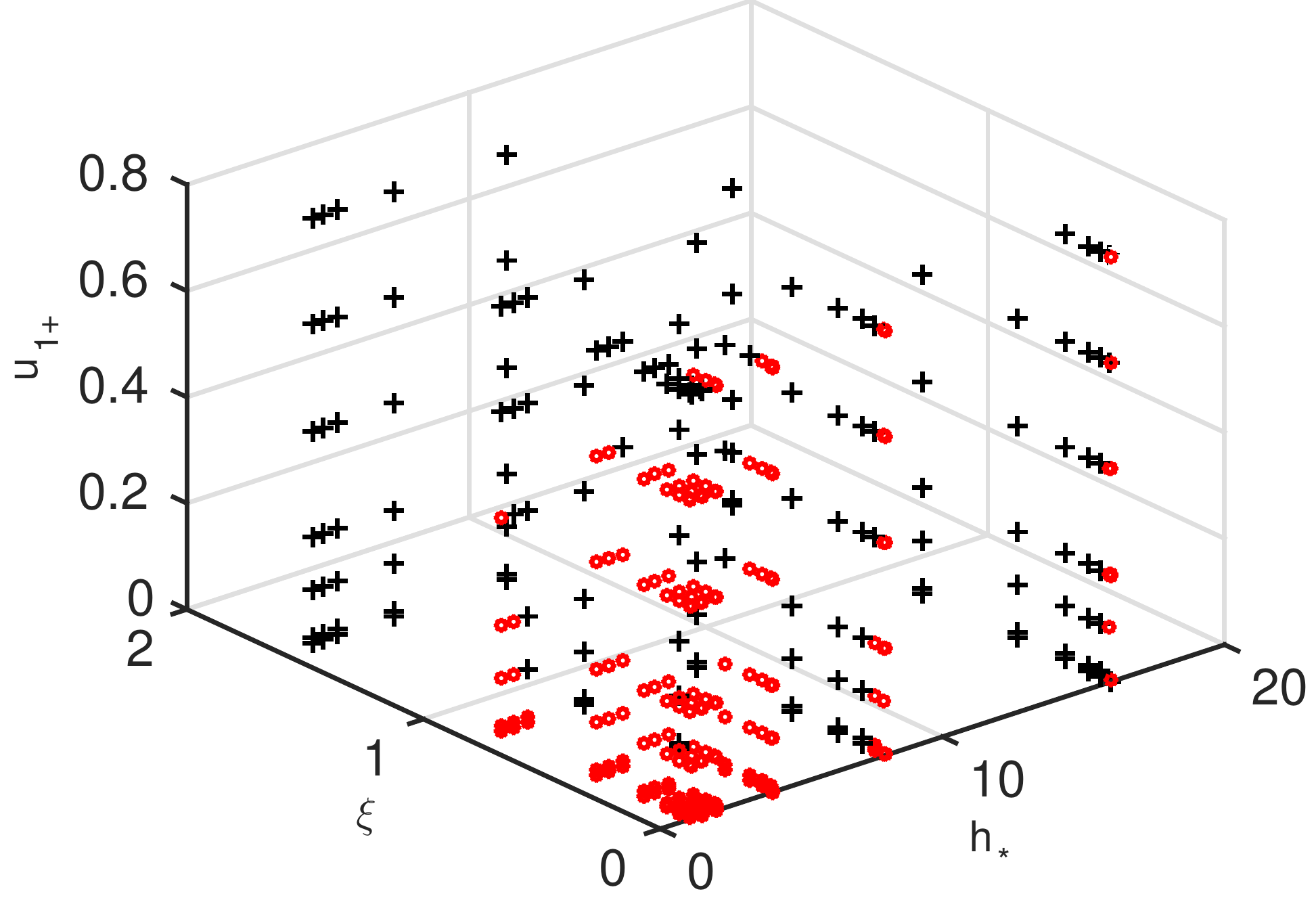}
  \end{SCfigure}

\subsection{Discussion and open problems}\label{discussion}
Our general results on Evans functions for systems with constraints pave the way for a unified treatment of multi-dimensional
viscous shock stability in MHD, elasticity, and related equations arising in continuum mechanics.
Our introduction of the $\beta$-model, though more special to MHD, by putting the equations into a standard symmetrizable conservative form, has the tremendous advantage that it allows computations using existing ``off-the-shelf'' code in the numerical stability package STABLAB \cite{STABLAB}.
As discussed in \cite{HLyZ2,BHLyZ1,BHLyZ2},  success or failure of multi-dimensional computations is highly dependent on the specific algorithm used, with a number of catastrophic possible pitfalls that must be avoided.
Thus, the ability to use existing, already-tested algorithms is of significant practical advantage.

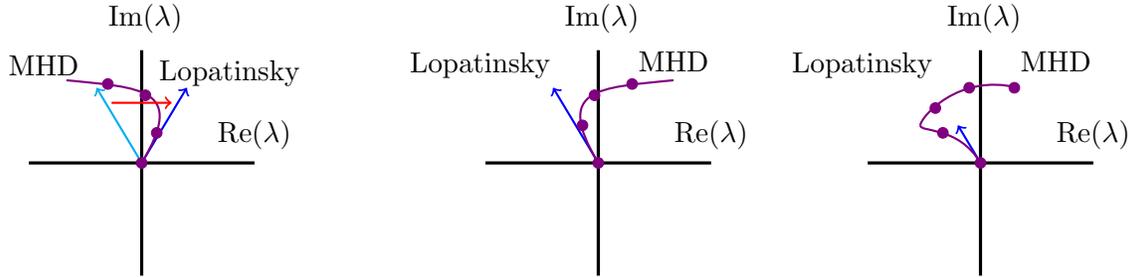
\begin{figure}[h]
        \centering
        \begin{subfigure}[b]{0.3\textwidth}

\begin{tikzpicture}
    \draw(-1.5,0)--(1.5,0) [very thick];
    \draw(0,-1.5)--(0,1.5) [very thick];
    \draw[->] (0,0)--(0.6,1) [blue,thick];
    \draw[->] (0,0)--(-0.6,1) [cyan,thick];
    \draw (0,0)..controls (0.6,1) and (0,1) ..(-1,1.1) [violet, thick];
    \draw[->] (-0.4,0.8)--(0.4,0.8) [red, thick];
    \filldraw [violet] (0,0) circle [radius = 2pt];
    \filldraw [violet] (0.2, 0.4) circle [radius = 2pt];
    \filldraw [violet] (0.05, 0.9) circle [radius = 2pt];
    \filldraw [violet] (-0.45, 1.05) circle [radius = 2pt];
    \draw(0.1,1.2) node[anchor = west] {Lopatinsky};
    \draw(-0.7,1.3) node[anchor = east] {MHD};
    \draw(1.5,0.05) node[anchor = south] {$\mathrm{Re}(\lambda)$};
    \draw(0.05,1.6) node[anchor = south] {$\mathrm{Im}(\lambda)$};
    \end{tikzpicture}
            \end{subfigure}%
        \quad
        \begin{subfigure}[b]{0.3\textwidth}
    \begin{tikzpicture}
    \draw(-1.5,0)--(1.5,0) [very thick];
    \draw(0,-1.5)--(0,1.5) [very thick];
    \draw[->] (0,0)--(-0.6,1) [blue,thick];
    \draw (0,0)..controls (-0.6,1) and (0,1) ..(1,1.1) [violet, thick];
    \filldraw [violet] (0,0) circle [radius = 2pt];
    \filldraw [violet] (-0.21, 0.5) circle [radius = 2pt];
    \filldraw [violet] (-0.05, 0.9) circle [radius = 2pt];
    \filldraw [violet] (0.45, 1.05) circle [radius = 2pt];
    \draw(0.4,1.35) node[anchor = west] {MHD};
    \draw(-0.5,1.3) node[anchor = east] {Lopatinsky};
        \draw(1.5,0.05) node[anchor = south] {$\mathrm{Re}(\lambda)$};
    \draw(0.05,1.6) node[anchor = south] {$\mathrm{Im}(\lambda)$};
    \end{tikzpicture}
        \end{subfigure}
        ~ 
        \begin{subfigure}[b]{0.3\textwidth}
    \begin{tikzpicture}
    \draw(-1.5,0)--(1.5,0) [very thick];
    \draw(0,-1.5)--(0,1.5) [very thick];
    \draw[->] (0,0)--(-0.3,0.5) [blue,thick];
    \draw (0,0)..controls (-0.3,0.5) and (-0.8,0.4).. (-0.8,0.5) [violet, thick];
    \draw (-0.8,0.5)..controls (-0.8,0.6) and (-0.3,1.2).. (0.5,1) [violet, thick];
    \filldraw [violet] (0,0) circle [radius = 2pt];
    \filldraw [violet] (-0.6, 0.73) circle [radius = 2pt];
    \filldraw [violet] (-0.15, 1) circle [radius = 2pt];
    \filldraw [violet] (0.45, 1) circle [radius = 2pt];
    \filldraw [violet] (-0.5,0.4) circle [radius = 2pt];
    \draw(0.4,1.35) node[anchor = west] {MHD};
    \draw(-0.5,1.3) node[anchor = east] {Lopatinsky};
    \draw(1.5,0.05) node[anchor = south] {$\mathrm{Re}(\lambda)$};
    \draw(0.05,1.6) node[anchor = south] {$\mathrm{Im}(\lambda)$};
    \end{tikzpicture}
        \end{subfigure}
        \caption{Diagram indicating the ways in which a subcritical and critical Hopf bifurcation might manifest itself in the bifurcation diagrams when taking into account the relationship between the Evans function and the Lopatinsky determinant.
        }\label{transitions}
\end{figure}

We point out some further background and implications from a more general perspective
contrasting viscous and inviscid stability.
As shown in \cite{ZS}, viscous stability is closely related at low frequencies to inviscid stability, hence uniform viscous stability implies uniform inviscid stability.
This means that, as shock or magnetic field strength is increased from a stable regime, a transition to inviscid instability implies a corresponding transition to viscous instability, occurring in low-frequency modes.
The reverse is not true, as it is possible that a transition to viscous instability could occur in advance of the transition to inviscid instability due to destabilization of an intermediate- or high-frequency mode unrelated to the inviscid problem. Indeed, let $\xi\in \R$ be the Fourier frequency parameter in the direction $x_2$ transversal to the shock front.
Then, as depicted in Fig. \ref{transitions}, there are essentially 3 different scenarios for viscous vs. inviscid stability transitions in a finite-cross section channel, 
depending mainly on concavity vs. convexity of the neutral spectral curve $\lambda = \lambda(\xi)$ for the viscous case,
bifurcating from $\lambda(0)=0$, given by the sign of $\lambda''(\cdot)$ near $\xi=0$.

Recall \cite{ZS}, that this curve is tangent at $(\xi, \lambda)=(0,0)$ to the corresponding inviscid stability curve, given by homogeneity by
a ray through the origin.
The viscous spectral curve $\lambda(\cdot)$ is depicted in various cases in Fig. \ref{transitions} along with its tangent inviscid ray.
In the first case, $\lambda(\cdot)$ is concave ($\lambda''(\cdot)<0$) and we see that the transition to instability occurs simultaneously in the whole space 
(for which $\xi\in \R$) for viscous and inviscid problems, and slightly later for a duct of finite cross-section (for which $\xi \in \ZZ$, hence
viscous spectra lags behind inviscid by a fixed finite amount).  In the second, $\lambda(\cdot)$ is convex ($\lambda''(\cdot)>0$) and viscous instability occurs slightly before inviscid instability as the bifurcation parameter is varied, for either the whole space or finite cross-section.
In the third, $\lambda(\cdot)$ is concave, but high-frequency instabilities cause the viscous problem to destabilize first.
This is consistent with the results of \cite{ZS} in the whole space, where it is shown that the viscous transition occurs not later than the inviscid one.  For the whole space problem ($\xi\in \R$), our discussion above shows that
simultaneously precisely in case 1, and strictly sooner in cases 2 and 3.

Our numerical results (see Section \ref{s:finding}) for typical parameters $\gamma = 5/3$, $u_1^+ = 0.86$
indicate that the viscous transition occurs at $H^*\approx 2$, while the
inviscid transition occurs at $H^* \approx 1.995$. 
Likewise, one can see from Fig. \ref{fig101} that the second derivative of the spectral parameter $\lambda = \lambda(\xi)$ with respect to $\xi$ is negative. 
Indeed, we find this to be the case for all parameter values; see Section \ref{s:concavity} and particularly
Fig. \ref{fig273}(a).
That is, we appear to be in the first case depicted in Fig. \ref{transitions}.
This has the important consequence that, considered in the whole space, {\it viscous and inviscid stability transitions coincide}
for the 2-D transverse instabilities considered here, similarly as was seen in \cite{BFZ} for the 1-D longitudinal instabilities considered there.
Thus, though viscosity can in principle according to the results of \cite{ZS} hasten the onset of instability, for the 
two studies carried out so far (in \cite{BFZ} and here) for gas dynamics and MHD, this possibility 
was not in practice observed.
It remains a very interesting open problem whether such ``viscosity-enhanced instability'' can occur for physically relevant models of gas dynamics or MHD, both philosophically, and- since inviscid transitions may often be computed 
explicitly, whereas viscous transitions require substantial numerical computation- from a practical point of view.
Indeed, we note that it has been shown in \cite{FrKSch} that the neutral inviscid stability curve, where the Lopatinsky condition precisely vanishes 
for $\lambda=0$, $\xi\neq 0$, {\it may be determined explicitly,} making this a practical condition indeed.

\begin{figure}
        \centering
        \begin{subfigure}[b]{0.45\textwidth}
		 (a) \includegraphics[scale=0.4]{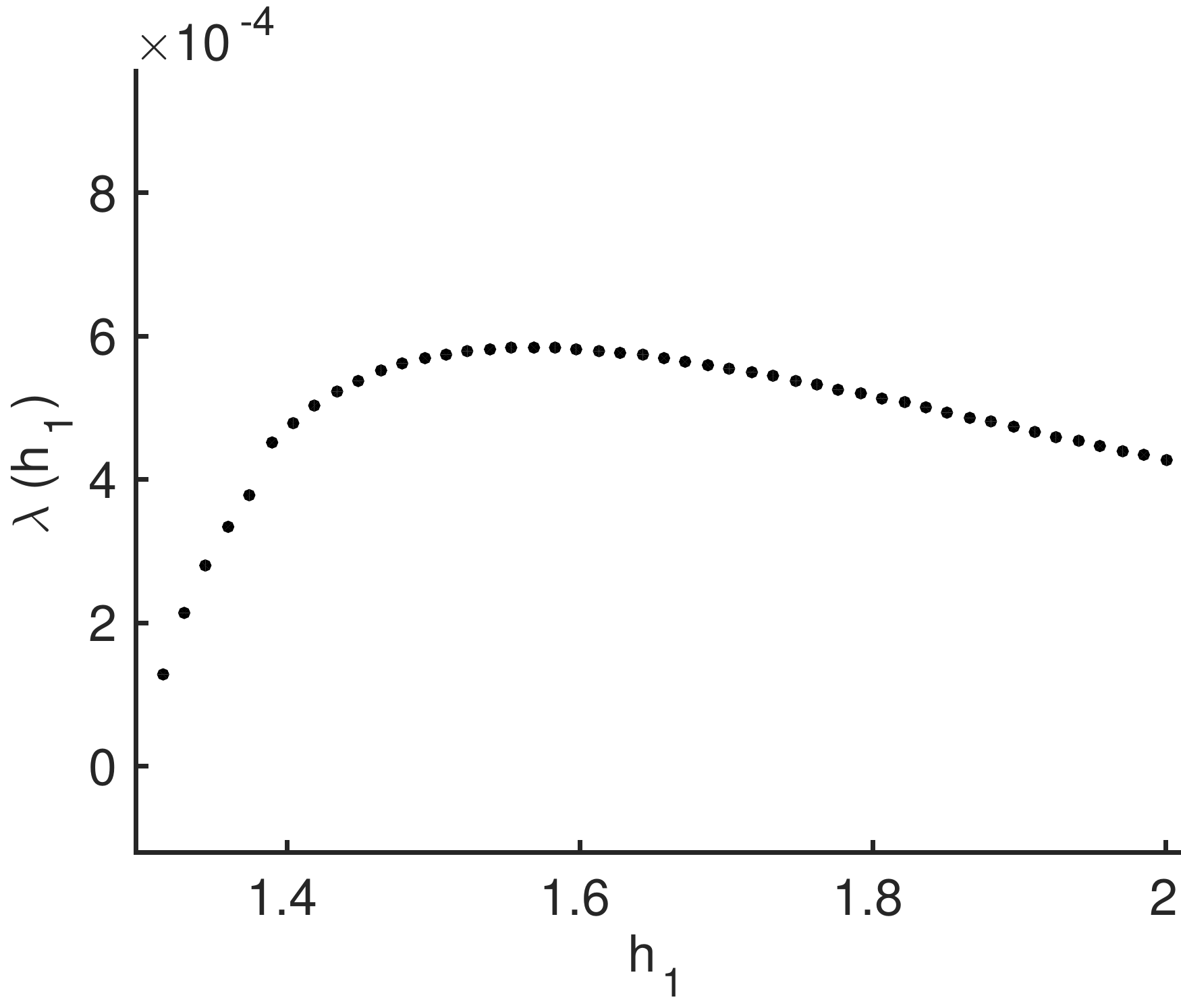}
		 \end{subfigure} 
		 \quad
		 \begin{subfigure}[b]{0.45\textwidth}
		(b) \includegraphics[scale=0.4]{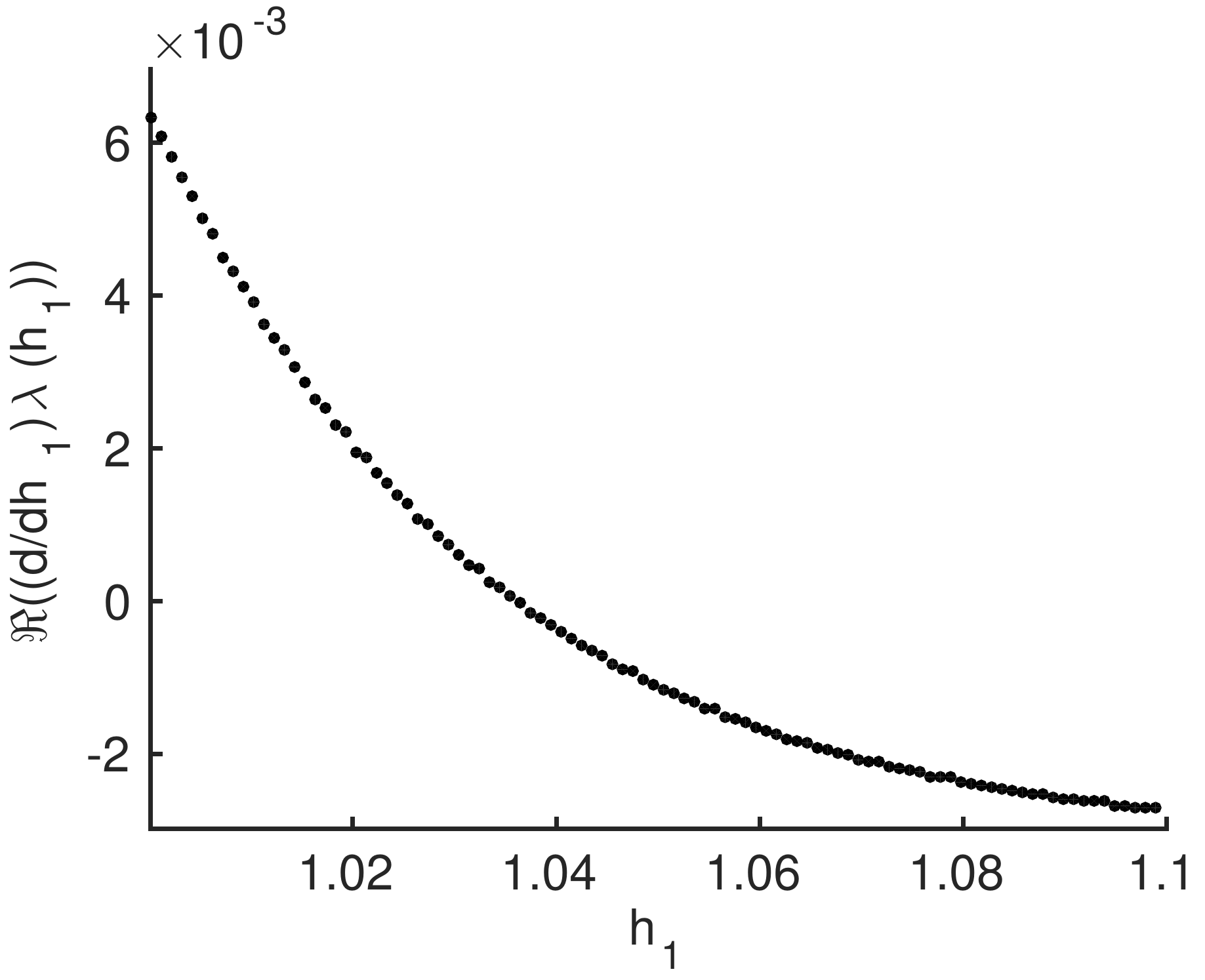}  
		 \end{subfigure}
		 	\caption{Plot of (a) $\lambda(h_1)$ against $h_1$ and (b) $(d/dh_1)\lambda(h_1)$ against $h_1$ when $\gamma = 5/3$ and $u_1^+ = 0.6$. For computational details, see Section \ref{additional_numerics}.
			}
			\label{figs237and238}
		\end{figure}

Our numerical study confirms not only stability transition but also that the spectral bifurcation assumptions assumed in \cite{Mo, TZ1,TZ2} indeed hold:  the loss of stability happens through a double eigenvalue crossing the imaginary axis.
For fully parabolic ``artificial viscosity'' versions of MHD, the actual nonlinear bifurcation would follow from the results presented in \cite{Mo,TZ1,TZ2} restricted to the space of divergence free magnetic field functions, provided that the corresponding
spectral scenario were verified.
Extending this result to the physical, ``real'' viscosity case considered here is an important open problem.
Presumably, one could expect a similar spectral bifurcation for the artificial viscosity case (giving the full nonlinear result), but 
we do not investigate this here.
The extension of our investigations to a complete ``all-parameters'' study of MHD shocks analogous to that done for gas dynamics 
in \cite{HLyZ2} is another important direction for further study:
likewise, spectral stability of small amplitude nonextreme shocks as considered here, 
both inviscid and viscous.\footnote{
For extreme shocks, i.e., $1$- or $5$-shocks in the artificial viscosity case, see \cite{FS1,FS2}.}

It is worthwhile to emphasize the loss of planar structure observed here (see Section \ref{viscous_evans_function}), in the numerical experiments of Edelman and Stone (cf. \cite[Section 3.1 and 3.3]{SE}) and in the context of steady bifurcations in a $\mathcal{O}$(2)-symmetric strictly parabolic system of conservation laws; see \cite{Mo}.
This indicates that the resolution of the 2-D-Riemann problem for slow shocks in MHD is not realized through planar shocks, but
generically involve nonplanar ``corrugated'' fronts as component slow waves; see also the recent numerical results of \cite{Scotternst2011nonlinear}.
Spectral and nonlinear stability of these nonplanar waves is another very interesting open problem for further investigation.

Finally, we mention an interesting related analysis of Freist\"uhler, Kleber, and Schropp
\cite{FrKSch} for the inviscid isothermal ($\gamma=1$) case, in which they find by explicit computation the inviscid stability
boundary $\Delta (0, \pm 1)=0$ (in the notation of Section \ref{sec:lop_det_constr}) for slow parallel shocks; a similar computation should be possible for general $\gamma$,
sharpening our description of the inviscid boundary in Fig. \ref{fig264}-(b).\footnote{Namely, the dashed line in the figure,  here computed  numerically as described in Section \ref{sec:destabilization}. }
Freist\"uhler et al investigate numerically also the fast shock case, showing that parallel isothermal fast shocks are uniformly stable, but nonparallel ones experience a transition to instability across a particular parameter surface.
Similarly, for the general isentropic case ($\gamma>1$), Trakhinin \cite{T} has shown that
fast nonparallel shocks can be unstable in some regimes.
A very interesting further study would be to carry out the corresponding analysis of viscous stability transition in these cases for fast shocks as we have done here for slow shocks.

\subsection{Notation}\label{Notation:section} In this paper  we denote the real part (resp., imaginary part)  of a number $z\in \mathbb{C}$ by $\mathrm{Re}(z)$ (resp., $\mathrm{Im}(z)$). We let $u \in \R^n$ be a vector of states assuming values $u^{\pm}$ accross a shock. Given any function $u \mapsto f(u)$ we write $[f(u)] = f(u^+) - f(u^-)$. The partial derivative of a differentiable mapping $(x_1,x_2, \ldots, x_n)\mapsto g(x_1,x_2, \ldots, x_n)$ with respect to the variable $x_i$ is written as $g_{x_i}.$
Given a matrix $A \in \mathbb{C}^{n\times n}$, we write $A^{-1}$ to denote its inverse, and $\sigma\left(A\right)$ to represent its eigenvalues. 

\subsection*{Acknowledgments} R.M. would like to thank
A. Mailybaev, D. Marchesin and C. Rohde for interesting conversations.
K.Z. thanks David Lannes, Jinghua Yao, and Alin Pogan for interesting conversations regarding constraints.
Numerical computations in this paper were made with the use of STABLAB, a MATLAB-based stability platform developed by B.B and K.Z. together with Jeffrey Humpherys and Joshua Lytle; SYMPY; and MATHEMATICA.

\section{The $\beta$-model} \label{mathematical_settings}
Focusing our attention without loss of generality (by Galillean invariance of \eqref{eq1})
on shocks propagating in the $x_1$ direction,
we now introduce the $\beta$-model, obtained by adding  a constant
multiple of $\mathrm{div}(h)$ to the  equation \eqref{eq1-c} in \eqref{eq1}, yielding in place of equations \eqref{eq1} the system
\begin{subequations}\label{eq-2}
\begin{align}
\rho_t + \mathrm{div}(\rho u) &=0,\\
(\rho u)_t + \mathrm{div}(\rho u \otimes u - h\otimes h)+\nabla q &= 
	\mu \Delta u + (\eta + \mu) \nabla \mathrm{div} (u), \label{eq2-b}\\
h_t - \nabla \times (u\times h) +\beta \mathrm{div}(h)e_1&= 
	\nu \Delta h ,\label{eq2-c} 
\end{align}
\end{subequations}
where $\beta$ is a fixed real-valued parameter and $e_1 = (1,0)^{t}$. 
Recalling the constraint $\mathrm{div}(h) =0$ we see that \eqref{eq1} and \eqref{eq-2} are equivalent
for smooth solutions. However, \eqref{eq-2} has certain practical advantages, as we now explain. 

An important reason to modify \eqref{eq1} is that, considered without the constraint $\mathrm{div}(h) =0$, 
the inviscid version $\mu= \eta=\nu=0$, of \eqref{eq1} is not hyperbolic. 
A standard resolution of this problem is to substitute in \eqref{eq1} relations
$ \nabla(|h|^2/2) - \mathrm{div}( h\otimes h) = h\times (\nabla \times  h) - h \mathrm{div}( h) = 
h\times(\nabla \times h) $
and
\begin{align*}
 \nabla \times(h\times u) &= \mathrm{div}(u) h + (u \cdot \nabla)h  -( \mathrm{div}(h))u - (h\cdot \nabla)u
 = \mathrm{div}(u) h + (u \cdot \nabla)h  - (h\cdot \nabla)u,
\end{align*}
which amounts to adding the nonconstant multiples 
$h \mathrm{div}( h)$ and $u \, \mathrm{div}(h)$ of the constraint $\mathrm{div}(h)$ 
to \eqref{eq1-b} and \eqref{eq1-c},
to obtain the symmetrizable hyperbolic system \cite{FT,BT1,MeZ2}
\be \label{MeZ2sys}
\rho_t + \mathrm{div}(\rho u) = 0,\quad
	(\rho u)_t + \mathrm{div}(\rho u \otimes u) + h\times (\nabla \times  h) +\nabla p = 0 ,\quad
h_t - \nabla \times(u\times h)= 0.
\ee
This has the advantage of both recovering  hyperbolicity, and  providing a  symmetrizable 
system of equations.
On the other hand, this handy device results in loss of conservative form of the equations, 
as a result of which jump conditions across shocks are not computable for this version of the equations. 
Meanwhile, the jump conditions for the conservative version \eqref{eq1}
are degenerate (not full rank), and must be supplemented with that of the constraint in order to 
obtain the correct number of boundary conditions for the shock problem.
See \cite[\S 7]{MeZ2} for further discussion.
We refer elsewhere to this model, used in \cite{BT2,MeZ2,FT}, as the ``hybrid'' or ``standard'' model.

It is clear that the inviscid
$\beta$-model preserves the conservative structure of the equations; however, it is not symmetrizable. 
Nonetheless, as we will see shortly, it maintains (weak) hyperbolicity, among
other useful properties, in particular those needed for linearized stability analysis by the study
of Majda's Lopatinsky determinant \cite{Majda,ZS,MeZ2}.
Moreover, it may be used ``as is'' for both interior equations and jump conditions, without additional 
modifications, allowing the use of standard numerical schemes for investigation of Lopatinsky stability.
Likewise, the viscous version of the $\beta$-model is conservative and, 
though not of classical symmetrizable ``Kawashima''
form \cite{Kaw}, exhibits the same favorable type of dispersion relation enjoyed by models
of that type, among other properties needed for linearized stability analysis by the study of
the Evans function \cite{AGJ,GZ,Z1,Z2,Z3,ZS,MeZ2}.

\subsection{The inviscid $\beta$-model and its properties}\label{s:gensys}
To explore the properties of the $\beta$-model we find it convenient to work in an abstract setting, following the discussion in \cite[\S 5.4]{Da}, afterwards specializing to our model.
We start with the more complicated inviscid case, 
then finish by indicating briefly
the treatment of the viscous case following \cite[pp. 2 and 62-64]{JYZ}.

System \eqref{eq-2}, in the inviscid case $\mu=\eta=\nu$, may be written in the general form 
 \begin{align}\label{eq-2:abstract_form}
f_0(\mathcal{W})_t + \sum_{j=1}^{d} f_j(\mathcal{W})_{x_j} = 0,
	 \quad \mathcal{W}\in \R^n,
 \end{align}
with $d=2$, $n=5$, $\mathcal{W}= (\rho, u_1,u_2,h_1,h_2)^T$, and
\ba\label{f's}
f_0&=(\rho,\rho u_1, \rho u_2, h_1, h_2)^T,\\
f_1&=(\rho u_1,\rho u_1^2 + \frac12(h_2^2 - h_1^2) + p, \rho u_1 u_2- h_1h_2 , \beta h_1, -h_1u_2+h_2u_1)^T, \\
f_2&=(\rho u_2,\rho u_1u_2- h_1h_2, \rho u_2^2+ \frac12(h_1^2 - h_2^2) + p, \beta h_2 + h_1u_2 - h_2 u_1, 0)^T.
\ea
The linearization of \eqref{eq-2:abstract_form} about a constant solution 
$\mathcal{W}\equiv \overline{\mathcal{\bar W}}$ is given by 
\begin{equation}\label{eq_linearized}
	A_0 \mathcal{W}_t +\sum_{j=1}^d A_j \mathcal{W}_{x_j}=0,
	\qquad
	A_j :=  Df_j(\overline{\mathcal{W}}).
\end{equation}

Taking $\mathcal{V} = f_0\left(\mathcal{W}\right)$, where $f_0(\cdot)$ is assumed to be a (local) diffeomorphism from $\mathbb{R}^n$ to itself, we can rewrite \eqref{eq-2:abstract_form} in the standard form 
\be\label{eqs}
\mathcal{V}_t + \mathcal{F}(\mathcal{V})=0,
\ee
\be\label{Fs}
\mathcal{F}(\mathcal{V}):= \sum_{j=1}^d \mathcal{F}_j(\mathcal{V})_{x_j}
\qquad \mathcal{F}_j(\mathcal{V}):= f_j(\mathcal{V}(\mathcal{W}))=f_j(f_0^{-1}(\mathcal{W})),
\ee
with linearization about a constant solution $\mathcal{V}\equiv \overline{\mathcal{V}}=f_0(\overline{\mathcal{W}})$ 
given by
\begin{equation}\label{Veq_linearized}
	\mathcal{V}_t +\sum_{j=1}^d \mathcal{A}_j \mathcal{V}_{x_j}=0,
	\qquad
	\mathcal{A}_j := D\mathcal{F}_j(\overline{\mathcal{V}})=
	A_jA_0^{-1}(\overline{\mathcal{W}})).
\end{equation}

Likewise, the divergence-free constraint may be written in the general form
\be\label{Gamma}
\Gamma \mathcal{V}:=\sum_{j=1}^d \Gamma_j \partial_{x_j}\mathcal{V}\equiv 0,
\quad \Gamma_j\in \R^{m\times n}\equiv \const,
\ee
satisfying Dafermos' compatibility condition (involution) (cf. \cite[Chap. V, Section 5.4]{Dafermos-involution}):
\be\label{inv}
\Gamma \mathcal{F}(\mathcal{V})= - \mathcal{M}\Gamma \mathcal{V},
\ee
where $\mathcal{M}$ is first-order constant-coefficient: 
\be\label{Meq}
\mathcal{M}:=\sum_{j=1}^{d} \mathcal{M}_j \partial_{x_j},
\quad \mathcal{M}_j\in \R^{m\times m}\equiv \const.
\ee
and hyperbolic:
\be\label{Mhyp}
\sigma\left( \sum_{j=1}^d k_j \mathcal{M}_j\right) \; \hbox{\rm real and semisimple}, \quad \forall k \in \R^d.
\ee
Note that \eqref{inv} implies also the linearized version:
\be\label{Linv}
\Gamma \mathscr{L}=- \mathcal{M}\Gamma,
\ee
where $\mathscr{L}:=-\sum_{j=1}^d \mathcal{A}_j^{\pm}  \partial_{x_j}$, $x_1\gtrless 0$, is the linearized operator 
about a background constant solution.

\begin{remark}\label{rmkvalues}
	For 
	\eqref{eq-2} we have $d=2$, $m=1$, $n=5$, 
	$\mathcal{V}= \left(\rho, \rho u_1, \rho u_2, h_1, h_2\right)$, 
	$\Gamma_1= (0,0,0,1,0)$,  $\Gamma_2=(0,0,0,0,1)$, $\mathcal{M}_1=\beta$, and $M_2=0$. 
	Thus,
	$\Gamma \mathcal{V} = \mathrm{div}(h) $
	and $\mathcal{M}= \beta \partial_{x_1}$.
\end{remark}

\subsubsection{Hyperbolicity}\label{s:hcons1}
For a given reference state $\overline{\mathcal{V}}$ and direction vector $k\in \R^d$, define 
$$
\mathcal{A}(k)= \sum_{j= 1}^d k_j \mathcal{A}_j^{\pm},\quad \Gamma(k)= \sum_{j= 1}^d k_j \Gamma_j,
\quad 
\mathcal{M}(k)= \sum_{j= 1}^d k_j \mathcal{M}_j,
$$
From \eqref{Linv}, we have
\be\label{syminv}
\Gamma(k) A(k)= \mathcal{M}(k)\Gamma(k),
\ee
whence $\ker \Gamma(k)$ is a right invariant subspace and $\Range  \, \Gamma(k)^T$ a left invariant subspace
of $A(k)$.
Likewise, $\coker \Gamma(k)$ is a left invariant subspace and $\Range  \, \Gamma(k)$ a right 
invariant subspace of $\mathcal{M}(k)$.

Hyperbolicity of \eqref{eqs} with constraint \eqref{Gamma} is defined as the
property that the eigenvalues $\sigma(A(k)|_{\ker \Gamma(k)})$ of 
$A(k)|_{\ker \Gamma(k)}$ be real and semisimple for all $k\in \R^d$,
and ``weak hyperbolicity'' of \eqref{eqs} with \eqref{Gamma} as the
property that $\sigma(A(k)|_{\ker \Gamma(k)})$ be real for all $k\in \R^d$.
Hyperbolicity of \eqref{eqs} without constraint \eqref{Gamma} is defined as the
property that $\sigma(A(k))$ be real and semisimple for all $k\in \R^d$,
and ``weak hyperbolicity'' without \eqref{Gamma} as the
property that $\sigma(A(k))$ be real for all $k\in \R^d$.

\begin{proposition}\label{hypprop}

Under the above assumptions, it holds that
 \begin{enumerate}[label=(\roman*), ref=\thetheorem(\roman*)]
\hangindent\leftmargin
 \item System \eqref{eqs} 	is weakly hyperbolic with constraint \eqref{Gamma}
	if and only if it is weakly hyperbolic without the constraint.
	In particular, $\mu$, $r$ are an eigenvalue, eigenvector pair of $\mathscr{L}(k)$ if and only if either
	(a) $\Gamma(k) r=0$, or (b) $\gamma:=\Gamma(k) r$ is an eigenvector of $\mathcal{M}(k)$ with eigenvalue $\mu$,
	$0\neq k\in \R^d$, whence
	\be\label{specunion}
\sigma(A(k)) =\sigma(A(k)|{\ker \Gamma(k)})\cup \sigma(\mathcal{M}(k)|_{\Range  \, \Gamma(k)}).
\ee
\item The equation \eqref{eqs} is hyperbolic with the constraint if it is hyperbolic without the constraint.
	If \eqref{eqs} is hyperbolic with the constraint, then for every $k$ for which
	$\sigma(\mathcal{M}(k)|_{\Range  \, \Gamma(k)})\cap \sigma(A(k)|_{\ker \Gamma(k)})=\emptyset$, 
	$\sigma (A(k))$ is imaginary and semisimple.
	\item  If \eqref{eqs} is hyperbolic with the constraint, and $\sigma (A(k))$ is semisimple
	for all $k\in \R^d$, then \eqref{eqs} is hyperbolic without the constraint.
	However, in general hyperbolicity with constraint \eqref{Gamma} does not imply hyperbolicity 
	without the constraint.
\end{enumerate}
\end{proposition}

\begin{proof}

	Applying \eqref{syminv} to $(A(k)-\mu)r=0$, we obtain $(\mathcal{M}(k)-\mu) \gamma=0$, with $\gamma=\Gamma(k)r$, and thus the second assertion in (i), dichotomy (a)-(b).
	In case (a), $r\in \ker \Gamma(k)$ and so $\mu\in \sigma(A(k)|_{\ker \Gamma(k)})$; in 
	case (b), $\mu\in \sigma(\mathcal{M}(k)|_{\Range  \, \Gamma(k)})$. Combining, we obtain \eqref{specunion}.
	The first assertion in (i) then follows by hyperbolicity of $\mathcal{M}$, which implies 
	$\sigma(\mathcal{M}(k))$ real, hence
	$\sigma(\mathcal{M}(k)_S)$ real on any invariant subspace of $\mathcal{M}(k)$.

	The first assertion in (ii) follows by the fact that the eigenvalues and eigenvectors
	of $A(k)$ restricted to an invariant subset are a subset of the 
	eigenvalues and eigenvectors of $A(k)$. The second assertion in (ii) holds because
	$\sigma(\mathcal{M}(k)|_{\Range  \, \Gamma(k)})\cap \sigma(A(k)|_{\ker \Gamma(k)})=\emptyset$
	implies that the number of eigenvectors and eigenvalues of
	$\mathcal{M}(k)|_{\Range  \, \Gamma(k)}$ and $A(k)|_{\ker \Gamma(k)}$ is then $n$, together
	accounting for all eigenvalues of $A(k)$, along with the fact that $(M-\mu)\Gamma S=0$ implies
	$\Gamma (A-\mu) S=0$ and so $(A-\mu)S\in \ker \Gamma$- this, along with separation of $\sigma( M(k)|_{\Range \, \Gamma}) $ and $\sigma (A(k)|_{\ker \Gamma})$ gives $(A-\mu)(S+R)=0$ for some $R\in \ker \Gamma$, yielding 
	semisimplicity of eigenvalues lying in $\sigma(M(k)|_{\Range \, \Gamma})$.
	
	The first assertion in (iii) follows by the observation that weak hyperbolicity plus semisimplicity of
	$\sigma(A(k))$ for all $k\in \R^d$ implies hyperbolicity.
	Finally, the example $\mathcal{A_1}=\bp 0 & 1\\ 0 & 0\ep$, $\Gamma_1=(0,1)$, $\ker \Gamma_1=(1,0)^T$,
	$\mathcal{M}_1=0$ for $d=1$, $n=2$, and $m=1$ satisfies hyperbolicity with constraint (since
	$A|_{\ker \Gamma_1}=0$) and hyperbolicity of $\mathcal{M}$, but is not hyperbolic without 
	constraint, demonstrating the second assertion in (iii).
\end{proof}

\br\label{jordanrmk}
Assuming hyperbolicity of $\mathcal{A}$ with constraint $\Gamma$ and hyperbolicity of $\mathcal{M}$,
consider the generalized eigenvector equation $(\mathcal{A}(k)-\mu)^i r=0$.
Applying $\Gamma(k)$ on the left, we obtain
$(\mathcal{M}(k)-\mu)^i \Gamma(k)r=0$, giving 
$(\mathcal{M}(k)-\mu) \Gamma(k)r=0$ by hyperbolicity of $\mathcal{M}$.
But, this in turn gives $0=\Gamma (\mathcal{A}(k)-\mu) r$, or $(\mathcal{A}(k)-\mu)r\in \ker \Gamma(k)$,
from which we may determine that $r$ is a generalized eigenvector of $\mathcal{A}(k)$ of
height $2$ over a genuine eigenvector in $\ker \Gamma(k)$.
Thus, the counterexample given in the proof is indicative of the general case, the only obstruction to hyperbolicity
being possible appearance of Jordan blocks of height $2$.
\er

\begin{proposition}\label{betahyp} The $\beta$-model is weakly
hyperbolic without constraint \eqref{Gamma} and hyperbolic with the constraint.
	In the parallel case $u=(\bar u_1,0)$, $h=(\bar h_1,0)$, 
	it is not hyperbolic without the constraint if $\beta \neq \bar u_1$.
\end{proposition}

\begin{proof}
	Hyperbolicity of the MHD equations with constraint \eqref{Gamma} 
	is well known; see, e.g., \cite{MeZ2}.  Likewise, $\mathcal{M}$ in this case, since scalar ($m=1$),
	is automatically hyperbolic.  Thus, we obtain weak hyperbolicity without constraint by Proposition 
	\ref{hypprop}.
	On the other hand, taking without loss of generality $\bar u=0$, by
	Galillean invariance, direct computation using \eqref{f's}, \eqref{Veq_linearized}(ii) gives
	$A_0=\diag \{ 1, \rho, \rho, 1, 1\}$, and
	\be
	\mathcal{A}(k)= (k_1 A_1+k_2 A_2) A_0^{-1}=
	 \bp
0 & k_1 & k_2 & 0 & 0 \\
p_{\rho}k_1 & 0 & 0 & -\bar h_1 k_1- \bar h_2k_2 &\bar h_2k_1 -\bar h_1k_2 \\
p_\rho k_2 & 0 & 0 &\bar h_2k_2 -\bar h_2k_1   & - \bar h_1 k_1-\bar h_2k_2   \\
	0 &  -\frac{\bar h_2}{\bar \rho} k_2 & \frac{\bar h_1}{\bar \rho}k_2& \beta k_1 & \beta k_2 \\
	0 & \frac{\bar h_2}{\bar \rho} k_1 & -\frac{\bar h_1}{\bar \rho} k_1 & 0 & 0
\ep
	,
	\ee
	hence, in the parallel case, for $k=(0,1)$, $\Gamma(k)=(0,0,0,0,1)$,
	$\mathcal{M}(k)=0$, and
	$$
	\mathcal{A}(k)= \bp
	0 & 0 & 1  & 0 & 0\\ 
	0 & 0 & 0 & 0 &  -\bar h_1 \\ 
	p_\rho & 0 & 0 &   0    & 0\\
	0 & 0  & \frac{\bar h_1}{\bar \rho} & 0   & \beta  \\
	0 & 0 & 0  &   0  & 0
	\ep
	,
	\qquad
	\mathcal{A}(k)|_{\ker \Gamma(k)}= \bp
	0 & 0 & 1  & 0 \\ 
	0 & 0 & 0 & -\bar h_1 \\ 
	 p_\rho & 0 & 0 &   \bar h_1    \\
	0 & 0 & \frac{\bar h_1}{\bar \rho} & 0  
	\ep
	.
	$$
	Noting that the left kernel $\Span \{ (-\bar h_1,0,0,\bar \rho)\}$ of
	$\mathcal{A}(k)|_{\ker \Gamma(k)}$ is not orthogonal to the upper righthand
	block $(0,0, 0, \beta)^T$ of $\mathcal{A}(k)$ unless $\beta=0$, 
	we find that \eqref{eq-2} is not hyperbolic if $\beta \neq 0$, or, in original coordinates,
	$\beta \neq \bar u_1$.
\end{proof}

The appearance of Jordan blocks with maximum height 2 suggests that one might establish local existence
for \eqref{eqs} with loss of one derivative, as for example in the model (linear) case $u_t= v_x$, $v_t=0$.
However, we mention this, and the details of Proposition \ref{betahyp}
for general interest only;
for our purposes (namely, analysis of linear shock stability), weak hyperbolicity is what is needed.

\subsubsection{Weak involution and persistence of constraints}\label{persistence_constraint}
For $C^2$ solutions $\mathcal{V}$ of \eqref{eqs}, it is easy to see that constraint \eqref{Gamma}
is preserved under the evolution of \eqref{eqs}.  For, applying $\Gamma$ to \eqref{eqs}, we obtain
$(\partial_t -\mathcal{M})(\Gamma \mathcal{V})=0$, whence, by well-posedness (hyperbolicity) of
the constant-coefficient equation
$(\partial_t-\mathcal{M})\mathcal{V}=0$, 
if $\Gamma \mathcal{V}\equiv 0$ at time $t=0$, then $\Gamma \mathcal{V}\equiv 0$ for all $t\geq 0$ as well.
We show now that the same holds true for \emph{weak solutions} of \eqref{eqs}; indeed,
\eqref{inv} holds (weakly) also for discontinuous $\mathcal{V}$.

The key to these results is to observe that condition \eqref{inv}, 
relating a nonlinear expression on the left to a linear constant-coefficient one on the right,
implies substantial structure.  Decomposing
\be\label{Qs}
\mathcal{F}_j(\mathcal{V})= \mathcal{N}_j\mathcal{V} + \mathcal{Q}_j(\mathcal{V}),
\ee
where $\mathcal{Q}_j(\mathcal{V})=O(|\mathcal{V}|^2)$ denotes the nonlinear part of $\mathcal{F}_j$,
we have the following algebraic relations.

\begin{lemma}\label{alglem}
With $\mathcal{N}_j$, $\mathcal{Q}_j$ as above, involution property \eqref{inv} is equivalent to
\ba\label{alg}
\Gamma_j \mathcal{N}_j&= \mathcal{M}_j\Gamma_j, \quad \Gamma_j \mathcal{Q}_j= 0,\\
\Gamma_j \mathcal{N}_k +\Gamma_k \mathcal{N}_j &= \mathcal{M}_j\Gamma_k +\mathcal{M}_k\Gamma_j , 
	\quad \Gamma_j \mathcal{Q}_k +  \Gamma_k \mathcal{Q}_j = 0.
\ea
\end{lemma}

\begin{proof}
	Relations \eqref{alg}(i) may be deduced from \eqref{Linv} evaluated at $\mathcal{V}=0$, 
	comparing coefficients of like 
	derivatives of $\mathcal{V}$. Subtracting \eqref{alg}(i) from \eqref{inv} gives
	$\Gamma \sum_j \partial_{x_j}\mathcal{Q}(\mathcal{V})=0$, whence \eqref{alg}(ii) follows
	by taking $\mathcal{V}(x)=\sum w_j x_j$ for arbitrary $w\in \R^n$ to find that
	$$
	\sum_{j,k=1}^d \langle w_j, d^2(\Gamma_j \mathcal{Q}_k+ \Gamma_k \mathcal{Q}_j) w_k\rangle=0
	$$
	for all $w_j$, $w_k$, whence each coefficient
	$d^2(\Gamma_j \mathcal{Q}_k+ \Gamma_k \mathcal{Q}_j)$ vanishes separately and so
	$d(\Gamma_j \mathcal{Q}_k+ \Gamma_k \mathcal{Q}_j)=\const$, and 
	$(\Gamma_j \mathcal{Q}_k+ \Gamma_k \mathcal{Q}_j)$ must be linear in $\mathcal{V}$.
	By  $\mathcal{Q}_j(\mathcal{V})=O(|\mathcal{V}|^2)$, this implies 
	$(\Gamma_j \mathcal{Q}_k+ \Gamma_k \mathcal{Q}_j)=0$, or \eqref{alg}(ii).
\end{proof}

\begin{corollary}\label{weakinv}
	For $\mathcal{V}\in L^\infty$, \eqref{inv} holds in \emph{weak sense}; indeed, for $\phi\in C^2$,
\be\label{weakform}
	\sum_{j,k=1}^d\langle \partial_{x_k}\partial_{x_j}\phi, \Gamma_j  \mathcal{F}_k(\mathcal{V})\rangle=
	\sum_{j,k=1}^d\langle \partial_{x_k}\partial_{x_j}\phi, \mathcal{M}_j \Gamma_k  \mathcal {V}\rangle
; a.e.
\ee
\end{corollary}

\begin{proof}
Relation \eqref{weakform} follows immediately from \eqref{alg}, comparing like $\phi$-derivatives term by term.
Weak satisfaction of \eqref{inv} then follows by integration in $x$ of \eqref{weakform} for $\phi \in C^\infty_0$.
\end{proof}

\begin{corollary}\label{persistence}
An $L^\infty$ weak solution $\mathcal{V}$ of \eqref{eqs} satisfying \eqref{Gamma} weakly at time $t=0$
satisfies $(\partial_t-\mathcal{M})\Gamma \mathcal{V}=0$ and \eqref{Gamma} weakly for all $t\geq 0$.
\end{corollary}

\begin{proof}
	If $\mathcal{V}$ is a weak solution of \eqref{eqs}, then it is a weak solution of
	$\Gamma (\partial_t \mathcal{V} - \mathcal{F}(\mathcal{V})=0$, hence, by
 Corollary \ref{weakinv}, a weak solution of $(\partial_t-\mathcal{M})\Gamma \mathcal{V}=0$.
	Next, we recall that a function $U\in L^p$ satisfies a constant-coefficient linear differential
	differential equation $LU=0$ in $z=(z_1,z_2)\in \R^{m+n}$ weakly if and only if 
	$LU^\eps=0$ for all $\eps>0$, where $U^\eps:=U* \eta^\eps$ is the mollification of
	$U$ in coordinate $z_1$ 
	by a symmetric smoothing kernel $\eta^\eps(z_1):=\eps^{-1/m}\eta(|z_1|/\eps)$.
	For, denoting $\langle g, h \rangle_*=\int_{\R^m} g^* f(z)dz$,
	the relation 
	$ \langle L^*\phi, U^\eps\rangle_* =\langle L^*\phi^\eps, U \rangle_* $
	together with $U^\eps \to U$ in $L^1$ as $\eps\to 0^+$ gives the result.
	Letting $\mathcal{V}^\eps(\cdot, t) \in C^\infty$ denote mollification 
	in $x$, we thus have $ \phi^\eps:= \Gamma V^\eps\equiv 0$
	(strongly) at $t=0$ while $(\partial_t-\mathcal{M})\phi^\eps=0$ weakly
	for $t\geq 0$. It follows by hyperbolicity/uniqueness of weak solutions
	that $\phi^\eps\equiv 0$ and thus $\Gamma \mathcal{V}^\eps=0$ for all $t\geq 0$.
	Taking $\eps\to 0^+$, we find that $\Gamma \mathcal{V}=0$ weakly for all $t\geq 0$.
\end{proof}

\begin{corollary}\label{bpersist}
For any $\beta\in \R$, an $L^\infty$ weak solution $\mathcal{V}$ of the $\beta$-model \eqref{eq-2} satisfying 
	${\rm div}(h)=0$ weakly at $t=0$ satisfies \eqref{Gamma} weakly for all $t\geq 0$.
\end{corollary}

\br\label{partialver}
Corollary \ref{bpersist} includes the case $\beta=0$ for which the $\beta$-model coincides
with the original, conservative form \eqref{eq1} of the MHD equations.
For weak solutions with ${\rm div}(h)$ initially weakly vanishing,
this justifies the manipulations used
in deriving both the general $\beta$-model, and the hybrid symmetrizable hyperbolic/augmented jump condition
system of \cite{FT,BT1,MeZ2}, based on weak satisfaction of the constraint
${\rm div}(h)$.
In particular, it shows that {\it weak solutions of all three models coincide}, where defined, 
for weakly divergence-free data.
We note for the hybrid model that weak solutions are by its nature defined only for piecewise smooth solutions
with entropic shock discontinuities.
Within this class, persistence of the constraint is shown, e.g., in
\cite[Proposition 7.1]{MeZ2}, via 
weak satisfaction of 
\be\label{weakMeZ}
{\rm div}(h)_t + {\rm div}(u{\rm div}(h))=0,
\ee
the associated jump condition $[(s-u_1){\rm div}(h)]=0$ for the constraint at a shock traveling with speed $s$
%
being obtained from the jump condition for shocks of the augmented equation \eqref{weakMeZ}.
For a linearized version of this result, see \cite[Remark 3.2]{FT}.
We offer a second proof of persistence for the linearized $\beta$-model in Appendix \ref{appendix_FT_proof}, 
through  an argument like that in \cite{FT,MeZ2}. 
\er

\br[Smooth vs. weak persistence]\label{mhdpersistencermk}
It is important to notice that our framework here is rather special, even in situations when
constraints are preserved for smooth solutions: for example, when
the operator $\mathcal{M}$ in \eqref{inv} is quasilinear hyperbolic
rather than linear, constant-coefficient.
In this case, satisfaction of the constraint may be ``broken'' by appearance of shock discontinuities;
that is, persistence of constraints could hold for smooth, but not general weak solutions.
A very interesting example pointed out to us by D. Lannes is 2-D shallow water flow,
for which vorticity can be created in initially irrotational flow through 
wave breaking/formation of shock waves \cite{Lan}.
\er

\subsubsection{Compatibility of constraints I: shock waves}\label{s:shocks}
Corollary \eqref{persistence} suggests that shock waves of \eqref{eqs}
associated with characteristics of the constrained model \eqref{eqs}, \eqref{Gamma}
should be compatible with the constraint in the sense that they satisfy $\Gamma \mathcal{V}$ weakly, 
since such shocks may be expected to form from data that is initially smooth and satisfying the constraint.
That is, for a shock wave 
\be\label{genshock}
\mathcal{V}(x,t)=\overline{\mathcal{V}}(x\cdot k -st)=\begin{cases}
	\mathcal{V}_+ & x\cdot k-st>0\\
	\mathcal{V}_- & x\cdot k-st \leq 0
\end{cases}
\ee
propagating in direction $k\in S^{d-1}$ with speed $s$, and satisfying Rankine-Hugoniot conditions
\be\label{genRH}
s[\mathcal{V}]=[\mathcal{F}(k)], \qquad \mathcal{F}(k):=\sum_{j=1}^d \mathcal{F}_j k_j,
\ee
where $[h]:=h(\mathcal{V}_+)-h(\mathcal{V}_-)$ denotes jump in $h$ across the shock, 
whose speed is {\it not associated with} the fixed, complementary set of characteristics 
$\sigma(\mathcal{M}(k)|_{\Range  \,\Gamma})$ governing propagation of constraints (recall decomposition 
\eqref{specunion}),
we expect there to hold also the jump conditions 
\be\label{Gammajump}
[\Gamma(k)\mathcal{V}]=0
\ee
associated with \eqref{Gamma}.
We now verify this intuition by direct computation using Lemma \ref{alglem}.

Let us first remark that small-amplitude shocks do certainly exist in simple
characteristic fields $\alpha_j(k)\in \sigma(\mathcal{A}(k)|_{\ker \Gamma(k)})$ 
of the constrained model \eqref{eqs}, \eqref{Gamma}, by the
standard bifurcation analysis of Lax \cite{La,Sm}, with jump $[\mathcal{V}]$
in $\mathcal{V}$ across the shock front lying approximately in the associated
eigendirection $r(k)\in \ker \Gamma$, hence satisfying \eqref{Gammajump} to first order at least.
Likewise, in the complementary case of simple characteristic fields
$\alpha_j(k)\in \sigma(\mathcal{M}(k)|_{\Range \, \Gamma(k)})$ associated with propagation
of $\Gamma \mathcal{V}$.  However, in the latter case, it is easy to see that these are
contact discontinuities with $[\mathcal{V}]$ lying approximately in the associated eigendirection
$r(k)$, which, moreover, is specifically not annihilated by $\Gamma(k)$.
Hence, for such discontinuities, \eqref{Gamma} definitely {\it does not} hold weakly.
For large-amplitude shocks, or multiple characteristics, shock existence must
be studied on a case-by-case basis.
For MHD, global existence is well known, as recorded for the parallel case
in Lemma \ref{parametrization_lemma} below.

	Recall finally \cite{ZS} that shocks are classified as different types-- 
	e.g., ``Lax,'' ``overcompressive,'' or ``undercompressive'' -- according to
the difference between the numbers of incoming and outgoing
	characteristics at their endstates $\mathcal{V}_\pm$.

\begin{proposition}\label{shockprop}
Assuming involution property \eqref{inv},
a shock solution \eqref{genshock} of \eqref{eqs} satisfying \eqref{genRH}
with speed $s\not \in \sigma(\mathcal{M}(k)|_{\Range \, \Gamma})$
automatically satisfies \eqref{Gamma} in the weak sense \eqref{Gammajump}.
Moreover, the type  of the shock is the same considered
with respect to characteristics of the system with or without constraint.
\end{proposition}

\begin{proof}
	Expanding $\mathcal{F}(k)= \mathcal {N}(k) + \mathcal{Q}(k)$, with
	${N}(k):=\sum_{j=1}^d \mathcal{N}_j k_j$ and $ \mathcal{Q}(k):= \sum_{j=1}^d \mathcal{Q}_j k_j$,
	and using \eqref{alg}, we find that
$0= \Gamma(k) [ \mathcal{F}(k)-s\mathcal{V}]= \Gamma(k) (\mathcal{N}(k)-s)[\mathcal{V}]
= (\mathcal{M}(k)-s) [\Gamma(k) \mathcal{V}]$,
whence \eqref{Gammajump} follows by the assumed invertibility 
of $\mathcal{M}(k)-s$ on $\Range \, \Gamma(k)$.	
	Moreover, by assumption, the additional characteristics 
	$\sigma(\mathcal{M}(k)|_{\Range \, \Gamma})$ present for
	the unconstrained system besides the characteristics $\sigma(\mathcal{A}|_{\ker \Gamma})$
	of the constrained one are distinct from $s$, hence noncharacteristic.
	Moreover, since $\mathcal{M}(k)$ is independent of $\mathcal{V}_\pm$, the number of
	resulting new characteristics incoming to the shock is equal to the number of new
	characteristics outgoing from the shock, each incoming characteristic on one side being matched by
	an outgoing characteristic of the same sign relative to $s$ on the other side.
	As shock type is defined by the difference between the numbers of incoming and outgoing
	characteristics \cite{ZS}, it is therefore unaffected by presence or absence of the constraint.
\end{proof}

\begin{corollary}\label{shockMHD}
	For stationary MHD shocks in the $x_1$ direction 
	$ \mathcal{W}(x,t)=\overline{\mathcal{W}}(x_1)= \mathcal{W}_\pm $
	for $ x_1\gtrless 0$
	of the $\beta$-model \eqref{eq-2} and any $\beta\neq 0$, 
	${\rm div} h=0$ holds weakly, i.e., $[h_1]=0$ across the shock.
	Moreover, the type of the shock is the same considered with or without the constraint.
\end{corollary}

\begin{proof}
	For the $\beta$-model, we have (see Remark \ref{rmkvalues}) $\mathcal{M}_1=\beta$, $\mathcal{M}_2=0$,
	hence $\sigma(\mathcal{M}(k))=\beta$ for shock speed $s=0$ and direction $k=(1,0)$,
	and so $s\not \in \sigma(\mathcal{M}(k))$ precisely when $\beta\neq 0$.
\end{proof}

\br\label{shockrmk}
In the case $\beta=0$, corresponding to \eqref{eq1}, the jump condition $[h_1]=0$ 
for ${\rm div} h=0$ is consistent with but not implied by the Rankine-Hugoniot conditions,
which are in this case degenerate (not full rank) \cite{BT1,MeZ2,FT}.
We note that any inviscid shock obtainable as the limit of smooth viscous shock profiles
of \eqref{eq1} with $\mu, \eta, \nu \neq 0$ necessarily satisfies ${\rm div} h=0$,
so this assumption is not only internally consistent (as just shown), but also
consistent with the vanishing viscosity point of view.
\er

\subsubsection{Compatibility with constraints II: normal modes}\label{sec:persistency_H}
For definiteness, restrict now (without loss of generality) to the case of a zero-speed planar shock in direction $x_1$,
of standard Lax type, in particular satisfying 
\be\label{Anonchar}
\det \mathcal{A}_1(\mathcal{V}_\pm)\not = 0.
\ee
Note that \eqref{Anonchar}, by \eqref{specunion}, automatically implies 
$0\not \in \sigma(\mathcal{M}_1|_{\Range \, \Gamma_1})$.
Recall from, e.g., \cite{ZS}, the normal modes equations after shifting to a frame $(z,x_2,\dots,x_d)$ 
with discontinuity at $x_1=0$:
\be\label{int}
\lambda \widehat{\mathcal{V}} + \mathcal{A}_1^{\pm} \widehat{\mathcal{V}}_z + \sum_{j= 2}^d i\xi_j \mathcal{A}_j^{\pm} \widehat{\mathcal{V}}=0, \qquad x_1\gtrless 0,
\ee
(interior equation) and
\be\label{LRH}
Y \left(\lambda [\,\overline{\mathcal{V}}\,] + \sum_{j\ne 1} i\xi_j[f_j(\,\overline{\mathcal{V}}\,)]\right) -[\mathcal{A}_1\widehat{\mathcal{V}}]=0,
\ee
(linearized jump conditions),
where $[\cdot]$ denotes jump across $x_1=0$.
Here, $\widehat{\mathcal{V}}\in \C^n$ denotes the Laplace--Fourier transform of the coordinated-shifted
solution $\mathcal{V}$ and $Y\in \C$ the Laplace--Fourier transform of the shift, or front location in the
original spatial coordinates,
with $\xi=(\xi_2, \dots, \xi_d)\in \R^{d-1}$ the Fourier frequency in directions $(x_2, \dots, x_d)$
and $\lambda\in \C$ the Laplace frequency in time $t$.
Associated with each solution $\mathcal{V}\in L^2(\R)$ of \eqref{int}-\eqref{LRH} is a normal mode 
\be\label{nmode}
\mathcal{V}(x_1,t)= e^{\lambda t+ \sum_{j= 2}^di\xi_jx_j}\widehat{\mathcal{V}}(x_1),
\ee
of the linearized coordinate-shifted equations, 
with $\mathrm{Re}(\lambda)>0$ corresponding to linear instability.

Hereafter, we replace the shifted coordinate $z$ by its original designation $x_1$.
Define symbols 
$$
\widehat{\mathcal{A}}(\xi)= \mathcal{A}_1^{\pm} \partial_{x_1} + \widetilde{ \mathcal{A}}(\xi),\quad
\widehat{\Gamma}(\xi)= \Gamma_1 \partial_{x_1} + \widetilde{ \Gamma}(\xi),\quad \mbox{and} \quad 
\widehat{\mathcal{M}}(\xi):=\mathcal{M}_1\partial_{x_1} + \widetilde{ \mathcal{M}}(\xi),
$$
where 
$ \widetilde{ \mathcal{A}}(\xi)=\sum_{j= 2}^di\xi_j \mathcal{A}_j$, 
$\widetilde{ \Gamma}(\xi)=\sum_{j= 2}^di\xi_j \Gamma_j$, and
$\widetilde{ \mathcal{M}}(\xi)= \sum_{j= 2}^di\xi_j \mathcal{M}_j$.
Likewise, denote
$$
\widehat{\mathcal{N}}(\xi)= \sum_{j= 2}^d i\xi_j \mathcal{N}_j, \quad 
\widehat{\mathcal{Q}}(\xi)= \sum_{j= 2}^d i\xi_j \mathcal{Q}_j, \quad \mbox{and} \quad 
\widehat{\mathcal{F}}(\xi)= \sum_{j= 2}^d i\xi_j \mathcal{F}_j.
$$

\begin{lemma}\label{weaksoln}
	Assuming \eqref{inv}, \eqref{Anonchar}, for
	$\widehat{\mathcal{V}}$ piecewise smooth and satisfying \eqref{int}-\eqref{LRH}:
	\begin{enumerate}
	 \item  $(\lambda + \widehat{\mathcal{M}}) \widehat{\Gamma}\widehat{\mathcal{V}}=0$ for $x_1\gtrless 0$.
	 \item  $[\mathcal{M}_1 \widehat{\Gamma}\widehat{\mathcal{V}}]=0$ at $x_1=0$,
	hence $(\lambda + \widehat{\mathcal{M}}) \widehat{\Gamma}\widehat{\mathcal{V}}=0$ weakly on $\R$.
	\end{enumerate}

\end{lemma}

\begin{proof}
From \eqref{Linv}, \eqref{alg}, we obtain the corresponding relations
\be\label{rels}
	\Gamma_1 \mathcal{A}_1= \mathcal{M}_1\Gamma_1,\quad
\widetilde{ \Gamma} \widetilde A= \widetilde{\mathcal{M}} \widetilde{ \Gamma},
\quad
\widetilde{ \Gamma} A_1+ \Gamma_1\widetilde A= \widetilde{\mathcal{M}}\Gamma_1+ \mathcal{M}_1\widetilde{ \Gamma},
\ee
\ba\label{algFL}
	\widetilde {\Gamma} \widetilde{\mathcal{N}}&= \widetilde{\mathcal{M}}\widetilde{\Gamma}, 
	\quad \widetilde{\Gamma} \widetilde{\mathcal{Q}}= 0; \qquad
	\Gamma_1 \widetilde{\mathcal{N}} +\widetilde{\Gamma} \mathcal{N}_s1 = 
	\mathcal{M}_1\widetilde{\Gamma} +\widetilde{\mathcal{M}}\Gamma_1 , 
	\quad \Gamma_1 \widetilde{\mathcal{Q}} +  \widetilde{\Gamma} \mathcal{Q}_1 = 0.
\ea

Recall the jump condition \eqref{genRH} for the shock, specialized to the $x_1$-directional case $k=(1,0,\dots,0)$:
\be\label{1RH}
	[\mathcal{F}_1]=[\mathcal{N}_1 \overline{\mathcal{V}} + \mathcal{Q}_1(\overline{\mathcal{V}})]=0,
\ee
and the jump condition \eqref{Gammajump} for the constraint:
\be\label{1Gammajump}
	[\Gamma_1 \overline{ \mathcal{V}}]=0.
\ee
obtained by applying $\Gamma_1$ to \eqref{1RH} and using \eqref{alg} and $0\not \in \sigma( \mathcal{M}_1)$.
Similarly, applying $\tilde \Gamma$ to \eqref{1RH} and using \eqref{algFL}, we obtain
	$ 0= \tilde \Gamma [\mathcal{N}_1 \overline{\mathcal{V}} + \mathcal{Q}_1(\overline{\mathcal{V}})]
	= - \Gamma_1 [\widetilde {\mathcal{N}}\overline{\mathcal{V}} + \mathcal{Q}_1(\overline{\mathcal{V}})]
	+ [\widetilde{\mathcal{M}} \Gamma_1+ \mathcal{M}_1 \tilde \Gamma], $
or
\be\label{2Gammajump}
	\Gamma_1 [\widetilde {\mathcal{N}}\overline{\mathcal{V}} + \mathcal{Q}_1(\overline{\mathcal{V}})]=
	 (\widetilde{\mathcal{M}} \Gamma_1+ \mathcal{M}_1 \tilde \Gamma)[\overline{\mathcal{V}}].
\ee

Applying now $\Gamma_1$ to \eqref{LRH}, and using \eqref{rels}-\eqref{algFL} and \eqref{1Gammajump}, gives 
$$
\begin{aligned}
\mathcal{M}_1 [\Gamma_1 \widehat{\mathcal{V}}]&= \Gamma_1 [\mathcal{A}_1 \widehat{\mathcal{V}}] =
-Yi\xi \Gamma_1 [\widetilde{\mathcal{N}}\overline{\mathcal{V}} + \widetilde{\mathcal{Q}}(\overline{\mathcal{V}}]
= -Yi\xi (\widetilde{\mathcal{M}}\Gamma_1 + \mathcal{M}_1\tilde \Gamma) [\overline{\mathcal{V}}]
= -Yi\xi \mathcal{M}_1\tilde \Gamma [\overline{\mathcal{V}}],
\end{aligned}
$$
	yielding, by invertibility of $\mathcal{M}_1$ (a consequence of \eqref{Anonchar}, as noted above),
\be\label{rel4}
[\Gamma_1 \widehat{\mathcal{V}}]=  -Yi\xi \tilde \Gamma [\overline{\mathcal{V}}].
\ee

Writing \eqref{int} as
	$(\lambda + \mathcal{A}_1\partial_{x_1} + \widetilde{\mathcal{A}})\widehat{\mathcal{V}}=0$,
applying $\hat \Gamma$ on the left, and using \eqref{rels}, we obtain immediately
	$(\lambda + \widehat{\mathcal{M}})\hat \Gamma \widehat{\mathcal{V}}$, verifying (i).
Using \eqref{int} to express
$
\partial_{x_1} \widehat{\mathcal{V}}= \mathcal{A}_1^{-1}(\lambda + \widetilde{\mathcal{A}})\widehat{\mathcal{V}},
$
we obtain
$$
\mathcal{M}_1 [\hat \Gamma \widehat{\mathcal{V}}]:=
\mathcal{M}_1 [(\Gamma_1\partial_{x_1} +\tilde \Gamma) \widehat{\mathcal{V}}]=
\mathcal{M}_1 (-\Gamma-1\mathcal{A}_1^{-1}(\lambda +\widetilde {\mathcal{A}})) [\widehat{\mathcal{V}}],
$$
or, using $\mathcal{M}_1\Gamma_1=\Gamma_1 \mathcal{A}_1$ to express $\mathcal{M}_1\Gamma_1 \mathcal{A}_1^{-1}=\Gamma_1 $,
$
\mathcal{M}_1 [\hat \Gamma \widehat{\mathcal{V}}]=
\big(-\Gamma_1(\lambda +\widetilde {\mathcal{A}})
+\mathcal{M}_1 \widetilde{\Gamma} \big) [\widehat{\mathcal{V}}].
$
Using \eqref{rels} to express
$\mathcal{M}_1 \tilde \Gamma- \Gamma_1 \widetilde{\mathcal{A}}= \tilde \Gamma \mathcal{A}_1 -\widetilde {\mathcal{M}}$,
we obtain finally
$$
\mathcal{M}_1 [\hat \Gamma \widehat{\mathcal{V}}]=
-(\lambda +\widetilde {\mathcal{M}})[ \Gamma_1 \widehat{\mathcal{V}}]
+\tilde \Gamma [\mathcal{A}_1 \widehat{\mathcal{V}}],
$$
or, substituting for $ [ \Gamma_1 \widehat{\mathcal{V}}]$ using \eqref{1Gammajump},
and then applying \eqref{algFL}:
$$
\mathcal{M}_1 [\hat \Gamma \widehat{\mathcal{V}}]= 
Y[(\lambda +\widetilde {\mathcal{M}})\tilde \Gamma \overline{\mathcal{V}}]
+\tilde \Gamma [\mathcal{A}_1 \widehat{\mathcal{V}}]
= \tilde{\Gamma}\Big( Y  (\lambda [\overline{\mathcal{V}}]  +[\widetilde {\mathcal{F}}(\overline{\mathcal{V}})])
	+ [\mathcal{A}_1 \widehat{\mathcal{V}}] \Big), 
$$
which vanishes by \eqref{LRH}. 
By invertibility of $\mathcal{M}_1$, we thus obtain $[\widehat{\Gamma}\widehat{\mathcal{V}}]=0$,
verifying (ii).
\end{proof}

\begin{corollary}\label{sat}
Assuming \eqref{inv}, \eqref{Anonchar}, for $\widehat{\mathcal{V}}\in L^2$ piecewise smooth
and satisfying \eqref{int}-\eqref{LRH},
	we have $\widehat{\Gamma} \widehat{\mathcal{V}}\equiv 0$.
	That is, decaying normal modes automatically satisfy the 
	Laplace-Fourier transformed version of constraint \eqref{Gamma}.
\end{corollary}

\begin{proof}
By Lemma \ref{weaksoln}, $\widehat{\vp}:=\hat{\Gamma}\widehat{\mathcal{V}}$ 
	is a weak solution of $(\lambda+ \widehat{\mathcal{M}}(\xi))\widehat{\vp}=0$.
Mollifying $\widehat{\vp}$ by convolution with a standard smoothing kernel $\eta^\eps$, 
we obtain a family of $C^\infty \cap L^2$ solutions 
	$\widehat{\vp}^\eps:=\widehat{\vp} * \eta^\eps$ of $(\lambda-\widehat{\mathcal{M}}(\xi))f=0$, 
converging in $L^2$ to $\vp$ as $\eps \to 0^+$.
These solutions, if nontrivial, would represent eigenfunctions of $-\widehat{\mathcal{M}}(\xi)$ 
with eigenvalue $\lambda$.  
However, $\mathcal{M}$, being constant coefficient, has only continuous spectrum, and so each $\widehat{\vp}^\eps$ must vanish identically, as therefore does $\widehat{\vp}$ in the limit as $\eps\to 0$.
\end{proof}

\begin{corollary}\label{normalMHD}
For stationary MHD shocks in the $x_1$ direction of the $\beta$-model \eqref{eq-2}, $\beta\neq 0$, 
and $\widehat{\mathcal{V}}\in L^2$ piecewise smooth and satisfying \eqref{int}-\eqref{LRH},
	we have $\tilde \Gamma\widehat{\mathcal{V}}= \beta (\partial_{x_1}\widehat{\mathcal{V}}_4 + i\xi  \widehat{\mathcal{V}}_5 )\equiv 0$.
	That is, decaying normal modes automatically satisfy the (transformed)
	divergence-free constraint.
\end{corollary}

\br \label{normrmk}
Relation $(\lambda + \widetilde{ \mathcal{M}} \tilde \Gamma )\widetilde {\mathcal{V}}=0$
may be recognized as the Laplace-Fourier transformed version of $(\partial_t + \mathcal{M})\Gamma \mathcal{V}=0$
of Corollary \ref{persistence}. For the hybrid model of \cite{BT2,MeZ2,FT}, weak
satisfaction of the Fourier-Laplace transform analog
$$
(\lambda + \partial_{x_1}\bar u_1 + i\xi \bar u_2) (\partial_x \tilde h_1+i\xi \tilde h_2)=0
$$
of \ref{weakMeZ}, Remark \ref{partialver}, pointed out in \cite[(47)-(48), Remark 3.2]{FT},
likewise gives the result that decaying normal modes satisfy the divergence-free constraint,
a fundamental result first noted in \cite{BT1,BT2}.
\er 

\br \label{nodecayrmk}
For both the $\beta$-model and the standard (hybrid) MHD model of \cite{BT2,MeZ2,FT}, 
the fact that the mode of propagation of the constraint is scalar convection/advection shows
that not all decaying modes of $(\mathcal{A}_1\partial_{x_1}+ (\lambda + i\xi \mathcal{A}_2))
\widetilde{\mathcal{V}}=0$ satisfy the constraint $\tilde \Gamma \widetilde {\mathcal{V}}=0$.
For, inverting the dispersion relation $\lambda(k_1, \xi)= ik_1 \bar u_1 + i\xi \bar u_2$
to solve for the associated eigenvalue $\mu:=ik=(\lambda -i\xi\bar u_2)/\bar u_1 $
of $\mathcal{A}_1^{-1}(\lambda + i\xi \mathcal{A}_2)$, we see for $\xi=0$ that $\mathrm{Re}(\mu)=\mathrm{Re}(\lambda)$,
and so the mode carrying the constraint is decaying as $x\to -\infty$, but does not satisfy the constraint.
Thus, the situation of Corollary \ref{normalMHD}
is more subtle than just checking that the constraint is satisfied for all relevant modes.
\er

\subsubsection{Weak Lopatinsky stability}\label{betastab}
Gathering information, we are now ready to make conclusions regarding linearized and nonlinear stability
with and without constraints.
We first recall the {\it weak stability}, or ``weak Lopatinsky'' condition of Majda \cite{Majda},
that (i) the dimensions of the subspaces $\mathcal{E}^\pm$
of decaying modes of \eqref{int} on $x\gtrless 0$ sum to $n-1$
(consistent splitting \cite{AGJ,ZS}) and (ii) for $\mathrm{Re}(\lambda) >0$, there exist no 
decaying normal modes $\widehat{\mathcal{V}}\in L^2$ piecewise smooth and satisfying \eqref{int}-\eqref{LRH}.
Condition (i) asserts effectively that essential spectrum of the linearized evolution operator 
lies in $\mathrm{Re} (\lambda) \leq 0$, condition (ii) that point spectrum lies in $\mathrm{Re}(\lambda) \leq 0$; see
\cite{He} for further discussion.
Failure of weak Lopatinsky stability, also known as {\it strong instability},
implies exponential linear instability.

\begin{corollary}\label{weakstab}
Assuming involution property \eqref{inv}, a Lax type shock solution 
\eqref{genshock} of \eqref{eqs} in direction $x_1$ with speed 
$s\not \in \sigma(\mathcal{M}_1|_{\Range \, \Gamma_1})$
satisfies the weak Lopatinsky condition with constraint \eqref{Gamma} if and only if it 
satisfies the condition without the constraint.
\end{corollary}

\begin{proof} 
	In the one-dimensional case $\xi=0$, condition (i) reduces to
	the requirement that the shock be Lax type.
Since, by Proposition\ref{shockprop}, the type of the shock is the same with or without
constraint, this verifies the result in dimension one.
But, by Hersh's Lemma (cf. \cite{Hersh}, or \cite[Chap. 4, Lemma 4.1]{benzoni2007multi}),	
satisfaction of this condition in multi-dimensions is equivalent for weakly hyperbolic systems
to satisfaction in dimension one, whence by Proposition \ref{hypprop} the result follows also in multi-d.

Namely, one observes that existence of a pure imaginary eigenvalue $\mu=ik_1$ of 
$
	-\mathcal{A}_1^{-1} (\lambda + i\sum_{j\neq 1}\xi_j \mathcal{A}_j)
	$
with associated eigenvector $\mathcal{R}$, implies
$-\mathcal{A}_1^{-1} (\lambda + i\sum_{j\neq 1}\xi_j \mathcal{A}_j) \mathcal{R}= ik_1 \mathcal{R},$
and thus $(ik_1 A_1 + \sum_{j\neq 1}i\xi_j \mathcal{A}_j)\mathcal{R}=-\lambda \mathcal{R}$, yielding
$\lambda$ real by weak hyperbolicity, hence contradicting $\mathrm{Re}(\lambda) >0$.
But this implies that no eigenvalues cross the imaginary axis as $\xi$ and $\lambda$ are
varied within $\xi\in \R^{d-1}$, $\mathrm{Re}(\lambda)>0$, and so the number of positive/negative real
part eigenvalues remains constant, in particular equal to that in the dimension one case
$\xi=0$.
Combining our conclusions in one- and multi-d, we find that condition (i) holds automatically,
with or without constraint.
But, by Corollary \ref{sat}, condition (i) holds without constraint if and only if it holds with constraint.
\end{proof}

\begin{corollary}\label{weakstabMHD}
For stationary MHD shocks in the $x_1$ direction of {\emph either} 
the $\beta$-model \eqref{eq-2}, $\beta\neq 0$, 
{\emph or of the standard (hybrid) model of \cite{BT2,MeZ2,FT}},
weak Lopatinsky stability holds with the divergence free constraint if and only if it holds
without the constraint.
\end{corollary}

\begin{proof}
The result for the $\beta$-model follows by Corollary \ref{weakstab}.  The result for 
the hybrid model follows by the same argument, noting that all relevant properties hold for the hybrid model as well.
\end{proof}

A convenient way to test for weak Lopatinsky stability is via the {\it Lopatinsky determinant}
 $$
 \Delta(\lambda, \xi): =det(A_1^+\mathscr{E}^+, A_1^-\mathscr{E}^-, \lambda[f_0] + \sum_{j\neq 2} i\xi_j[f_j])\Big|_{x_1=0},
 $$
 defined for all $\xi\in \R^{d-2}$, $\mathrm{Re}(\lambda) >0$, where
$\mathscr{E}^{\pm}(\lambda, \xi)$ denote manifolds of (spatially) decaying  solutions of \eqref{int}
on $x_1 \gtrless 0 $; see \cite{MeZ2,ZS}, or Section \ref{sec:lop_det_constr} and \eqref{lopatinski_determinant} 
for further discussion.
Evidently, vanishing of the determinant $\Delta$ is equivalent to existence of a decaying normal mode,
or solution of \eqref{int}--\eqref{LRH}; hence, nonvanishing of the Lopatinsky determinant on $\mathrm{Re}(\lambda)>0$
is equivalent to weak Lopatinsky stability, a condition that is useful both numerically and analytically.
In the present context, by Corollary \ref{weakstabMHD}, one may test weak Lopatinsky stability using
either the hybrid model of \cite{BT2,MeZ2,FT} or, as we do here, the $\beta$-model, both {\it without}
imposition of the divergence-free constraint.
The advantage of the $\beta$ model is that it is of a form determined by a single set of equations,
so treatable by existing, well-tested code: for example, the
numerical stability package STABLAB \cite{STABLAB}, applicable to general models \eqref{eq-2:abstract_form}.

\subsubsection{Uniform Lopatinsky stability}\label{ubetastab}
We next recall the {\it uniform Lopatinsky condition} of Majda \cite{Majda}, which requires, further, that
for $\mathrm{Re}(\lambda)>0$ and  $|\lambda, \xi|=1$, the Lopatinsky determinant is uniformly bounded,
$|\Delta(\lambda, \xi)|\geq c>0$.
Under various additional structural conditions on \eqref{eqs} (cf. \cite{MeZ2}), 
the subspaces $\mathcal{E}^\pm$ of decaying modes of \eqref{int} on $x\gtrless 0$, 
hence also the Lopatinsky determinant $\Delta$, may be extended 
continuously to the closure $\mathrm{Re}(\lambda) \geq 0$, in which case uniform Lopatinsky stability is 
equivalent to nonvanishing of $\Delta$ on $\mathrm{Re}(\lambda)\geq 0$ except at the origin $(\lambda, \xi)=(0,0)$.
Under additional structural conditions on \eqref{eqs}, one may conclude further, by the symmetrizer
method of Kreiss, that uniform Lopatinsky stability implies uniform resolvent bounds on nearby perturbations of the
stationary shock, yielding also {\it nonlinear stability}, defined as well-posedness of the associated time-evolution
problem \cite{Majda,MeZ2}.

For MHD, such conditions were verified for the standard (hybrid) model in \cite{MeZ2}, yielding the following
{\it sufficient condition} for nonlinear stability.

\begin{proposition}[\cite{MeZ2}]\label{MeZprop}
	For shock waves of MHD, uniform Lopatinsky stability for the linearization of the hybrid model of
	\cite{BT2,MeZ2,FT}, ignoring the divergence-free constraint, yields nonlinear stability
	of the the perturbation equations about a Lax-type shock, with or without the constraint.
\end{proposition}

Proposition \ref{MeZprop} does {\it not} assert that uniform Lopatinsky stability with the constraint
is equivalent to stability without the constraint.  And, indeed, Corollary \ref{sat}, though still true
for $\mathrm{Re}(\lambda)=0$, is not 
much use, since the continuous extension to $\mathrm{Re}(\lambda)=0$ of decaying
subspaces $\mathcal{E}^\pm$ are no longer necessarily decaying, but may contain neutral, oscillatory modes.
Thus, vanishing of the Lopatinsky determinant implies only a solution of \eqref{LRH} belonging to 
special subspaces at plus and minus spatial infinity, and not existence of a decaying normal mode.

However, we can remedy this with a bit of further investigation.  Following \cite{MeZ2}, define
$(i\tau, \xi_0)$, $\xi_0), \tau\in \R$, to be a {\it nonglancing point} for \eqref{int} if the stable 
eigenvalues $\mu^\pm(\lambda, \xi)$ of \eqref{int} and the unstable eigenvalues $\nu^\pm (\lambda,\xi)$
have distinct limits as $(\lambda, \xi)\to (i\tau,\xi_0)$ from $\{\mathrm{Re}(\lambda)>0\}$;
otherwise, call it a {\it glancing point} for \eqref{int}.
(This simplified definition may be seen to be equivalent to the notions of glancing defined in \cite{MeZ2} 
for the cases considered there, including MHD.)
Call $(i\tau, \xi_0)$, $\xi_0), \tau\in \R$, a {\it point of continuity} for \eqref{int} 
if the associated stable subspaces $\mathcal{E}^\pm$ have limits
as $(\lambda,\xi)\to (i\tau,\xi_0)$ from $\{\mathrm{Re}(\lambda) >0\}$.
At a point of continuity, call $\widehat{\mathcal{V}}$, piecewise smooth, an {\it extended normal mode}
if it satisfies \eqref{int}-\eqref{LRH}, and approaches $\lim_{(\lambda, \xi)\to (i\tau+0^+, \xi)} \mathcal{E}^\pm$ as $x_1\to \pm \infty$.

\begin{proposition}\label{glanceprop}
Let $(i\tau, \xi_0)$, $\tau, \xi_0\in \R$, be a point of continuity for \eqref{int}, and a nonglancing point for 
	\be\label{Mint}
\lambda \hat{\phi} + \mathcal{M}_1^{\pm} \hat{\phi}_z + \sum_{j= 2}^d i\xi_j \mathcal{M}_j^{\pm} 
	\hat{\phi}=0, \qquad x_1\gtrless 0.
	\ee
Then, assuming \eqref{inv}, \eqref{Anonchar}, for an extended normal mode $\widehat{\mathcal{V}}$,
we have $\widehat{\Gamma} \widehat{\mathcal{V}}\equiv 0$.
That is, extended normal modes automatically satisfy the 
	Laplace-Fourier transformed version of constraint \eqref{Gamma}.
\end{proposition} 

\begin{proof}
	Let $e_\pm$ and $f_\pm$ denote the subspaces of decaying and growing solutions
	of \eqref{Mint} as $x_1\to \pm \infty$ for $\Re \lambda>0$, and their continuous extensions
	to $\Re \lambda \geq 0$ (well-defined, by the nonglancing assumption for \eqref{Mint}).

	For $\mathrm{Re}(\lambda)>0$, subspaces $\mathcal{E}^\pm$ correspond to decaying modes 
	$\widehat{\mathcal{V}}^\pm$ at $x_1\to \pm \infty$ of
	\be\label{Gdef}
	\d_{x_1}\widehat{\mathcal{V}}^\pm= -\mathcal{A}_1^{-1}(\lambda + \sum_{j\neq 2}i\xi_j \mathcal{A}_j)
	\widehat{\mathcal{V}}^\pm
	=: G^\pm(\lambda, \xi)\widehat{\mathcal{V}^\pm},
	\ee
	whence, by Lemma \ref{weaksoln}, 
$
\hat \phi^\pm:=
	\hat \Gamma \widehat{\mathcal{V}}^\pm= (\Gamma_1 G^\pm + \sum_{j\neq 2}i\xi_j \Gamma_j) \widehat{\mathcal{V}}^\pm
$
	is a decaying weak solution of \eqref{Mint},
	hence lies in the decaying subspaces $e^\pm$ of 
	$-\mathcal{M}_1(\lambda j+ \sum_{j\neq 2}i\xi_j \mathcal{M}_j)$ as $x_1\to \pm \infty$,
	with jump condition $[\mathcal{M_1}\hat \phi]=\mathcal{M}_1 \hat \phi=0$,
	or, by invertibility of $\mathcal{M}_1$ (recall: a consequence of \eqref{Anonchar}, 
	$ [\hat \phi]=0 $.

	By the assumed continuity of $\mathcal{E}^\pm$, we have at $(i\tau, \xi_0)$ that
	normal modes $\widehat{\mathcal{V}}$ may be realized as limits of decaying normal modes on $\mathrm{Re}(\lambda) >0$
	in the limit as $(\lambda, \xi)\to (i\tau, \xi_0)$, hence $\hat \phi=\hat \Gamma \widehat{\mathcal{V}}$
	(piecewise continuous) is the limit of decaying solutions of \eqref{Mint},
	i.e., vectors lying in 
	$e^\pm$ for $x_1\gtrless 0$.
	By the assumed nonglancing of $(i\tau, \xi_0)$ with respect to
	\eqref{Mint}- specifically, the implied continuity of 
	$e^\pm$ 
	as $(\lambda, \xi)\to (i\tau, \xi_0)$ from $\mathrm{Re}(\lambda) >0$- 
	$\hat \phi$ must thus lie in 
	$e^\pm$ 
	for all $x_1$, with $[ \hat \phi]=0$
	forcing $\hat \phi(0^+)=\hat \phi(0^-)$ at $x_1$=0.
	Noting, since $\mathcal{M}_j$ are constant, that
	$e^\pm= f^\mp$, and applying the nonglancing assumption
	as second time- specifically, the fact that $e^\pm$ and $f^\pm$ converge to 
	transverse
	limiting subspaces, or $e^+\cap f^+=\emptyset$ and $e^-\cap f^-=\emptyset$,
	we find that $\hat \phi(0)$ must lie in 
	$(e\cap f)^\pm= \{0\}$, and thus 
	$\hat \phi= \hat \Gamma \widehat{ \mathcal{V}}\equiv 0$ as claimed.
\end{proof}

Before discussing continuity of subspaces, we establish a preliminary result of interest in its own right,
comprising the Laplace--Fourier transform analog of \eqref{Linv}.

\begin{lemma}\label{keyG}
Assuming \eqref{inv}, \eqref{Anonchar}, 
	let $G_\pm(\lambda, \xi)$ as in \eqref{Gdef} be defined as 
	$-\mathcal{A}_1^{-1}(\lambda + \sum_{j\neq 2}i\xi_j \mathcal{A}_j)_\pm$,
	and $H_\pm(\lambda, \xi)$ and $L_\pm$ as 
	$-\mathcal{M}_1^{-1}(\lambda + \sum_{j\neq 2}i\xi_j \mathcal{M}_j)_\pm$
	and $(\Gamma_1 G + \sum_{j\neq 2}i\xi_j \Gamma_j)_\pm$.
	Then, 
	\be\label{keyGrel}
	(LG)_\pm= (HL)_\pm.
	\ee
\end{lemma}

\begin{proof}
	Using \eqref{rels} to express 
	\be\label{rels2}
	\Gamma_j A_j= M_j\Gamma_j; \qquad 
	\Gamma_1 A_1^{-1}= M_1^{-1}\Gamma_1; 
	\qquad \Gamma_1 A_j= M_j \Gamma_1 + M_1 \Gamma_j - \Gamma_j A_1, 
	\ee
	we may represent $L$ alternatively (dropping $\pm$ for ease of writing) as
	$$
	-(\Gamma_1 A_1^{-1}(\lambda + \sum_{j=2}^n i\xi_j A_j) + i\sum_{j=2}^n i\xi_j \Lambda_j
	=-M_1^{-1}  \Gamma_1 (\lambda + \sum_{j=2}^n i\xi_j A_j) + i\sum_{j=2}^n i\xi_j \Lambda_j,
	$$
	or, expanding and recombining using \eqref{rels2},
	\be\label{Lalt}
	L= H\Gamma_1  +  M_1^{-1}\sum_{j=2}^n \Gamma_j A_1
	\ee
	Thus, 
	$$
	LG= H\Gamma_1 G  +  M_1^{-1}\sum_{j=2}^n \Gamma_j A_1 G,
	$$
	which is equal to $H(\Gamma_1 G + \sum_{j=2}^n i\xi_j \Gamma_j)$ if and only if
	$ M_1^{-1}\sum_{j=2}^n \Gamma_j A_1 G= H \sum_{j=2}^n i\xi_j \Gamma_j), $
	or
	$$
	 -M_1^{-1}\sum_{j=2}^n \Gamma_j (\lambda+ \sum_{j=2}^n i\xi_j A_j)= 
	 -M_1^{-1}(\lambda + \sum_{j=2}^n i \xi_j M_j) \sum_{j=2}^n i\xi_j \Gamma_j),
	 $$
	 as follows by inspection using \eqref{rels2} and comparing coefficients of $\lambda$, $\xi$, and $\xi^2$
	 terms.
\end{proof}

\begin{corollary}\label{Ginv}
Assuming \eqref{inv}, \eqref{Anonchar}, $\ker L_\pm$ is a right invariant subpace 
and $\Range L_\pm^*$ a left invariant subspace of $G_\pm$.
	Assuming $L_\pm(\lambda, \xi)$ is full rank $\ell$ for all $\lambda, \xi$, then we have the
	decomposition $\C= \ker L_\pm \oplus \Range L_\pm^*$, with $\ker L$ an $n-\ell$ dimensional
	analytically varying invariant subspace of $G_\pm$, consisting precisely of those normal modes
	$\hat V$ satisfying the constraint $\hat \Gamma \hat V=0$, while $\Range L^T$ gives a complementary
	antiholomorphic space.
\end{corollary}

For both the hybrid MHD system and the $\beta$-model, it is clear by inspection that $L_\pm$ is full rank one,
hence the Corollary applies.
Corollary \ref{Ginv} has the interesting interpretation as defining a ``reduced'' interior system in the Laplace-Fourier transform setting analogous to that of the Fourier settings, pertaining precisely to those modes satisfying the 
constraint $\Gamma$.
Moreover, coordinatizing via the splitting $\ker L\oplus \Range L^*$, we may convert $G$ by a $C^\infty$
transformation ($G$, $L$ holomorphic in $\lambda$, $\xi$, and $L^*$ anti-holomorphic )
to the canonical upper block triangular form
\be\label{canonG}
	\begin{pmatrix}  G|_{\ker L}(\lambda, \xi) &  b(\lambda, \xi)\\ 0 & H(\lambda, \xi)
	\end{pmatrix},
	\qquad H=-M_1^{-1}(\lambda + \sum_{j=2}^n i\xi_j M_j), \, L=(\Gamma_1 G+ \sum_{j=2}^n i\xi_j \Gamma_j)(\lambda, \xi),
	\ee
where the lower right-hand block is obtained from \eqref{keyGrel}, via the computation 
$ (LL^*)^{-1}L G L^*= (LL^*)^{-1}HL L^*=H.  $

Defining nonglancing and points of continuity for the reduced systems $G|_{\ker L}$ and $H$ in the obvious way,
we may at a point of continuity for $G_{\ker L}$, by a further {\it continuous} (not necessarily
$C^\infty$) transformation, reduce to form
\be\label{concanonG}
	\begin{pmatrix}  g_+ & c_+ & b_+\\ 0 & g_- &  b_-\\ 0 & 0 & H \end{pmatrix} (\lambda, \xi),
	\ee
	where $g_+$ is the reduction of $G_{\ker L}$ to its stable subspace, and, at a nonglancing point of
	$H$, still further to
\be\label{congcanonG}
\begin{pmatrix}  g_+ & c_+ & b_+ & d_+\\ 0 & g_- &  b_- & d_-\\ 0 & 0 & h_+& 0\\ 0 & 0 & 0& h_- \end{pmatrix}
		(\lambda, \xi),
	\ee
	where $h_\pm$ are the reductions of $H$ to its stable and unstable subspaces.

	We connect continuity of the full and reduced sytems in two important cases.

\begin{proposition}\label{redcont}
	Assume \eqref{inv}, \eqref{Anonchar}, and $L_\pm$ is full rank for all $\lambda,\xi$. 
	\begin{enumerate}
	 \item If, local to $(i\tau, \xi_0)$, $\tau, \xi_0\in \R$, 
	there is a continuous right invariant space
	$\Span \{ R_1, \dots, R_\ell\}$ of $G$ such that 
			$LR$ is invertible for 
			$R$ defined as
			$R(\lambda, \xi)=(R_1, \dots, R_\ell)$, then $(i\tau, \xi_0)$ is a point of continuity
	for \eqref{int} if and only if it is a point of continuity of both the associated reduced system
	and the complementary system $\partial_x z=Hz$.
	\item  A point $(i\tau, \xi_0)$, $\tau, \xi_0\in \R$ that is a nonglancing point for \eqref{Mint}
	is a point of continuity for \eqref{int} if it is a point of continuity of the associated reduced system
	$G|_{\ker L}$. 
	\end{enumerate}

\end{proposition}

\begin{proof}
	(i) In this case, coordinatizing via the splitting $\ker L \oplus \Range (R)$ reduces $G$ to block-diagonal
	form $G=\blockdiag \{G|_{\ker L}, H\}$, whence the result follows trivially.
	(ii) Using the canonical coordinatization \eqref{congcanonG}, we find, for $\mathrm{Re}(\lambda) >0$ and
	local to $(i\tau, \xi_0)$, that
	$E_+=
	\begin{pmatrix} \Id & 0 & 0 & 0\end{pmatrix} \oplus r_+$, where $r$ is of form
	$r_+(\lambda, \xi)= \begin{pmatrix} * & 0 & \Id & 0\end{pmatrix}$; equivalently, that there exists $r_+$ of
		this form with range a right-invariant subspace of $G$ modulo $e_+$, 
		where $e_+$ denotes the continuous
		extension of the stable subspace of $H$.
		(Here, we are using in an important way the fact that $\sigma(g_-)$ and $\sigma (h_+)$ 
		are separated for $\mathrm{Re}(\lambda)>0$, in eliminating the second row of $r_\pm$.)
		But this evidently is equivalent to 
		$$
		E_+= \begin{pmatrix} \Id \\ 0 \\0\\0\end{pmatrix} \oplus 
		\begin{pmatrix} 0 \\ 0 \\\Id\\0\end{pmatrix}\equiv \constant,
			$$
		yielding continuity as a trivial consequence.
\end{proof}

\begin{corollary}\label{mhdcont}
For stationary MHD shocks in the $x_1$ direction of {\emph either} 
the $\beta$-model \eqref{eq-2}, $\beta\neq 0$, 
{\emph or of the standard (hybrid) model of \cite{BT2,MeZ2,FT}},
	all $(i\tau,\xi_0)$ are points of continuity for the associated interior equation \eqref{int},
	for both reduced and full models, so long as 
	\be\label{ndeg}
	(|\bar h|^2/p'(\bar \rho))_\pm\neq 1.
	\ee
\end{corollary}

\begin{proof}
	({\it hybrid model})
	It was shown in \cite{MeZ2} that, assuming \eqref{ndeg}, all points $(i\tau, \xi_0)$ are points of 
	continuity for the hybrid model without constraint,
a consequence of \cite[Lemma 7.2]{MeZ2} specialized to the 2D case (that is, ignoring middle, 
Alfven modes $\lambda_2^\pm$), together with \cite[Theorem 5.2]{MeZ2}.
	Moreover, the hyperbolic characteristic $\lambda_0(k)= \bar u_1k_1 + \bar u_2 k_2$ carrying nonzero
	divergence of $h$, with associated eigenvector $R_0(k)=(0,0,0, k_1,k_2)$,
	may be directly inverted to obtain an analytic
	eigenvalue $\mu_0(\lambda, \xi)=  \bar u_1^{-1} (\lambda+ \bar u_2 i\xi)$ 
	of $G=-\mathcal{A}_1^{-1}(\lambda + i\xi_0 \mathcal{A}_2)$
	with associated eigenvector $S_0=(0,0,0, \mu/i, \xi)$; see \cite[(A.4)]{MeZ2}.
	For $\mu_0(i\tau, \xi_0)=i k_1$ pure imaginary, and $(\lambda, \xi)\to (i\tau, \xi_0)$,	this is not annihilated by the constraint, since
	$$
	(\Gamma_1 \mu_0 + \Gamma_1 i \xi)S_0 = \mu_0 \left(\frac{\mu_0}{i}\right) + i\xi (\xi)= i\Big ( \left(\frac{\mu_0}{i}\right)^2 + \xi^2\Big)
	\to
	 i(k_1^2+\xi^2)\neq 0
	$$
	as $\mu_0\to ik_1$.
	But, this implies also $LS_0\neq 0$, since $LS_0=(\Gamma_1 \mu + i\xi \Gamma_2)S_0$.
	Such points are therefore points of continuity for the reduced system by Proposition \ref{redcont}(i).
	If $\mu_0$ is not pure imaginary, on the other hand, then for $(\lambda, \xi)$ near $(i\tau, \xi_0)$,
	$\lambda_0$ is strictly stable, and so, along with other strictly stable modes, continues analytically
	and transversally to the space $\ker (L)$, so that continuity for the full system reduces also in this case
	to continuity of the stable subspace for the reduced problem.
	Thus, we conclude from continuity for the full system continuity for the reduced system as well.

	({\it $\beta$-model}) 
	For both models, the complementary system associated with $H$, being scalar, is necessarily
	nonglancing.  Since we have already shown continuity for the reduced system via our
	treatment of the hybrid model, we thus obtain continuity for the full $\beta$-model by
	Proposition \ref{redcont}(ii).
\end{proof}

We define {\it reduced weak Lopatinsky stability} as nonexistence of solutions of
\eqref{int}-\eqref{LRH} lying in 
$\mathcal{E}^\pm$
for $x_1\gtrless \infty$.
By Corollary \ref{weakstab},
this is equivalent to weak Lopatinski stability for the full system, without constraint.

When the decaying subspaces {$\mathcal{E}^\pm_0$ of the reduced system extend continuously for all $\lambda=i\tau$, we may define a notion of {\it reduced uniform Lopatinsky stability} as nonexistence of solutions of
\eqref{int}-\eqref{LRH} lying in the limiting spaces $\mathcal{E}^\pm_0$ for $x_1\gtrless 0$.
When the subspaces $\mathcal{E}^\pm$ of the full system extend continuously as well, we have by
Proposition \ref{glanceprop} that full and reduced uniform Lopatinski stability are equivalent.

Combining results, we have as an immediate consequence of Corollary \ref{mhdcont} the following definitive conclusion,
elucidating and significantly extending 
the fundamental observation of Blokhin-Trakhinin \cite{BT2,FT}
that weak stability for constrained MHD is equivalent to weak stability without constraint.
In particular, this shows for the first time that the nonlinear stability condition of
Proposition \ref{MeZprop} is sharp.

\begin{proposition}\label{mhdglance}
For stationary MHD shocks in the $x_1$ direction of {\emph either} 
the $\beta$-model \eqref{eq-2}, $\beta\neq 0$, 
{\emph or of the standard (hybrid) model of \cite{BT2,MeZ2,FT}},
uniform Lopatinsky stability holds with the divergence free constraint if and only if it holds
	without the constraint, if and only if it holds for the reduced model defined by $G|_{\ker L}$.
\end{proposition}

\subsection{The viscous case}\label{s:viscous}
We conclude this section by a brief treatment of the analogous but simpler viscous case,
generalizing and expanding on the approach used in \cite[pp. 2 and 62-64]{JYZ} for the equations of viscoelasticity
(there introduced somewhat implicitly and without particular emphasis in the course of other computations).

System \eqref{eq-2}, in the viscous case $\mu,\eta,\nu >0$, may be expressed in the general form 
\be\label{veqs}
\mathcal{V}_t + \mathcal{F}(\mathcal{V})=
	 \sum_{j,k=1}^{d} (B_{j,k}(\mathcal{V})\mathcal{V}_{x_j})_{x_k} 
\ee
 augmenting the inviscid system \eqref{eqs},
with linearization about a constant solution $\mathcal{V}\equiv \mathcal{V}_0$
given by
\begin{equation}
	\mathcal{V}_t= L\mathcal{V}:=  -\sum_{j=1}^d \mathcal{A}_j \mathcal{V}_{x_j}+
	 \sum_{i,j =1}^{d} B_{i,j}\mathcal{V}_{x_i,x_j},
	 \notag
\end{equation}
where $\mathcal{A}_j=D\mathcal{F}_j(\mathcal{V}_0)$ and $\mathcal{B}_{j,k}=\mathcal{B}_{j,k}(\mathcal{V}_0)$. 
Likewise, Dafermos' involution condition \eqref{inv} is replaced by a viscous counterpart
\be\label{vinv}
\Gamma \mathcal{F}(\mathcal{V})= - \mathcal{M}\Gamma \mathcal{V},
\ee
where $\mathcal{M}$ is now second-order constant-coefficient: 
\be\label{vMeq}
\mathcal{M}:=\sum_{j=1}^{d} \mathcal{M}^1_j \partial_{x_j} + 
	 \sum_{i,j =1}^{d} \mathcal{M}^2_{i,j} \partial_{x_i} \partial_{x_j},
\ee
and (strictly) parabolic:
\be\label{vMhyp}
\mathrm{Re} \sigma\left( \sum_{j=1}^d k_j \mathcal{M}^1_j
	 + \sum_{i,j=1}^{d} \mathcal{M}^2_{i,j} k_i k_j \right) \leq - \theta  |k|^2,
	 \qquad \hbox{\rm some $\theta>0$.}
\ee
As in the inviscid case, \eqref{vinv} induces the linearized version
\be\label{vLinv}
\Gamma L=- \mathcal{M}\Gamma.
\ee

\subsubsection{Dissipativity}\label{s:hp} At a constant state $\mathcal{V}_0$,
the viscous condition analogous to weak hyperbolicity in the inviscid case is 
{\it weak dissipativity}:
\be\label{weakdiss}
\mathrm{Re}\ \sigma( L(k)):=
\mathrm{Re}\ \sigma\left( -\sum_{j=1}^d k_j \mathcal{A}_j
	 + \sum_{i,j=1}^{d} \mathcal{B}_{i,j} k_i k_j \right) \leq - \theta\frac{ |k|^2}{1+ |k|^2},
	 \qquad \hbox{\rm some $ \theta>0$,}
\ee
a condition on the dispersion relation for $L$ related to the genuine coupling condition of Kawashima
\cite{Kaw,Z1,Z2,GMWZ6}.
From \eqref{vMhyp}--\eqref{vLinv}, we find, applying $\Gamma(k)$ to the eigenvalue equation
$(L(k)-\lambda)\mathcal{V}=0$ that $(-\mathcal{M}(k)-\lambda)\Gamma(k) \mathcal{V}=0$ for
$\mathcal{M}(k):= \sum_{j=1}^{d} \mathcal{M}^1_j k_j+ \sum_{i,j =1}^{d} \mathcal{M}^2_{i,j} k_i k_j $,
whence $\Gamma(k) \mathcal{V}=0$ or else $\lambda \in \sigma (-\mathcal{M}(k))$ verifying the weak dissipativity
condition by \eqref{vMhyp}.
Thus, similarly as in Proposition \ref{hypprop} of the inviscid case, 
weak dissipativity holds with constraint $\Gamma(k) \mathcal{V}=0$
if and only if it holds without the constraint.

It was verified in \cite{MeZ2} that \eqref{weakdiss} holds for the hybrid MHD model with and without constraint, whence
it holds for the $\beta$-model as well.

\subsubsection{Persistence of constraints}\label{s:vpersistence}
From \eqref{vMeq}--\eqref{vMhyp}, we obtain immediately persistence of the constraint $\Gamma \mathcal{V}=0$
for smooth solutions of \eqref{veqs}, similarly as in Corollary \ref{persistence} for the inviscid case.
For, applying $\Gamma$ to \eqref{veqs} gives 
\be\label{vGeq}
(\partial_t -M)\Gamma \mathcal{V}=0,
\ee
with initial data $\mathcal{V}\equiv 0$, whence we obtain $\Gamma \mathcal{V}=0$ for all $t\geq 0$ by well-posedness
of \eqref{vGeq}.
Recall that the standard viscous existence/stability theory concerns smooth ($H^s$) solutions \cite{Z1}.

\br\label{viscMrmk}
In the case of the MHD $\beta$-model, $Mz= -\beta \partial_{x_1} z + \nu \Delta z$, and $M\Gamma \mathcal{V}= -\beta \partial_{x_1}
(\rm{div} h) + \nu \Delta ({\rm div} h)$, so that solutions of \eqref{veqs} satisfy
$$0= (\partial_t- M\Gamma \mathcal{V}= \partial_t ({\rm div} h) +
+\beta \partial_{x_1} (\rm{div} h) + \nu \Delta ({\rm div} h).
$$
For the hybrid model of \cite{BT2,MeZ2,FT}, solutions satisfy
$ 0= \partial_t ({\rm div} h) + \nabla \cdot (u (\rm{div} h)) + \nu \Delta ({\rm div} h) $,
giving persistence by a similar argument to the above \cite{MeZ2}.
\er

\subsubsection{Compatibility of shock profiles}\label{s:vshockcompatibility}
Likewise, similarly as in Proposition \ref{shockprop}, we find 
for smooth viscous shock profiles
\be\label{vprof}
\mathcal{V}(x,t)=\overline{\mathcal{V}}(x\cdot k -st),
\qquad \lim_{z\to \pm \infty} 
\overline{\mathcal{V}}(z)=\mathcal{V}_\pm,
\ee
$|k|=1$, that
shock profiles of the unconstrained system automatically satisfy
the constraint.  For, applying $\Gamma$ to the traveling-wave ODE $ -s k\cdot \nabla \mathcal{V} +\mathcal{F}(\mathcal{V})=0$,  rearranging, and integrating from $-\infty$ to $z$ gives 
$$
\gamma '= \mathcal{N}(k) \gamma
$$
where $\gamma= \Gamma \mathcal{V}= \Gamma(k) \overline{\mathcal{V}}'$ and 
$\mathcal{N}(k)= (\mathcal{M}^2(k))^{-1}(\mathcal{M}^1(k)-s)$ is constant-coefficient.
(Here, $\mathcal{M}^2(k)(k)$ is invertible by parabolicity.)
Noting that $\gamma(z)\to 0$ as $z\to \pm \infty$ by convergence of the shock profile to endstates
$\mathcal{V}_\pm$, we find that $\gamma \equiv 0$, since there are no nontrivial solutions of a 
constant-coefficient ODE decaying at both $\pm \infty$.

\subsubsection{Compatibility of normal modes}\label{s:vnormalcompatibility}
Restricting without loss of generality to the case of a zero-speed planar shock in direction $x_1$,
$\mathcal{V}=\overline{\mathcal{V}}(x_1)$, linearizing, and taking the Fourier transform in transverse directions
$x_2,\dots, x_d$, we obtain \cite{ZS,Z1,Z2} the generalized eigenvalue, or normal modes equation
\ba\label{vint}
0= (\lambda  -\mathcal{L}(\xi)) \widehat{\mathcal{V}}&=
\lambda \widehat{\mathcal{V}} + (\mathcal{A}_1^{\pm} \widehat{\mathcal{V}})_z + 
\sum_{j= 2}^d i\xi_j \mathcal{A}_j^{\pm} \widehat{\mathcal{V}}\\
&=
	  (\mathcal{B}_{1,1} \widehat{\mathcal{V}}_z)_z
	  +\sum_{j=2}^{d} i\xi_j \mathcal{B}_{1,j} \widehat{\mathcal{V}}_z 
	  \quad +\sum_{i=2}^{d}i \xi_i  (\mathcal{B}_{i,1} \widehat{\mathcal{V}})_z 
	  +\sum_{i,j=2}^{d}-\xi_i \xi_j \mathcal{B}_{i,j} k_i k_j  \widehat{\mathcal{V}},
\ea
with $(\xi_2, \dots, \xi_d)\in \R^{d-1}$ corresponding to Fourier frequencies, 
and existence of decaying solutions $\widehat {\mathcal{V}}$ for $\mathrm{Re}(\lambda)>0$ corresponding to
exponential linear instability.

Denoting by $\widehat {\mathcal{M}}(\xi)$ the Fourier transform in directions $x_2, \dots, x_d$ of $\mathcal{M}$,
we have from \eqref{vinv}, evidently, $\hat \Gamma \widehat{\mathcal{L}}= -\widehat{\mathcal{M}}\hat \Gamma$,
whence normal modes, similarly as in the inviscid case, satisfy also the constant-coefficient ODE
$$
(\lambda + \widehat{\mathcal{M}})(\hat \Gamma \widehat{\mathcal{V}})=0.
$$
By the absence of nontrivial solutions of constant-coefficient ODE that decay at both $\pm \infty$, we find
immediately as in Corollary \ref{sat} for the inviscid case,
that {\it decaying normal modes automatically satisfy the constraint}, i.e.,
$\hat \Gamma \widehat {\mathcal{V}}=0$.

\subsubsection{Weak and strong Evans-Lopatinski stability}\label{s:vweak}
In general, \eqref{weakdiss} is assumed to hold at the endstates $\mathcal{V}_\pm$ of a shock,
whence, by a standard Hersch-type lemma \cite{ZS,Z1,Z2}, one obtains ``consistent splitting'' of the
eigenvalue equations \eqref{vint} on $\{(\xi,\lambda): \, \mathrm{Re}(\lambda) \geq 0\} \setminus \{0,0\}$, that is, the property
that the dimensions of the subspaces of decaying solutions at $+\infty$ and $-\infty$ of \eqref{vint}
sums to the total dimension $N$ of the solution space, as do the dimensions of the subspaces of exponentially growing
solutions at $+\infty$ and $-\infty$.  It follows that one may define an Evans function $\mathcal{D}$
consisting of the Wronskian
of $N$ solutions comprising bases of the decaying solutions at $\pm\infty$, with vanishing of $\mathcal{D}(\xi,\lambda)$
corresponding to existence of a decaying normal mode for $\xi,\lambda$.
Following \cite{ZS,Z1,Z2,GMWZ6}, we define strong Evans-Lopatinsky stability \cite{ZS,Z1,Z2,GMWZ6} 
as the absence of decaying normal modes- or, equivalently,
zeros of the Evans function $\mathcal{D}$- on $\{(\xi,\lambda):\, \mathrm{Re}(\lambda) \geq 0\} \setminus \{0\}$, and
weak Evans-Lopatinsky stability as the absence of zeros of the Evans function on $\{(\xi,\lambda):\, \mathrm{Re}(\lambda) >0\}$.
Failure of weak Evans-Lopatinski stability is also called strong Evans-Lopatinski instability, implying
exponential instability of the background shock.

From the definition, and the observations in Section \ref{s:vnormalcompatibility}, we obtain immediately
the viscous analogs of Corollaries \ref{weakstab} and \ref{weakstabMHD}, showing that {\it both weak and strong
Evans-Lopatinsky stability hold for the $\beta$- or hybrid model of MHD with constraint if and only if they
hold without constraint.}

\subsubsection{Uniform Evans-Lopatinsky stability and nonlinear stability}\label{s:vuniform}
The uniform Evans-Lopatinsky condition, analogous to the uniform Lopatinsky condition of the inviscid case,
is a uniform lower bound $|\mathcal{D}(\xi,\lambda)|\geq \theta /|(\xi,\lambda)|$, $\theta>0$, on
$\{(\xi,\lambda):\, \mathrm{Re}(\lambda) \geq 0\} \setminus \{0\}$.
In nice cases, in particular, for fast Lax shocks in MHD, this is equivalent to strong Evans-Lopatinsky stability
plus transversality of the shock profile as a solution of the traveling-wave ODE, plus satisfaction of the uniform
Lopatinsky condition for the associated inviscid shock; moreover, uniform Evans-Lopatinsky stability is equivalent
to nonlinear stability of the underlying planar shock both with respect to large time and small viscosity asymptotics
\cite{Z1,Z2,GMWZ6}.

In these favorable cases, uniform Evans-Lopatinsky stability with constraint is evidently (from our previously remarked oobservations, above) equivalent to uniform Evans-Lopatinsky stability without constraint.
On the other hand, the relation between uniform Evans-Lopatinsky stability and uniform stability of the associated inviscid shock has not been shown for slow MHD shocks; whether or not this holds has been cited in \cite{GMWZ6} as an
important open problem in the nonlinear stability theory for MHD shocks.
Thus, in the present setting, issues pertaining to uniform Evans-Lopatinski, nonlinear stability, and their relations,
remain unclear.
As our main object is {\it instability}, far from the origin $(\xi, \lambda)=(0,0)$,
this is no obstruction in the present study.  However, we mention it as an important aspect of the theory for further
development.

\section{Inviscid stability analysis: the Lopatinsky determinant}\label{Lopatinski_analysis}

The main goal in this section is to introduce the Lopatinsky determinant, which will be the main tool used in the study of inviscid stability. 
Initially we derive some properties of the shock type. 
Recall that shocks are categorized as Lax, undercompressive and compressive depending on the number of characteristics entering the shock. Lax shocks are further categorized by their characteristic field: the unique family entering on both sides. To begin with, we obtain a useful parametrization of the shocks we study. 

Recall that the Rankine-Hugoniot conditions \eqref{genRH} specialized to a stationary shock 
$ \mathcal{W}(x,t)=\overline{\mathcal{W}}(x_1)= \mathcal{W}_\pm $ in the $x_1$ direction
of the MHD equations \eqref{eq1} are 
\begin{eqnarray}\label{RH_shock}
[f_1(\overline{\mathcal{W}})] = \left[\begin{array}{c}
                                \brh \,\overline{u}_1 \\
                                \brh \,\overline{u}_1^2 - \frac{\overline{h}_1^2}{2}  + a\brh^{\gamma} \\
                                \brh \,\overline{u}_1 \overline{u}_2 - \overline{h}_1 \overline{h}_2  \\
                                \beta \overline{h}_1 \\
                                \overline{u}_1 \overline{h}_2 - \overline{u}_2\overline{h}_1
                               \end{array} \right] =0,
\end{eqnarray}
where $[f_1(\overline{\mathcal{W}})] := f_1(\overline{\mathcal{W}}^+) - f_1(\overline{\mathcal{W}}^-)$.  


\begin{lemma}[Parametrization of MHD planar shocks with zero speed]\label{parametrization_lemma} Let $$\mathcal{W}(x_1, x_2, t)  = \mathcal{W}^{\pm}(x_1) = (\rho^{\pm},u_1^{\pm}, u_2^{\pm},h_1^{\pm}, h_2^{\pm})(x_1),$$  $x_1 \gtrless 0$,  be a planar shock solution satisfying the Rankine-Hugoniot conditions in \eqref{RH_shock}. Assume that the shock is parallel, i.e.,  $u_2^{+} =0$ and $h_2^{+}=0$. 

\begin{enumerate}[label=(\roman*), ref=\thetheorem(\roman*)]
\hangindent\leftmargin
\item One can parametrize the slow shocks connecting to the state $(\brh^{+},u_1^{+}, 0,h_1^{+}, 0)$ to the right (i.e., $x_1>0$) using variables $R$ and $M$, defined as
 \begin{eqnarray} \label{R_M-parametrization}
 R = \frac{\rho^+}{\rho^-} = \frac{u_1^-}{u_1^+} \quad \mbox{and} \quad    M^2 = \frac{R^{\gamma} -1}{\gamma R^{\gamma}(R -1)}, 
\end{eqnarray}
where $M^2 = \frac{(u_1^+)^2}{p_{\rho}(\rho^+)}$ ($M$ is also known as the \textit{downstream  Mach number}), and  $u_2^-=0$, $h_2^- = 0$, $h_1^+ = h_1^-$;
\item\label{parametrization_lemma:type_of_shock}  For large magnetic field,  slow shocks, i.e., 2-shocks, are characterized by  $R >1$ and $M <1$; 

\item \label{parametrization_lemma:transition_parameter} There exist two numbers 
\begin{align*}
 H_* = u_1^+\sqrt{\rho^+}&, \qquad H^* = u_1^-\sqrt{\rho^-} ,
\end{align*}
 according to which three scenarios are possible: \textit{fast} Lax shocks (i.e., extreme, or gas-dynamical type) for $0\leq |h_1| \leq H_*$, \textit{intermediate} shocks for $H_* \leq |h_1| \leq H^*$ and \textit{slow} shocks  for $H^* \leq |h_1|$ (in particular $|h_1|\to \infty$). 
   \end{enumerate} 
\end{lemma}
 
We remark that  in the context of 3D MHD, the 2-D slow shocks seen here 
become $3$-shocks
\begin{proof}
Considering the jump conditions in \eqref{RH_shock} we readily observe that the ratios $\frac{\rho^+}{\rho^-}$ and  $\frac{u_1^-}{u_1^+}$ are in fact equal, so that one can define the variable $R$ as in \eqref{R_M-parametrization}.  Both properties $u_2^-=0$, $h_2^- = 0$ follow from solving the third and fifth rows of the jump conditions in \eqref{RH_shock}; the fourth equation in \eqref{RH_shock} implies that $h_1^+ = h_1^-$. Setting $M := \frac{u_1^+}{\sqrt{p_{\rho}(\rho^+)}}$, we obtain the rightmost condition in \eqref{R_M-parametrization} from the second relation in the jump condition \eqref{RH_shock}. This establishes (i). \newline
With regards to (ii), classical compressibility conditions in gas dynamics, namely $p_{\rho}(\cdot) >0$,  gives that  $R>1$; by inspection of \eqref{R_M-parametrization}, one can see that this condition implies that $0 \leq M <1$, for $\gamma \geq 1$.

Last, we prove (iii). The type of the shock is determined by the number of incoming characteristics, or
	eigenvalues of $A_0^{-1}A_1$, in the $x_1$ direction on either side of
	the shock, where $A_j$ are as in \eqref{eq_linearized}.
	Computing $A_0^{-1}A_1$ and its spectrum on both sides of the shock, we have: 
\begin{equation*}
 A_0^{-1}A_1 = \left(\begin{array}{ccccc}
u_{1} & \rho & 0 & 0 & 0 \\
\frac{p_{\rho}}{\rho} & u_{1} & 0 & 0 & 0 \\
0 & 0 & u_{1} & 0 & - \frac{h_{1}}{\rho} \\
0 & 0 & 0 & \beta & 0 \\
0 & 0 & -h_{1} & 0 & u_{1}
\end{array}\right), \qquad \sigma\left((A_0^{-1}A_1)^{\pm}\right) = \left\{u_1^{\pm} \pm\sqrt{p_{\rho}^{\pm}}, \beta, u_1^{\pm} \pm \frac{h_1}{\sqrt{\rho}^{\pm}}\right\}.
\end{equation*}
In the large magnetic field scenario  $h_1 \to \infty$, $\beta >0$ and a zero speed shock we count 2 negative/3 positive (resp. 1 negative/4 positive) eigenvalues for $x_1 >0$ (resp., $x_1 <0$) whenever $u_1^{\pm} \pm\sqrt{p_{\rho}^{\pm}} \lessgtr 0$.  Using the definitions in \eqref{R_M-parametrization} we conclude  that this is equivalent to $M<1$ when we use  the constraints for  $x_1 >0$; similarly, we derive that $R>1$   using the constraints for  $x_1 <0$. Therefore, the discussion above says that we just need to analyze the signs of $1 - \frac{h_1}{u_1^+\sqrt{\rho^+}}$ and $ 1 - \frac{h_1}{u_1^-\sqrt{\rho^-}} = 1 - \frac{h_1}{u_1^+\sqrt{R \rho^+}}$. As $R>1$ we have 
$$ 
1 - \frac{h_1}{u_1^+\sqrt{\rho^+}} <  1 - \frac{h_1}{u_1^-\sqrt{\rho^-}}  =  1 - \frac{h_1}{u_1^+\sqrt{R\,\rho^+}}.
$$
Therefore, recalling $R = \frac{u_1^-}{u_1^+}$, we obtain the following table  corresponding with \eqref{h_lower_upper_star}:

\begin{table}[h]
\begin{center}
\begin{tabular}{|c|c|c|}
\hline
Parameter range & Num of positive e-vals (Left,Right) & Shock type\\
\hline
\begin{minipage}[c][1.3cm]{0.3\textwidth}
\begin{center}
$ h_1 > H^*:= u_1^-\sqrt{\rho^-}$ 
\end{center}
\end{minipage}    & \begin{minipage}[c][1.3cm]{0.3\textwidth}
		      \begin{center}
		      (4,3) 
		      \end{center}
		      \end{minipage}                               & \begin{minipage}[c][1.3cm]{0.3\textwidth}
									\begin{center}
									Lax 2-shock (Slow)
									\end{center}
									\end{minipage}  \\
 \hline
\begin{minipage}[c][1.3cm]{0.3\textwidth}
\begin{center}
 $ h_1<H_*:=u_1^+\sqrt{\rho^+}$ 
\end{center}
\end{minipage}    & \begin{minipage}[c][1.3cm]{0.3\textwidth}
		      \begin{center}
		       (5,4)
		      \end{center}
		      \end{minipage}                               & \begin{minipage}[c][1.3cm]{0.3\textwidth}
									\begin{center}
									 Lax 1-shock (Fast)
									\end{center}
									\end{minipage}  \\									
\hline
\begin{minipage}[c][1.3cm]{0.3\textwidth}
\begin{center}
 $ H_*<h_1<	H^* $  
\end{center}
\end{minipage}    & \begin{minipage}[c][1.3cm]{0.3\textwidth}
		      \begin{center}
		      (5,3)
		      \end{center}
		      \end{minipage}                               & \begin{minipage}[c][1.3cm]{0.3\textwidth}
									\begin{center}
									 Doubly overcompressive
									\end{center}
									\end{minipage}  \\
 \hline
\end{tabular}
\end{center}
\caption{Shock types in 2-D MHD.\label{table:shock_types}}
\end{table}
\end{proof}
Notice that whenever $\beta >0$ the $\beta$-model preserves slow shocks, but this property can also be verified whenever 
a positive multiple  of $\mathrm{div}(h)$ is added to any upstream/downstream side of the propagating shock. 
In particular, one may add different multiples to upstream and downstream sides, as for example in \cite{MeZ2,FT},
allowing one to symmetrize the equations while preserving 2-shock structure.\footnote{
Nonconstant $\beta$ is inconvenient however for the viscous case, destroying conservative structure.}|
Last, we remark that the case $\beta =0$ is degenerate, for the matrix $A_1$ is not invertible in this case; the invertibility of $A_1$ is necessary in the construction of the Lopatinsky determinant, which we discuss next.

\subsection{Lopatinsky determinant: construction and asymptotic analysis}\label{sec:lop_det_constr} In this section we study the Lopatinsky determinant associated 
with inviscid parallel shocks and the onset of instability.  Initially, we study the behavior of the system \eqref{eq_linearized} by taking its Laplace-Fourier transform (in $t$ and  $x_2$, respectively). We obtain
\begin{eqnarray} \label{eq4_1}
\lambda A_0^{\pm}v + A_1^{\pm}v_{x_1} + i \xi A_2^{\pm}v= 0  \implies v_{x_1} = - (A_1^{\pm})^{-1}\left(\lambda A_0^{\pm} +  i \xi A_2^{\pm}\right)v, 
\end{eqnarray}
where $v$ is the Laplace-Fourier transform of $u$ and 
\begin{equation*}
	- (A_1^{\pm})^{-1}\left(\lambda A_0^{\pm} +  i \xi A_2^{\pm}\right)= \left(\begin{array}{ccccc}
-\frac{\lambda u_{1}}{u_{1}^{2} - p_{\rho}} & \frac{\lambda \rho}{u_{1}^{2} - p_{\rho}} & -\frac{i \, \rho u_{1} \xi}{u_{1}^{2} - p_{\rho}} & 0 & 0 \\
\frac{\lambda p_{\rho}}{{\left(u_{1}^{2} - p_{\rho}\right)} \rho} & -\frac{\lambda u_{1}}{u_{1}^{2} - p_{\rho}} & \frac{i \, p_{\rho} \xi}{u_{1}^{2} - p_{\rho}} & 0 & 0 \\
-\frac{i \, p_{\rho} u_{1} \xi}{\rho u_{1}^{2} - h_1^{2}} & 0 & -\frac{\lambda \rho u_{1}}{\rho u_{1}^{2} - h_1^{2}} & -\frac{i \, h_1 u_{1} \xi}{\rho u_{1}^{2} - h_1^{2}} & -\frac{h_1 \lambda}{\rho u_{1}^{2} - h_1^{2}} \\
0 & 0 & -\frac{i \, h_1 \xi}{\beta} & -\frac{\lambda}{\beta} & -\frac{{\left(i \, \beta - i \, u_{1}\right)} \xi}{\beta} \\
-\frac{i \, h_1 p_{\rho} \xi}{\rho u_{1}^{2} - h_1^{2}} & 0 & -\frac{h_1 \lambda \rho}{\rho u_{1}^{2} - h_1^{2}} & -\frac{i \, h_1^{2} \xi}{\rho u_{1}^{2} - h_1^{2}} & -\frac{\lambda \rho u_{1}}{\rho u_{1}^{2} - h_1^{2}}.
\end{array}\right),
\label{eq:A_mat_lop}
\end{equation*}
where the variables $(u_1, \rho, p_{\rho})$ should be read $(u_1^+, \rho^+, p_{\rho^+})$ (resp., $(u_1^-, \rho^-, p_{\rho^-})$) whenever $x_1 >0$ (resp. $x_1 <0$). Equation \eqref{eq4_1} consists of two systems of ODEs in the interior: one in  $x_1 > 0$, one in $x_1 < 0$,  coupled through Rankine-Hugoniot conditions at $x_1 =0$ (\textbf{Appendix \ref{Rankine}}; see also \cite{ZS} for further discussion on the technique). The Lopatinsky determinant is defined as
\begin{equation} \label{lopatinski_determinant}
 \Delta(\lambda, \xi) =\det(A_1^+\mathscr{E}^+, A_1^-\mathscr{E}^-, \lambda[f_0] + i\xi[f_2])\Big|_{x_1=0}.
\end{equation}
The parameter $\lambda$ is a spectral parameter indicating solutions to the system \eqref{eq_linearized} with growth $\sim e^{\lambda t}$ (thus, $\mathrm{Re}(\lambda) >0$ corresponding to instability);  $\mathscr{E}^{\pm}(\lambda, \xi)$ denote manifolds of (spatial) decaying  solutions in $x_1 \gtrless 0 $. In the rest of the paper we omit the dependence of these spaces on $\lambda$ and $\xi$, simply writing $\mathscr{E}^{\pm}$. 

\subsubsection{Large-$h_1$ asymptotics}

Following \cite{FT, BT1} we now study the large magnetic field $\overline{h}_1$ asymptotics. It is convenient to define the quantity $\displaystyle{\eps := \frac{1}{\overline{h}_1}}$, which parametrizes the underlying viscous profile $\overline{\mathcal{V}}=\overline{\mathcal{V}}^{(\eps)}$. 
Our analysis consists of Taylor expanding the roots of the Lopatinsky determinant defined in \eqref{lopatinski_determinant} considered as a function in $\eps$ with $\xi$ held fix at $1$, i.e., $\eps \mapsto \lambda(\eps)$. It is shown that
 \begin{equation}\label{compare_with_FT}
\lambda(\eps) = \lambda_2 \eps^2 + \mathcal{O}(\eps^3)  
 \end{equation}
 where $\lambda_2 >0$. One can conclude that an unstable regime occurs in the large magnetic field scenario, as verified in \cite{FT}.

The study of the spaces $\mathscr{E}^{\pm}$ is equivalent to  analyzing the eigenvalues of  $-(A_1^{\pm})^{-1}(\lambda A_0^{\pm} +i\xi A_2^{\pm})$ and their associated eigenspaces. The approach we adopt relies on careful analytical estimates allied with the use of symbolic computations (carried out in SAGE \cite{sage} \cite{github-BMZ}). The main idea is the following: assume that the spectral parameter $\lambda$ can be expanded as 
\begin{equation}\label{expansion_lambda}
\lambda = \lambda_0 + \lambda_1\eps + \lambda_2 \eps^2 + \ldots
\end{equation}
Thus each element $\mu \in \sigma(-(A_1^{\pm})^{-1}\left(\lambda A_0^{\pm} +  i \xi A_2^{\pm}\right))$ can be expanded as 
\begin{equation} \label{expansion_spectrum}
\mu = \mu(\lambda) =  \mu_0(\lambda_0) + \eps \mu_1(\lambda_0, \lambda_1) + \eps^2 \mu_2(\lambda_0, \lambda_1, \lambda_2) + \mathcal{O}(\eps^3), 
\end{equation}
Afterwards we find a set of eigenvectors of $-(A_1^{\pm})^{-1}\left(\lambda A_0^{\pm} +  i \xi A_2^{\pm}\right)$ spanning the spaces $\mathscr{E}^{\pm}$; namely, a mapping 
\begin{equation*} 
  \mu  \mapsto X^{\pm}(\mu): \sigma\left(-(A_1^{\pm})^{-1}\left(\lambda A_0^{\pm} +  i \xi A_2^{\pm}\right)\right) \to \mathscr{E}^{\pm}.
\end{equation*}
For the sake of 
convenience
we drop the indexes ``$\pm$'' for now, since these formulas work for both cases if one use the notation in \eqref{R_M-parametrization}. If we also expand  the eigenvector $X(\mu)$ in \eqref{expansion_eigenvector} in  $\eps$ terms we have
\begin{equation}\label{expansion_eigenvector}
 X(\mu) = X^{(0)}(\mu_0) + \eps X^{(1)}(\mu_0, \mu_1)+ \eps^2 X^{(2)}(\mu_0, \mu_1, \mu_2) +\ldots. 
\end{equation}

%
\begin{obs}\label{multilinearity_determinant_Lopatinski} With regards to the symbolic computations, a few remarks are in 
	order:
\begin{enumerate}[label=(\roman*)]\label{multilinearity_determinant_Lopatinski:numbered}
 \item \label{multilinearity_determinant_Lopatinski:numbered:1} In order to find the expansion in \eqref{expansion_spectrum} we  find the characteristic polynomial $p(\cdot)$ of the matrix 
$-(A_1^{\pm})^{-1}\left(\lambda A_0^{\pm} +  i \xi A_2^{\pm}\right)$ when \eqref{expansion_lambda} holds; we conclude by matching coefficients. The characteristic polynomial $p(\mu)$ can be shown to 
expand as 
$p(\mu) = p(\mu_0, \mu_1, \mu_2, \ldots) =  p_0(\mu_0) + \eps p_1(\mu_0, \mu_1) + \eps^2 p_2(\mu_0, \mu_1, \mu_2) + \ldots $,
with $p_j$ computed explicitly using \textit{SAGE}; 
\item  \label{multilinearity_determinant_Lopatinski:numbered:2}Higher order terms in the expansion \eqref{expansion_spectrum}  in terms of $\eps$ can be easily obtained using \textit{SAGE}, because once we have $\mu_0 \, \ldots \,  \mu_i$ the problem of finding $\mu_{i+1}$ is linear; these terms won't be written explicitly here though;
\item  \label{multilinearity_determinant_Lopatinski:numbered:3}It is not hard to show that  $- \frac{\lambda}{\beta}$ is an eigenvalue of 
	$-(A_1^{\pm})^{-1}\left(\lambda A_0^{\pm} +  i \xi A_2^{\pm}\right)$. This also holds true in  the nonparallel case, as we will show later;
\item  \label{multilinearity_determinant_Lopatinski:numbered:4}Notice that, upon scaling, we can take $\xi =1$ (redefine $\lambda \to \lambda \xi$);
\item  \label{multilinearity_determinant_Lopatinski:numbered:5}Let $X_i$ denote the $i^{th}$ column of a square matrix $A = [X_1, X_2, \ldots X_n]$ and $X_i(\eps) = X_i^{(0)} + \eps X_i^{(1)} + \eps^2 (X_i^{(2)}) + \mathcal{O}(\eps^3)$. Then 
\begin{equation*}
 det(A)  = A_0 + \eps A_1 + \eps^2 A_2 + \mathcal{O}(\eps^3),
\end{equation*}
where $A_k = \underset{a_1+\ldots + a_n = k }{\sum} det(X_1^{(a_1)}, \ldots, X_n^{(a_n)})$. 
\end{enumerate}
\end{obs}
Taking into account these observations, we plug \eqref{expansion_lambda}, \eqref{expansion_spectrum} and \eqref{expansion_eigenvector} into \eqref{lopatinski_determinant}, to obtain
\begin{equation} \label{expansion_lopatinki_determinant}
  \Delta^{(\eps)}(\lambda,1) =
  \Delta_{0} + \eps \Delta_{1} + \eps^2 \Delta_{2} + \mathcal{O}(\eps^3).
\end{equation}
In what follows we shall exploit the multi-linearity of the determinant function in order to find the terms in this expansion; each term $\Delta_i$ is a function of $\left(\lambda_i\right)_{i\in \mathbb{N}},$ but this dependence will be most of the time omitted and made explicit as we carry out our computations. We make use of \eqref{expansion_lopatinki_determinant} to verify the condition $\Delta^{(\eps)}(\lambda, 1)\equiv 0$: we look for $(\lambda_i)_{i\in \mathbb{N}}$ that gives    $\Delta_{i}(\lambda) = 0$ for all $i \in \mathbb{N}.$
An important step in the analysis consists of an explicit representation of the manifolds $\mathscr{E}^{\pm}$ referred to in \eqref{lopatinski_determinant}, an investigation that we reformulate as the study of the eigenvalues of the matrices in \eqref{eq4_1}. Indeed, one can observe by inspection that  $- \frac{\lambda}{\beta}$ is an eigenvalue of  $-(A_1^{\pm})^{-1}\left(\lambda A_0^{\pm} +  i A_2^{\pm}\right)$, while an expansion of the eigenvalues of the form \eqref{expansion_spectrum} readily shows that their zeroth order terms $\mu_0$ for $x_1 \gtrless 0$ are
\begin{eqnarray*}
\underbrace{ \pm 1  + \mathcal{O}( \eps^2) }_{\gtrless 0},\quad \, \underbrace{- \frac{1}{u_1 + \sqrt{p_{\rho}}}(\lambda_0  + \eps \lambda_1)  + \mathcal{O}(\eps^2)}_{<0},\quad  \underbrace{ - \frac{1}{u_1 - \sqrt{p_{\rho}}}(\lambda_0  + \eps \lambda_1)  + \mathcal{O}(\eps^2)}_{>0},\quad  \underbrace{-\frac{\lambda}{\beta}}_{<0}, \end{eqnarray*}
where all the variables $(u_1, \rho, p_{\rho})$ should be read $(u_1^+, \rho^+, p_{\rho^+})$ (resp., $(u_1^-, \rho^-, p_{\rho^-})$) whenever $x_1 >0$ (resp. $x_1 <0$). 
Therefore, since we are looking for ``decaying'' manifolds we must have, for  $x_1 >0$, eigenspaces associated to the following eigenvalues:
\begin{equation*} \label{negative_eigenvalues}
 - 1  + \mathcal{O}( \eps^2),\quad  - \frac{1}{u_1^+ + \sqrt{p_{\rho^+}}}(\lambda_0  + \eps \lambda_1)  + \mathcal{O}(\eps^2),\quad \mbox{and}\quad -\frac{\lambda}{\beta}.
\end{equation*}
Analogously,  ``decaying'' manifolds in  $x_1 <0$ must be eigenspaces associated to the eigenvalue   $\displaystyle{1  + \mathcal{O}( \eps^2)}$.  Notice that the number of eigenvalues in each interior $x_1 \gtrless 0$ is consistent with the analysis derived from Hersh's Lemma, hence it suffices to  analyze the number of positive and negative eigenvalues of $- A_1^{-1}A_0$ .

\subsubsection{Asymptotic instability} 

Choosing an appropriate parametrization of the  decaying manifolds $\mathscr{E}^{\pm}$ one can show that $\Delta_0 =\Delta_1=0$ in equation \eqref{expansion_lopatinki_determinant}. A careful computation shows then that
$\Delta_{2} = \Delta_{2}(\lambda_0,\mu_0) = det(X_1^{(1)}, X_2^{(0)}, X_3^{(0)}, X_4^{(1)}, X_5^{(0)}).$
We have 
\begin{align*}
&(X_1^{(1)}, X_2^{(0)}, X_3^{(0)}, X_4^{(1)}, X_5^{(0)}) = \nonumber \\
&=
\left(\begin{array}{ccccc}
0 & E_* & 0 & 0 & -\frac{2 \, \lambda_{0} u_1^+(\rho^- - \rho^+)}{(u_1^+)^{2} - p_{\rho^+}} \\
0 & \frac{i \, \lambda_{0}^{3}}{{\left(u_1^+ + \sqrt{p_{\rho^+}}\right)}^{2} \rho^+} - \frac{i \, \lambda_{0}}{\rho^+} & 0 & 0 & \frac{\lambda_{0} (\rho^- - \rho^+)}{\rho^+} \\
-(\lambda_{0} - u_1^+)^2  & 0 & 0 & -\frac{u_1^-(\lambda_0 + u_1^+)}{u_1^+} (\lambda_0 + u_1^-) & 0 \\
i u_1^+(\lambda_{0} - u_1^+) & 0 & -\lambda_{0} & i u_1^-(\lambda_0 + u_1^-) & 0 \\
\lambda_{0}(\lambda_{0} - u_1^+)& 0 & i \, u_1^+ & \frac{\lambda_0 u_1^-}{u_1^+}(\lambda_0 + u_1^-) & 0
\end{array}\right),
\end{align*}
where
$\displaystyle{ E_*=\frac{i \, \lambda_{0} (u_1^+)^{2}}{{\left(u_1^+ + \sqrt{p_{\rho^+}}\right)} p_{\rho^+}} + \frac{i \, \lambda_{0}^{3} u_1^+}{{\left(u_1^+ + \sqrt{p_{\rho^+}}\right)}^{2} p_{\rho^+}} - \frac{i \, \lambda_{0}^{3} (u_1^+)^{2}}{{\left(u_1^+ + \sqrt{p_{\rho^+}}\right)}^{3} p_{\rho^+}} - \frac{i \, \lambda_{0} u_1^+}{p_{\rho^+}}.}$
The block structure of this matrix allow us to see right away that $\Delta_{2} = \mathcal{O}(\lambda_0^3)$ and  $\Delta_{2} = o(\lambda_0)$. There exists a $\lambda_0 = \mathcal{O}(1)$ satisfying $\Delta_{2} =0 $ at which, however, the parametrization of the manifolds $\mathscr{E}^{\pm}$ through their eigenvectors is lost because two of those eigenvectors coincide. These points are called \textit{glancing modes};  this issue was also pointed out in \cite[\S 3.4]{FT}. Resorting to  generalized eigenvectors one can show that the Lopatinsky determinant does not vanish for this value of $\lambda_0$, hence the only solution to $\Delta_{2} =0$ is $\lambda_0=0$; see also \cite{github-BMZ}.

In the search for roots of  $\Delta_{3}=0$ the now look for  $\lambda_1$. Thanks to the  multilinearity of the determinant function, we readily obtain that   $\Delta_3=0$. We go to the next term in \eqref{expansion_lopatinki_determinant}: 
thanks to Observation \ref{multilinearity_determinant_Lopatinski} it is not hard to see that  $ \Delta_4 = det(X_1^{(1)}, X_2^{(1)}, X_3^{(0)}, X_4^{(1)}, X_5^{(1)}).$ 
A computation shows that the latter determinant has order $\mathcal{O}(\lambda_1^2)$, but not order $o(\lambda_1^2)$, which  implies that 
$\lambda_1 =0.$
A similar analysis leads to the expression of the next term, that is, 
$$\Delta_5 = det(X_1^{(1)}, X_2^{(2)}, X_3^{(0)}, X_4^{(1)}, X_5^{(1)}).$$
Since the first two rows are linearly dependent, we can see that the determinant of the latter matrix is zero. A more involved analysis is necessary in dealing with the $\eps^{6}$-order term in \eqref{expansion_lopatinki_determinant}:  $\Delta_{6}$ can be written as
\begin{eqnarray*}
 \Delta_6 = det(X_1^{(2)}, X_2^{(2)}, X_3^{(0)}, X_4^{(1)}, X_5^{(1)}) + det(X_1^{(1)}, X_2^{(2)}, X_3^{(0)}, X_4^{(2)}, X_5^{(1)}) +\nonumber \\
 det(X_1^{(1)}, X_2^{(2)}, X_3^{(0)}, X_4^{(1)}, X_5^{(2)}) +det(X_1^{(1)}, X_2^{(3)}, X_3^{(0)}, X_4^{(1)}, X_5^{(1)}).
\end{eqnarray*}
A simple analysis shows that the last determinant is zero due to the structure of columns 1, 3, 4 and 5. We end up with
\begin{eqnarray*}
 \Delta_{6} = \frac{i \, {\left(R - 1\right)}^{2} M^{3} R^{3} \lambda_{2} \rho^- (u_1^+)^{7}}{M^{2} - 1} - \frac{2 \, {\left(i \, R - i\right)} M^{2} R \lambda_{2}^{2} (u_1^+)^{4}}{M - 1}.
\end{eqnarray*}
%
Setting $ \Delta_{6} =0$ and solving for $\lambda_2$, we obtain
%
\begin{equation}\label{finally_the_coefficient}
 \lambda_2 = \frac{M R^{3} \rho^- (u_1^+)^{3} - M R^{2} \rho^- (u_1^+)^{3}}{2 \, {\left(M + 1\right)}} = \frac{{\left(R - 1\right)} M R^{2} \rho^- (u_1^+)^{3}}{2 \, {\left(M + 1\right)}}.
\end{equation}
This function is clearly positive for all $R >1$, from where we readily conclude inviscid instability for all values of $\gamma\geq1$, a result that extends and improves the results of \cite[Page 3036]{FT}, which were limited to the case $\gamma \in [1,2]$. Furthermore, our result shows a higher level of accuracy when compared to the results of \cite{FT} (see also the Appendix \ref{Rankine}) and those of \cite{SE}; this improvement is clear once we compare the predicted analytically determined value for the instability of the Lopatinsky determinant with those values observed numerically, as discussed in Section \ref{viscous_evans_function}.  

\subsubsection{Comparison with previous results: parallel case} \label{comparison}

We investigate the Lopatinsky determinant and its roots numerically for different values of $\gamma$, corresponding to monoatomic gas ($\gamma = 5/3$) and  diatomic gas ($\gamma = 7/5$; for instance, $O^2$) , $\gamma = 3$ (artificial gas) and compare those values to the analytically predicted result in equation \eqref{finally_the_coefficient}; the comparison is shown in table \ref{inviscid_table}.
\begin{table}[h]
\begin{small}
\begin{tabular}{|c|c|c|c|c|c|c|c|c|c|c|}
\hline
 $\gamma$ $\backslash  R$ & 1.5 & 2 & 2.5 & 3 & 3.5 & 4 & 4.5 & 5 & 5.5 & 6 \\
\hline 
\hline
$\gamma = 5/4$, exact& 7.40e-2 & 1.01e-1 & 1.13e-1& 1.18e-1 & 1.20e-1&1.20e-1&1.19e-1& 1.18e-1 &1.16e-1&14e-1\\
\hline
$\gamma = 5/4$, est& 6.65e-2 &9.76e-2 & 1.10e-1& 1.18e-1 & 1.16e-1&1.16e-1&1.16e-1& 1.18e-1&1.13e-1& 1.10e-1\\
\hline
$\gamma = 5/4$, err & 1.01e-1 & 3.64e-2 &2.98e-2& 2.54e-4& 2.64e-2&2.72e-2&2.16e-2& 2.26e-3 &2.29e-2&3.52e-2\\
\hline
\hline
$\gamma = 7/5$, exact&7.33e-2&0.100& 1.11e-1& 1.15e-1&1.17e-1&1.17e-1&1.16e-1&1.14e-1&1.13e-1&1.11e-1\\
\hline
$\gamma = 7/5$, est&6.65e-2&9.97e-2&1.10e-1& 1.15e-1&1.16e-1&1.16e-1&1.14e-1&1.15e-1&1.10e-1&1.10e-1\\
\hline
$\gamma = 7/5$, err&9.44e-2&2.32e-3&8.06e-3& 6.05e-3&4.70e-1&3.36e-3&1.51e-2&4.11e-3&2.02e-2&6.98e-3\\
\hline
\hline
$\gamma = 5/3$ exact& 7.24e-2&9.77e-2&1.08e-1&1.11&1.12e-1& 1.12e-1& 1.11e-1&  1.09e-1&1.07e-1& 1.05e-1 \\
\hline
$\gamma = 5/3$ est& 6.92e-2&9.84e-2&1.09e-1& 1.12&1.16e-1&1.11e-1& 1.12e-1&112e-1&1.08e-1&1.05e-1  \\
\hline
$\gamma = 5/3$ err&4.37e-2&7.23e-3& 7.70e-3 &1.73e-3&3.43e-2&3.02e-3&1.01e-2&2.54e-2&4.65e-3&1.88e-3 \\
\hline
\hline
$\gamma = 3$ exact& 6.78e-2&8.77e-2&9.40e-2&9.53e-2& 9.47e-2&9.32e-2&9.13e-2&8.93e-2&8.73e-2& 8.54e-2  \\
\hline
$\gamma = 3$ est& 7.11e-2&9.41e-2&9.98e-2& 1.02e-1&1.02e-1&1.05e-1&1.01e-1&9.60e-2&9.84e-2&9.29e-2  \\
\hline
$\gamma$ = 3 err& 4.89e-2&7.30e-2&6.17e-2&7.26e-2&7.98e-2&1.23e-1&1.10e-1&7.50e-2&1.26e-1&8.86e-2   \\
\hline
\end{tabular}
\caption{ Comparison of exact coefficient with numerical coefficient. Parameters $u_1^- = 1$, $\rho^- = 1$ are fixed; recall that $u_1^+ = u_1^-/R$. For each value of $\gamma$ and $R$ we record the exact answer, the numerical estimate (est), and the relative error (err) between the two. Numerical estimates of the coefficient were determined by computing the roots $\lambda(h_1)$ of the Lopatinsky determinant for several values of $h_1\subset [2,16]$ and then using curve fitting of $\log(\lambda)$ and $\log(1/h_1)$. The maximum relative error is 4.70e-1 and the average relative error is 4.83e-2. \label{inviscid_table}}\end{small}\end{table}
Another representative description of the good agreement between the analytical result and the numerical study can be also seen in Figures \ref{fig158} and \ref{fig157}. In particular, Fig. \ref{fig157} points out the accuracy of our results when compared to those presented in \cite{FT}. However, it is worthwhile to stress that the analysis in the latter paper gives the correct order for the root of the Lopatinsky determinant, i.e., $\lambda = \mathcal{O}(\eps^2)$; apparently, this result was known at a formal level in the Astrophysics community since the late 80's; see for example  \cite[\S 2.1]{SE}.

\subsection{Full inviscid stability diagram}
In this section, following the scaling of \cite{BHZ1}, we fix  $u_1^-=1$, $\rho^-=1$ so that  $\rho^- u_1^- = \rho^+ u_1^+ =1$, where the latter is due to the Rankine-Hugoniot conditions \eqref{RH_shock}. With regards to the parametrization of Lemma \eqref{parametrization_lemma}, it consists of  $R = \frac{1}{u_1^+}$. Note for this choice of parameters that the slow shock classification of \eqref{h_lower_upper_star} simplifies to 
$$ h_1 >H^* =1.$$

We complete our study of parallel inviscid shock stability by a numerical stability analysis over all parameters, complementing the asymptotic study of the previous subsections. Recall \cite{FT} that 1-D stability, or nonvanishing of $\Delta(\lambda,1)$,  has previously been verified. Thus, without loss of generality, we may fix $\xi=1$ by homogeneity 
(see Observation \ref{multilinearity_determinant_Lopatinski}-\ref{multilinearity_determinant_Lopatinski:numbered:4}), reducing the question of stability to nonvanishing of $\Delta(\lambda, 1)$ on $\mathrm{Re}\left( \lambda\right) \geq 0$. Noting that $\Delta(\lambda, 1)$ is analytic in $\lambda$, this can be done by a winding number computation. Indeed, if we evaluate $\Delta(\lambda, 1)$ along a contour in the complex plane that encloses any possible unstable roots of $\Delta(\lambda, 1)$, and if the resulting image contour has winding number 0, then the associated shock is stable, and if the winding number is positive, then the shock is unstable. We numerically evaluate $\Delta(\lambda, 1)$ along the contour $\partial (\{z\in \C: \mathrm{Re}(z) \geq 0\}\cap \{z\in \C: |z| = r\}$ where $R>0$ is is sufficiently large. In practice, we took $r = 10$, which appears to be amply large. When $\mathrm{Re}(\lambda)=0$, the real part of the eigenvalues of \eqref{eq:A_mat_lop} collapse to zero making it difficult numerically to detect the correct bases for evaluating $\Delta(\lambda, 1)$. To get around this technical difficulty, we simply in practice shift to the right of the contour on which we compute $\Delta(\lambda, 1)$ by $1e-4$. As displayed in Fig. \ref{fig264}  (b), we examine inviscid stability for $u_1^+\in \{ 0.05, 0.1,...,0.9,0.95\}$ and $h_1\in \{1.1,1.2,...,3.9,4\}$.

The results displayed in Fig. \ref{fig264}-(b) indicate a single stability transition for each fixed $u_1^+$ sufficiently close to $u_1^-=1$ as $h_1$ is increased from $H^* =1$ (stability) to $\infty$ (instability). For smaller $u_1^+$, corresponding to larger-amplitude waves, all slow shocks appear to be multi-d unstable {\it independent of the strength of the magnetic field $h_1$,}  hence there is no stability transition.  Similarly as in the large-$h_1$ case, we observe through a winding number computation of the Evans function that in the  unstable case the instability corresponds to a double real root, so that the stability transition as described in the introduction corresponds to passage of a double root through the origin.

\subsubsection{The critical destabilization parameter.}\label{sec:destabilization} Based on the above observations, to pinpoint the location of the stability transition $h_1$ for a given fixed $u_1^+$, we have only to numerically solve $\Delta(0,1)=0$, considered as an equation in $h_1$. (Note that, as a consequence of reflection symmetry,  $\Delta(\lambda, \xi)$ may be normalized to be real valued for $\lambda \in \mathbb{R}$).
Again, to avoid technical difficulties to do with pure imaginary $\lambda$, we solve the approximate equation $\Delta(\lambda_0, 1)=0$ where $\lambda_0 = 1e-5$ via the bisection method, where $\Delta(\lambda, 1)$ is normalized by $\Delta(0.1,1) = 1$.  The result is displayed in Fig. \ref{fig264}(a)(b) together with the results of our more complete coarse-mesh computations. The thick dashed line in Fig. \ref{fig264} marks the critical destabilization parameter. 
Apparently the stability region is exclusively determined by this critical destabilization parameter curve.

\subsection{Remarks on the nonparallel case}

Our approach to the nonparallel case relies on numerical winding number computations, using  similar techniques to those used in the  parallel case: we first compute an eigenbasis associated to decaying manifolds on both $ x_1 \gtrless 0$ sides, and later expand  their entries in $\eps$ in order to obtain an expression for the Lopatinsky determinant $\Delta$ in terms of $\eps$ as in \eqref{expansion_lopatinki_determinant}. One obtains an expansion of the type
%
\begin{eqnarray}\label{lop_nonparallel}
 \Delta^{(\eps)}(\lambda,1) = \frac{1}{\eps^6}\Delta_{-6} + \frac{1}{\eps^5}\Delta_{-5} + \frac{1}{\eps^4}\Delta_{-4} + \ldots
\end{eqnarray}
with $\lambda = \lambda_0 + \lambda_1 \eps + \mathcal{O}(\eps^2)$. It is not hard to see that $\Delta_{-6} = \Delta_{-5} = 0$, since both corresponding matrices have rows of zeros. As before, in order to find instability we need to show that there exists a solution to  $\Delta_{-4} = \Delta_{-4}(\lambda_0) = 0$ such that $Re(\lambda_0) >0$.
The analytical study of this determinant is very complicated, even if we use symbolic computations, so we approach this part numerically: we verify numerically using a winding number computations that  the Lopatinsky determinant \eqref{lop_nonparallel} has a simple  root $\lambda = \lambda_0 + \lambda_1 \eps + \mathcal{O}(\eps^2)$ such that $\lambda_0$  has a positive real part, so we do have instability. This result confirms and 
elucidates the assertions in \cite[Remark 3.5]{FT} regarding the nonparallel case. In particular, that the critical eigenvalue is order $\mathcal{O}(1)$ rather than order $\mathcal{O}(\eps^2)$ as in the parallel case. However, the root we found does not agree with the explicit formula given in \cite[Remark 3.5]{FT};  as the  proof of this formula is not given in \cite{FT} we are unable to determine the reason for this discrepancy.

Besides these asymptotics, for some specific shocks we also carried out a numerical Lopatinsky analysis, as one can see from the results displayed in Fig. \ref{fig183} that in the nonparallel case the roots are not real valued. This appearance of complex roots corresponds to a break of $\mathcal{O}(2)$ symmetry upon linearization (see also \cite{Mo,JYZ,FrKSch}).

\begin{SCfigure}
 \centering
\caption{Plot of the root of the non-parallel MHD Lopatinsky determinant as a function of $u_2^+$ when $u_1^+  = 0.9$, $\rho^+ = 1.1$, $u_1^- = 1$, $\rho^- = 1$, $\gamma = 5/3$, $h_1 = 5$, $h_2^+ = u_3^+ = h_3^+ = 0$,  and we vary $u_2^+$. The root is approximately $r = 9\times 10^{-4}-i u_2^+$.  }
\label{fig183}
\includegraphics[scale=0.4]{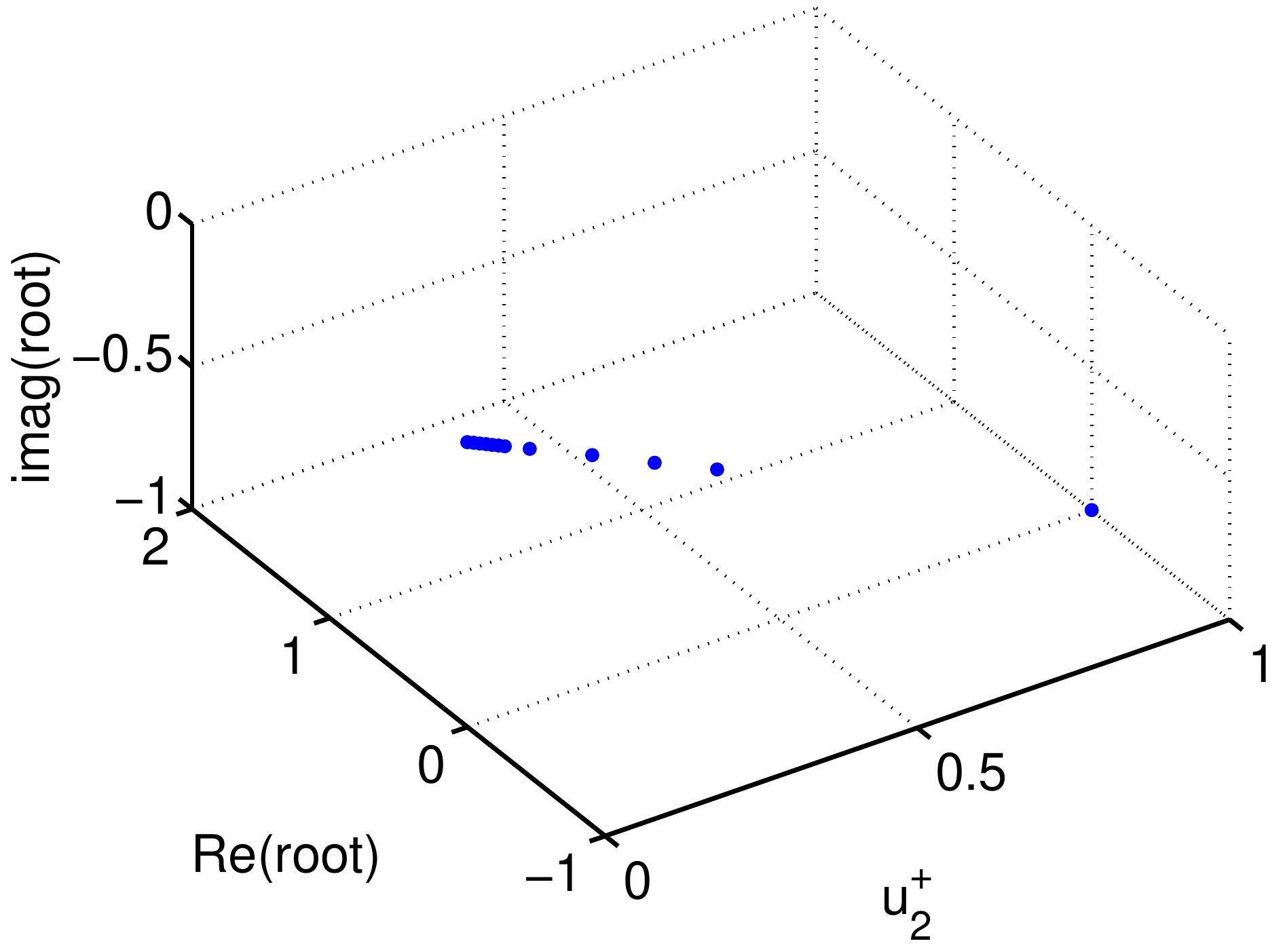}
\end{SCfigure}
%
%

\section{Viscous stability analysis: Evans function}\label{viscous_evans_function}

We also study the viscous linear stability analysis using Evans function techniques. Roughly speaking, the study of the Evans function proceeds as follows: consider the system,
\begin{equation*}
f_0(u)_t + \sum_{j=1}^dA_j(u)\partial_{x_j}u = \sum_{j,k=1}^d(B_{jk}(u)u_{x_k})_{x_j},
\end{equation*}
and make the change of coordinates $x_1 \to x_1 -s t$ to obtain,
\begin{equation*}
f_0(u)_t-sf_0(u)_{x_1}+\sum_{j=1}^d A_j(u)\partial_{x_j}u = \sum_{j,k=1}^d (B_{jk}(u)u_{x_k})_{x_j}. 
\end{equation*}
We linearize about a planar traveling wave solution $\bar u$, traveling in the direction $x_1$ to obtain,
\begin{equation*}
(\bar A^0u)_t -s(\bar A^0u)_{x_1} + \bar Cu+\sum_{j=1}^d \bar A_j\partial_{x_j}u= \sum_{j,k=1}^d \bar B_{jk} u_{x_kx_j} + \sum_{j=1}^d(dB_{j1}(\bar u) (u,\bar \partial_{x_1}u))_{x_j}, 
\end{equation*}
where $\bar A_j:= A_j(\bar u)$ and the linearization of $A_1(u)\partial_{x_1}u$ is $ \bar C u + \bar A_1\partial_{x_1}u$. Next we take the Fourier transform in the variables $\xi_2, \ldots, \,\xi_d$, and the Laplace transform in $t$ to obtain the eigenvalue problem,
\begin{equation}\label{ode_viscous}
\lambda \bar A^0 u - (\bar A_1)'u+\bar C u + \sum_{j=2}^d i\xi_j \widetilde A_j u + \sum_{j,k = 2}^d \xi_j\xi_k\bar B_{jk}u= \left( \bar B_{11} u' + \sum_{j=2}^d i\xi_j \bar B^j u - \hat A_1 u\right)', 
\end{equation}
where $\widetilde A_j:= \hat A_j + (\bar B_{j1})'$,  $\bar B^j:= \bar B_{j1}+\bar B_{1j}$, $\hat A_1u:= \bar A_1u-s\bar A^0u -dB_{11}(\bar u) (u,\bar \partial_{x_1}u)$, and $\hat A_ju:=  \bar A_ju -dB_{j1}(\bar u) (u,\bar \partial_{x_1}u)$ for $j\geq 2$.
We obtain our Evans function coefficient matrix by using the flux form, only wherever we see $\widetilde A^{\xi}:= \sum_{j=2}^d\xi_j \widetilde A_j$, we instead use $i(\bar A_1)' - i \bar C+\sum_{j=2}^d\xi_j \widetilde A_j$ (for further discussion, see \cite{BHLyZ1}). 

The Evans function  $(\lambda, \xi) \mapsto D(\lambda, \xi )$ consists of a measurement at  $x_1 =0$  of the transversality between the decaying manifolds of the ODE \eqref{ode_viscous} below when restricted to the spaces $x_1 \gtrless 0.$ 
It is  an analytic function of both its parameters whenever $\lambda$ is in the domain of consistent splitting
(see Section \ref{s:vweak}), which
in the present case includes $\{(\xi,\lambda): \Re \lambda \geq 0\}\setminus \{(0,0)\}$.
For our study of stability in a channel the relevant values of $\xi$ are
$\xi \in \frac{2\pi}{L}\mathbb{Z}$, $L$ being the width of the channel we are studying. 

According to the results  in \cite{ZS}, the Lopatinsky determinant is 
(in most cases) 
a first order approximation of the Evans function in the low frequency regime; consequently, inviscid instability implies viscous instability.\footnote{
Indeed this implication holds also in cases such as slow MHD shocks for which the approximation property
	is not known; see \cite[Theorem 2.30]{GMWZ6} and \cite[Lemma 8.3 and Prop. 83]{GMWZ6}. We note that even
	for slow MHD shocks, the first-order approximation property holds for {generic} frequency angles
	$(\xi,\lambda)$ \cite[Lemma 8.3]{GMWZ6}.}
The latter implication is our main  motivation in  the search for zeros of the Evans function, i.e., values of the spectral parameter $\lambda$ such that $D(\lambda, \xi) =0$ for some $\xi \in \frac{2\pi}{L}\mathbb{Z}$. 

Apart from constraint issues, substantial new difficulties in going from 1-D to multi-D Evans function computations arise:

\begin{enumerate}[leftmargin=*]
 \item  Number of equations/parameters: we deal with $5\times5$ system of equations with downstream/upstream shock conditions. The complexity, both mathematical and numerical, is enormous.  Symbolic computations are necessary to create the code without risk of human error. Fortunately, the $\beta$-model allows use of previously tested code (STABLAB; see more in  \cite{STABLAB});
 \item Unexpected issues with Evans function asymptotics related to Eulerian vs Lagrangian coordinates make computation of the multi-D numerical Evans function practically impossible; issues were only recently resolved in this project and, to knowledge of the authors, in only one other project (\cite{HLyZ2}).
\end{enumerate}

%

%

\subsection{Computing the profile}
To solve for the viscous profile numerically, we cut the domain in half and use a coordinate change to reflect the interval $(-\infty,0]$ to $[0,\infty)$. We use matching conditions at $x = 0$ and projective boundary conditions at $x = \infty$ that select the decaying solution. To solve the resulting three point boundary value problem, we use MATLAB's bvp5c solver with relative and absolute error tolerances set respectively to $10^{-6}$ and $10^{-8}$. 

\subsection{Computation of decaying manifolds and eigenfunctions}\label{numerics}
We recall that the Evans function takes the form
\begin{equation}
D(\lambda;\xi) = \det \left([W_1^{-\infty}(0;\lambda,\xi),...,W_k^{-\infty}(0;\lambda,\xi),W_{k+1}^{+\infty}(0;\lambda,\xi),...,W_n^{+\infty}(0;\lambda,\xi)]\right),
\notag
\end{equation}
where 
\begin{equation}
\frac{d}{dx}W_j^{\pm \infty}(x;\lambda,\xi) = A(x;\lambda,\xi)W_j^{\pm \infty}(x;\lambda,\xi) 
\label{eq:evans_ode}
\end{equation}
and $W_1^{-\infty},...,W_k^{-\infty}$ and $W_{k+1}^{+\infty},...,W_n^{+\infty}$ form a basis for the solution space of \eqref{eq:evans_ode} that decays as $x\to - \infty$ and as $x\to +\infty$, respectively. If $(\lambda_0,\xi_0,v_0)$ is an eigenvalue, Fourier mode, eigenfunction triple, then $D(\lambda_0,\xi_0) = 0$ and $v_0$ can be expressed as a linear combination of $W_1^{-\infty},...,W_k^{-\infty}$ when $x\in(-\infty,0]$ and as a linear combination of $W_{k+1}^{+\infty},...,W_n^{+\infty}$ when $x\in [0,+\infty)$. Hence, to solve for an eigenvalue $\lambda_0$ and eigenfunction $v_0$ corresponding to a fixed $\xi_0$, we may do as in \cite{Humpherys2015}, and solve \eqref{eq:evans_ode} as a boundary value problem with $\lambda$ as a free parameter. At $x = \pm \infty$, we use projective boundary conditions, $P^{\pm\infty}v_0(\pm \infty) =0$, which force the projection of $v_0$ onto the unstable and stable manifolds at $x = \pm \infty$, respectively, to be zero. The projective boundary conditions at $x = \pm \infty$ provide $k$ and $n-k$ boundary conditions, which leaves one additional boundary condition corresponding to the free parameter. We provide a phase condition, such as $\|v_0(0)\| = 1$ or a component of $v_0(0)$ is unity, which selects an eigenfunction from the family $\{ cv_0: c\in \C, c\neq0\}$. In practice, to numerically approximate $v_0$ we divide the domain into two parts, $ (-\infty,0]$ and $[0,+\infty)$,  and then perform the change of coordinates $x\to -x$ on $(-\infty,0]$, thus doubling the dimension of the system \eqref{eq:evans_ode} now posed on $[0,+\infty)$. We then pose the boundary value problem on the finite interval $[0,L]$ where $L$ is the truncation value approximating infinity as determined in solving the traveling wave profile. In the end, there is one phase condition given at $x = 0$ and $n$ projective boundary conditions given at $x = L$. 

To obtain an initial guess for the boundary value problem, we apply STABLAB's built in root finding capabilities, such as the method of moments or a two-dimensional bisection method using squares in the complex plane, to the Evans function to find a $\tilde \lambda_0$ which approximates the eigenvalue $\lambda_0$ of interest. To approximate the eigenfunction $v_0$, we set $$
W^L(x):= [W_1^{-\infty}(x;\tilde \lambda_0,\xi_0),...,W_k^{-\infty}(x;\tilde \lambda_0,\xi_0)]
$$
and $W^R(x):=[W_{k+1}^{+\infty}(x;\tilde \lambda_0,\xi_0),...,W_n^{+\infty}(x;\tilde \lambda_0,\xi_0)]$ and then find $C = (c_L, c_R)^T$ that minimizes $\|[W^L(0),-W^R(0)]C\|$ in the least squares sense subject to $\|C\| = 1$. 

In practice, solving for $W^L(x)$ and $W^R(x)$ is difficult because of competing modes of $A(x;\lambda,\xi)$ 
as $x\to \pm \infty$. Thus, we compute $W^L(x)$ and $W^R(x)$ using the method of continuous orthogonalization of \cite{HuZ1}. In this method, we set $W^L = \Omega_L\alpha_L$ and $W^R = \Omega_R\alpha_R$ where $\Omega_L$ and $\Omega_R$ are orthonormal basis of $W^L$ and $W^R$ respectively. In particular, as detailed in \cite{HuZ1}, $\Omega_L$ and $\alpha_L$ satisfy the
well-conditioned ODEs
\begin{equation}
\begin{split}
\Omega' & = (I-\Omega \Omega^*)A\Omega\\
\alpha' & = (\Omega^*A\Omega)\alpha.
\end{split}
\notag
\end{equation}
Thus, we can solve the ODE for $\Omega_L$ and $\Omega_R$ and then minimize $\|[\Omega_L(0),-\Omega_R(0)]C\|$ subject to $\|C\| = 1$, then afterward solve for $\alpha_L$ and $\alpha_R$ by initializing the associated ODE at $x = 0$. We note that solving for $\alpha_L$ from $x = 0$ to $x = -L$, or for $\alpha_R$ from $x = 0$ to $x = L$, is numerically well posed as error decays in this direction of integration. We then recover $W^L$ and $W^R$, which provides an initial guess for the boundary value problem described previously.

Fig. \ref{fig104} exemplifies the applicability of the numerical construction here described: it shows a graph of the real part of the variable $u_1$ of the eigenfunction associated to the bifurcating eigenvalue. Notice the loss in the planar structure, which is also pointed out in \cite{Mo} in the strictly parabolic case for $\mathcal{O}(2)$ steady bifurcations.

\subsection{Viscous stability diagram } 

To determine stability of the viscous shock waves, we compute the Evans function, similar to the Lopatinsky determinant, on a contour $\Omega_r:= \partial (\{z\in \C: \mathrm{Re}(z) \geq 0\}\cap \{z\in \C: |z| = r\}$, where $r$ is now chosen by curve fitting the Evans function to within 0.2 relative tolerance of its asymptotic behavior $D(\lambda) \sim C_1e^{C_2\sqrt{\lambda}}$, indicating that any zeros of the Evans function that may exist lie within $\Omega_r$; see Appendix \ref{winding_number_appendix} for more details. However, we do limit $R\leq 128$ for practical reasons since the time to compute $D(\lambda)$ becomes 
unreasonable for $|\lambda|$ too large. To compute the Evans function, we use the method of continuous orthogonalization (\cite{HuZ2}) 
described in Section \ref{numerics}, computed in pseudo-Lagrangian coordinates for better conditioning as described
in \cite{BHLyZ2}.
We initialize the Evans ODE with a basis that varies analytically in $\lambda$ via the method of 
Kato \cite{Kato} as described in \cite{HSZ}. 
All of these methods are built into the STABLAB platform with which we perform our computations \cite{STABLAB}.

\begin{table}[!th]
    \begin{tabular}{|c|c|c|c||c|c|c|c|c|}
        \hline
        $\gamma$ & $u_{1+}$ & $h_1$ & $\xi$ & WND & Root & Radius & Num Pnts & Run time \\
        \hline
        \hline
        5/3 & 0.0001 & 1.1 & 0.001 & 0 & NA & 128 & 539 & 197 \\
        \hline
        5/3 & 0.0001 & 1.1 & 0.1 & 1 & 1.2280e-04  &  128 & 491 & 162 \\
        \hline
        5/3 & 0.0001 & 1.5 & 0.05 & 1 & 7.0364e-05  & 128 & 517 & 171 \\
        \hline
        5/3 & 0.0001 & 8 & 0.2 & 0 & NA  & 128 & 543 & 381 \\
        \hline
        7/5 & 0.0001 & 1.5 & 0.005 & 0 & NA  & 128 & 577 & 182 \\
        \hline
        7/5 & 0.0001 & 4 & 0.05 & 1 & 1.1829e-05   & 128* & 599 & 282 \\
        \hline
        7/5 & 0.0001 & 16 & 0.1 & 0 & NA  & 128 & 527 & 747 \\
        \hline
        \hline
        5/3 & 0.01 & 1.1 & 0.001 & 1 & 2.8902e-05  & 128* & 341 & 161 \\
        \hline
        5/3 & 0.01 & 1.1 & 0.1 & 1 & 0.0024  & 128 & 297 & 128 \\
        \hline
        5/3 & 0.01 & 1.5 & 0.05 & 1 &  7.5206e-04 & 64* & 281 & 128 \\
        \hline
        5/3 & 0.01 & 8 & 0.2 & 0 & NA  & 128 & 361 & 295 \\
        \hline
        7/5 & 0.01 & 1.1 & 0.8 & 1 &  0.0050 & 128* & 287 & 139 \\
        \hline
        7/5 & 0.01 & 2 & 1.6 & 0 & NA  & 64* & 237 & 132 \\
        \hline
        7/5 & 0.01 & 8 & 1.6 & 0 & NA  & 8* & 133 & 163 \\
        \hline
        \hline
        5/3 & 0.2 & 1.1 & 0.005 & 1 &4.7035e-04   & 16* & 137 & 93.2 \\
        \hline
        5/3 & 0.2 & 1.1 & 0.2 & 1 & 0.0137  & 4* & 141 & 84.2 \\
        \hline
        5/3 & 0.2 & 1.5 & 0.1 & 1& 0.0044  & 32* & 159 & 112 \\
        \hline
        5/3 & 0.2 & 8 & 0.8 & 0 & NA  & 8* & 171 & 203 \\
        \hline
        7/5 & 0.2 & 1.5 & 0.05 & 1 & 0.0024  & 32* & 165 & 113 \\
        \hline
        7/5 & 0.2 & 4 & 0.1 & 1&  3.7694e-04  & 4* & 221 & 138 \\
        \hline
        7/5 & 0.2 & 16 & 0.1 & 0 & NA  & 8* & 251 & 521 \\
        \hline
    \end{tabular}
    \caption{Table providing computational details of the viscous stability study. The fifth through ninth columns respectively show the winding number of the computation, the location of the root (if applicable) computed with absolute tolerance of $5\times10^{-7}$, the outer radius of the contour on which the Evans function was computed, the number of points on the contour, and the time in seconds the computation took to run. The Evans function was computed on a semi-annulus with inner radius $10^{-5}$ and outer radius as stated. A * indicates that the radius was taken large enough that curve fitting the Evans function with its asymptotic behavior yields a relative error no greater than 0.2. \label{viscous_table}}
\end{table}


When $\gamma = 5/3$, we sample the winding number for the Fourier coefficient $\xi \in [0.001, 0.004, 0.007,$ $0.01, 0.04, 0.07, 0.1, 0.14, 0.17, 0.2]$ for various values of $u_1^+$ and $h_1$ and plot the resulting stability diagram in Fig. \ref{fig264}-(a). When $\gamma = 7/5$, we obtain the corresponding stability diagram plotted in Fig. \ref{fig258}. We note that the z-axis in Fig. \ref{fig258} indicates the value of $\xi$ to give a sense of which modes are unstable. In Table \ref{viscous_table} we indicate for various parameters the radius needed to enclose any potentially unstable eigenvalues and we indicate the unstable root when it exists. 

\subsection{On the symmetry of eigenfunctions and equivariance of the Evans function} It is not possible to conclude from our analytical results that the bifurcating eigenvalues have associated dimension 2 rather than $2 n$, $n \in \{2, 3, \ldots \}.$  On the other hand, as we discuss next, an interesting conclusion can be derived with regards to the symmetry of the decaying manifolds discussed in this section. 

By definition, in $\mathcal{O}(2)$ symmetric systems, if $v(x,y)$ is a solution then $\mathcal{R} v(x, -y)$ is also a solution, for  $\mathcal{R}$ an orthogonal matrix. Fourier transforming in the y-direction, we observe that  any real eigenvalue $\lambda$ has an eigenfunction $e^{iky}w(x)$ and also an eigenfunction $e^{-iky}\mathcal{R}w(x)$. However,  by complex symmetry so is $e^{iky}\overline{w}(x)$. Likewise, $e^{iky}\mathcal{R}\overline{w}(x)$. This suggests that an eigenvalue $\lambda=0$ associated with $k_*\neq 0$ should have total multiplicity 4, 2 for each of $\pm k_*$, unless the apparently non-generic situation occurs that  $\mathcal{R} w$ and $\overline{w}$ are constant multiples of one another.

The analogy to complex conjugation is apparent: given any eigenfunction $w$, form   $(w+ \mathcal{R}\overline{w})$, the \textit{real part}, and $(1/i)(w-\mathcal{R}\overline{w})$, the \textit{imaginary part} of the eigenfunction $w$.  It is easily seen that both of these are invariant under 
$$\mathcal{T}: f\to \mathcal{R}\bar f,$$ 
and span the space $\mathrm{span}\{w,\mathcal{R}\overline{w}\}$ contained in the eigenspace of $\mathscr{L}_{k_*}$.    The same reasoning gives a symmetric basis of the subspaces of decaying solutions of $(\lambda-\mathscr{L})w=0$ at $\pm \infty$.  So, we can construct an Evans function from these eigenfunctions, and whenever there is a zero, we can find an eigenfunction given by a real linear combination of them, which is thus itself symmetric under the mapping $\mathcal{T}$, that is, invariance under $\mathcal{T}$ is in fact generic. For there to be non-symmetric eigenfunctions, there would have to be a higher multiplicity of linear dependence.  Moreover, MHD gives an explicit example where the multiplicity is in fact 2. Indeed, if there are only 2 eigenvalues, then, 
choosing the representatives of eigenfunctions having symmetry, we see that this uses up all the dimensions and there cannot be more. Consequently, eigenfunctions may always be chosen with $\mathcal{O}(2)$ symmetry.  Furthermore, one can build an Evans function with $\mathcal{O}(2)$ symmetry and this detects ``nice" eigenfunctions having the desired symmetry; there may well be others, but this would be ``extra", and there is no reason they would need to be there, that is, they are not generic.

\subsection{Finding the critical destabilization parameter \texorpdfstring{$h_1$}{h1}}\label{s:finding}

In the following discussion we assume $\gamma = 5/3$ and $u_1^+ = 0.86$. Determining exactly where the stability transition occurs in $h_1$ for the Evans function is a little difficult because the contour on which we compute the Evans function comes close to a root near the stability transition. A closer approximation was tried, but we could only confirm that, for $\xi = 0.005$, the Evans function has a root when $h_1 = 2$, but it does not when $h_1 = 1.999$. Several values of $\xi$ were tested when $h_1 = 1.999$, and no zeros of the Evans function were found; on the other hand, a root is found when $h_1 = 2$, being approximately $2.86 \times 10^{-7}$. The Lopatinsky determinant for $h_1 = 2$ and $\xi =1$ has a root at $\lambda_0 = 4.26\times 10^{-4}$, which for $\xi =0.005$, corresponds to $\lambda_0 = 2.13\times 10^{-6}$. It was  verified that the Lopatinsky determinant has no root to the right of a the vertical line $\lambda = 10^{-4}$ for $h_1 = 1.995$, but it does for $h_1 = 1.996$. The 
contour cannot be taken  much closer to the imaginary axis than $10^{-4}$ because of the essential spectrum.

In summary, we can estimate that the stability transition for the Lopatinsky determinant occurs at $h_1 \approx 1.995$ and for the Evans function at approximately $h_1 = 2$.  As discussed in the introduction, this is due to discretization of Fourier
modes. 
The whole-space transition values agree, as a consequence of concavity of the associated spectral curves, 
illustrated in Fig. \ref{fig101}.

\begin{SCfigure}
  \centering
  \caption{Plot of the roots of the Lopatinsky determinant against $h_1$ for $\xi = 1$ when $\gamma = 5/3$ and $u_1^+ = 0.9$.}
\label{fig89}
\includegraphics[scale=0.4]{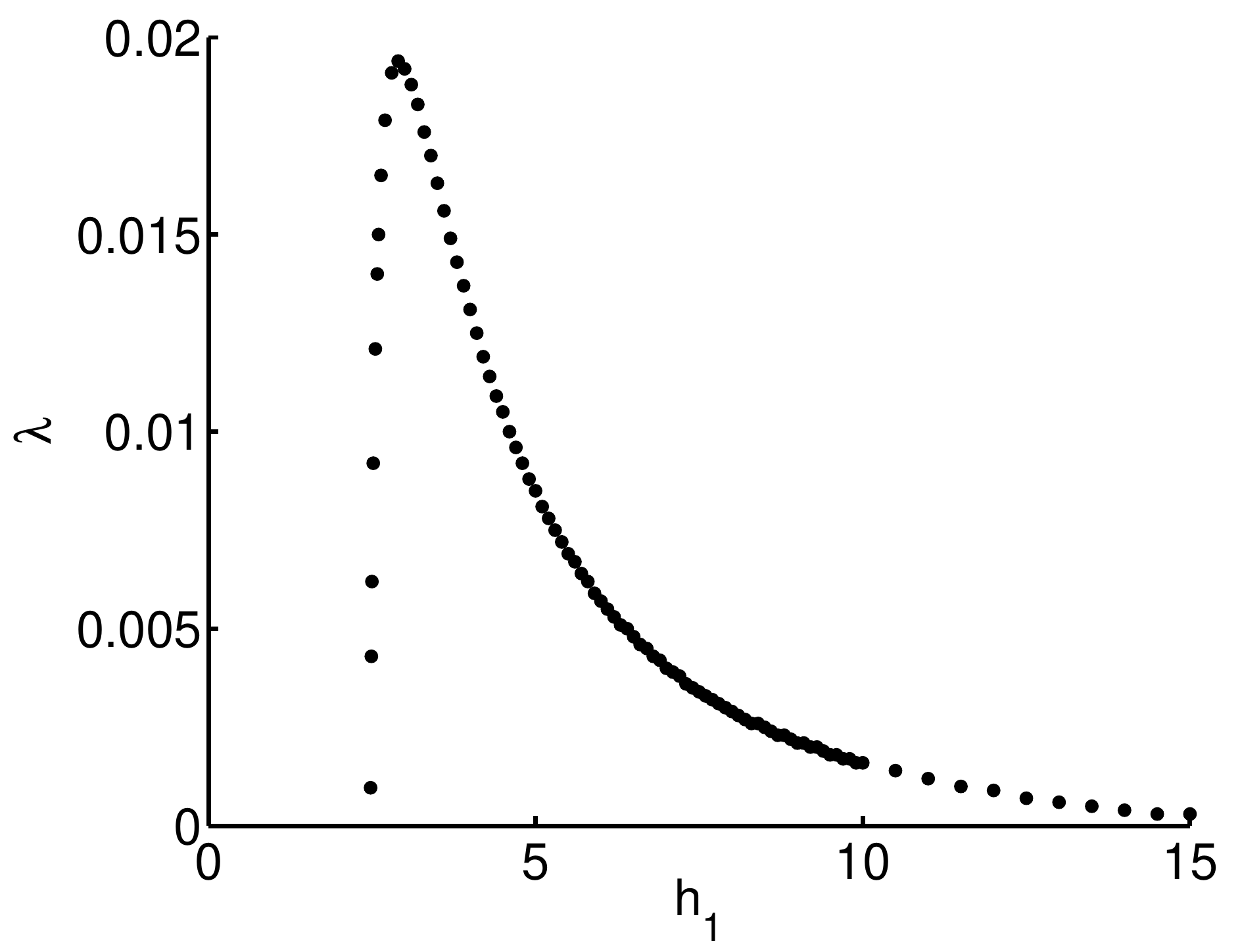}
\end{SCfigure}

\subsection{Verifying concavity}\label{s:concavity}
In the previous subsection, we have verified concavity of the critical spectral curve/agreement of 
(whole space) viscous and inviscid transition values for one (typical) choice of parameters, essentially by
force, by computing the critical spectral curve $\lambda_*(\xi)$ and approximating the second derivative.
In this subsection, we check concavity more efficiently using the implicit function theorem.

To verify concavity of the spectral curves $\lambda(\xi) = 0$ at the critical transition, we 
approximate the quantity 
$$
\sigma:=-\frac{\tilde D_{\rho}}{\tilde D_{\lambda_0}}|_{(\rho,\lambda_0,\xi_0) = (\eps,0,1)}
$$
using finite difference quotients, where $\tilde D(\rho,\lambda_0,\xi_0) = D(\rho \xi_0,\rho \lambda_0)$ is the Evans function in polar coordinates. 
This may be recognized as the negative of the ``effective viscosity coefficient'' of \cite{ZS,Z2,Z3,MR2448741},
with $\lambda(\xi)= \sigma \xi^2 + O(\xi^3)$.
Negativity of $\sigma$ corresponds to the ``refined stability condition'' of the references.

To approximate $\sigma$, we obtain an initial basis for the Evans function ODE by finding a 
basis $B_{\pm}$ at $(\lambda,\xi) = (0,\eps)$, and then multiplying $B_{\pm}$ on the left by an analytic projection onto the desired subspace, thus creating a locally analytically varying basis determining $D$. We then
compute the difference quotient approximation
\begin{equation}
\sigma \approx \frac{D(0,\eps)-D(0,2\eps)}{D(\eps^2,\eps)-D(0,\eps)}
\end{equation}
for various values of $\eps>0$.
We perform a convergence study to verify the correctness of our approximation of $\sigma$, using five evenly spaced values of $\eps$ between 1e-3 and 1e-6. We plot $\sigma$ against $u_1^--u_1^+$ in Fig. \ref{fig273}(a), demonstrating that $\sigma$ is always negative. Recall that $\sigma > 0$ indicates that instability may occur in the viscous system before it does in the inviscid system as the bifurcation parameter $h_1$ is increased. We also did a spot check to verify that we get the same value for $\sigma$ when we interpolate the curves $\lambda(\xi) = 0$ with quadratic polynomials for various values of $h_1$, and then interpolate the second derivative of these quadratic polynomials in the variable $h_1$ with a quartic polynomial, which we evaluate at $h_1 = H_*$. In Fig. \ref{fig273}(b), we demonstrate the quadratic interpolation of a typical curve $\lambda(\xi) = 0$. We note that we compute $\sigma$ only for $u_1^+ $ as small as $0.7005\approx U_*$, where $U_*$ is the value of $u_1^+$ at which shocks become inviscid stable at the minimum $h_1 = H^*$ value for which they are 3-shocks.

The clear conclusion from Fig. \ref{fig273}(a) is that the spectral curve is indeed concave at transition to instability,
for all relevant values of physical parameters, in the case of a monatomic gas $\gamma=5/3$.

\begin{figure}[htbp]
 \begin{center}
$
\begin{array}{lcr}
(a) \includegraphics[scale=0.25]{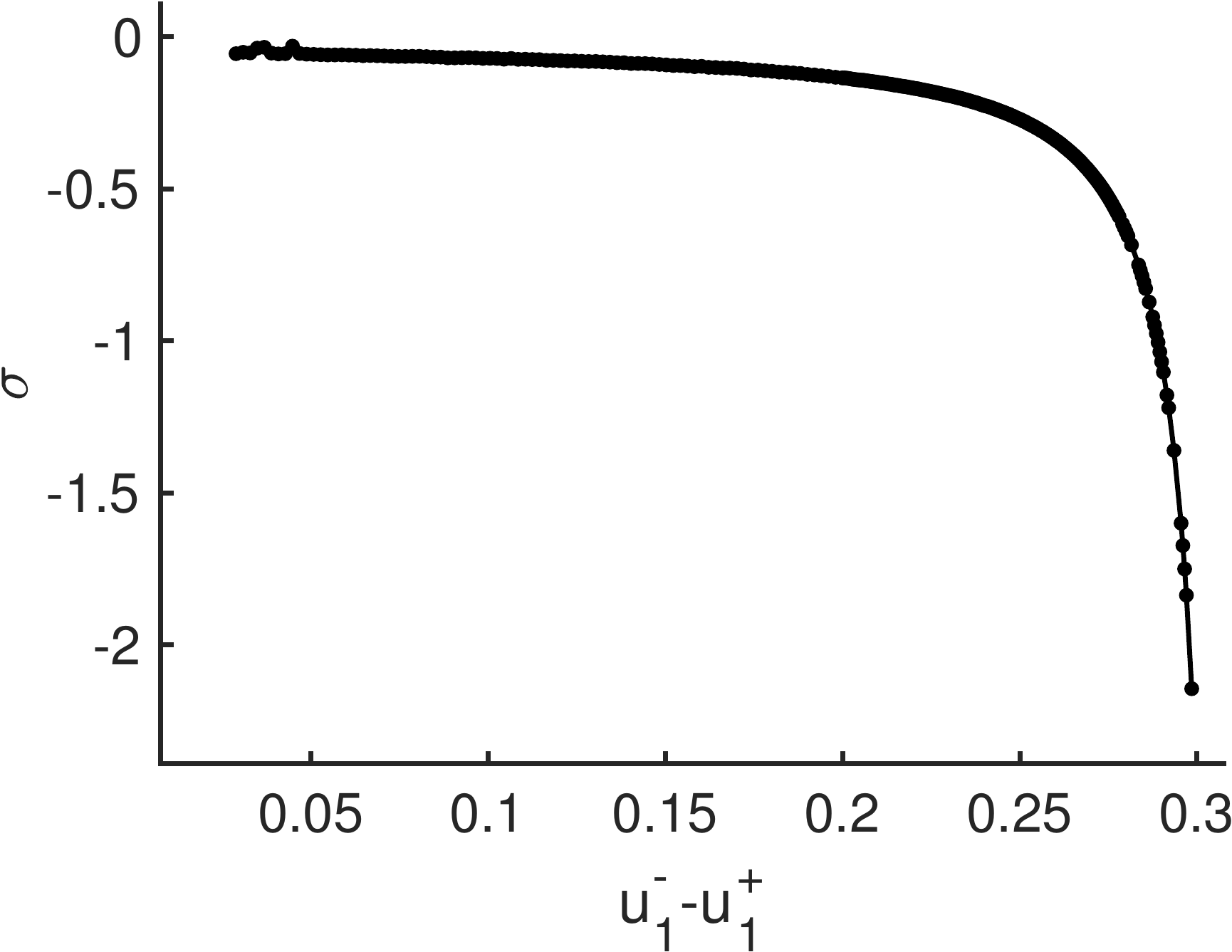} &(b) \includegraphics[scale=0.25]{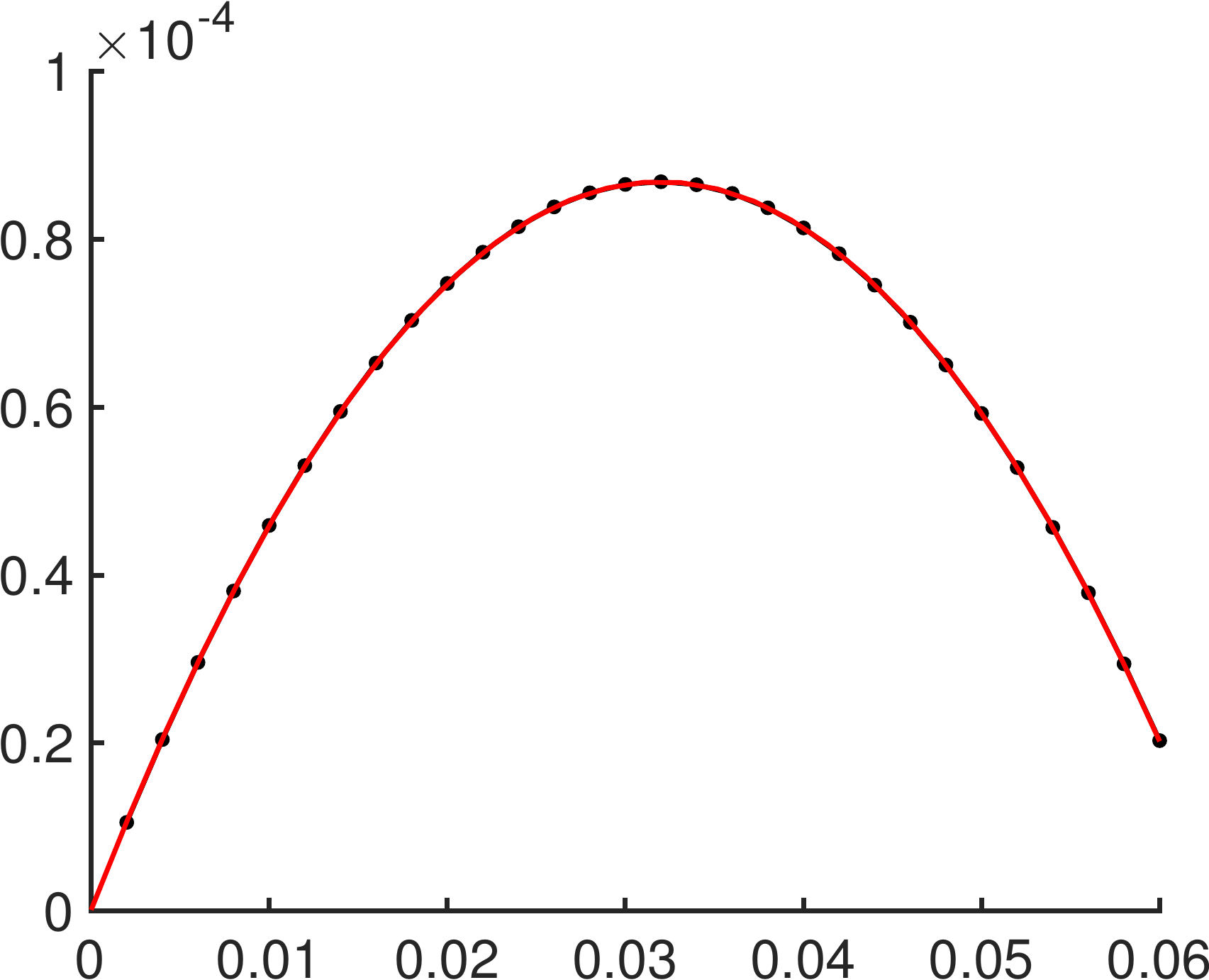}& (c) \includegraphics[scale=0.25]{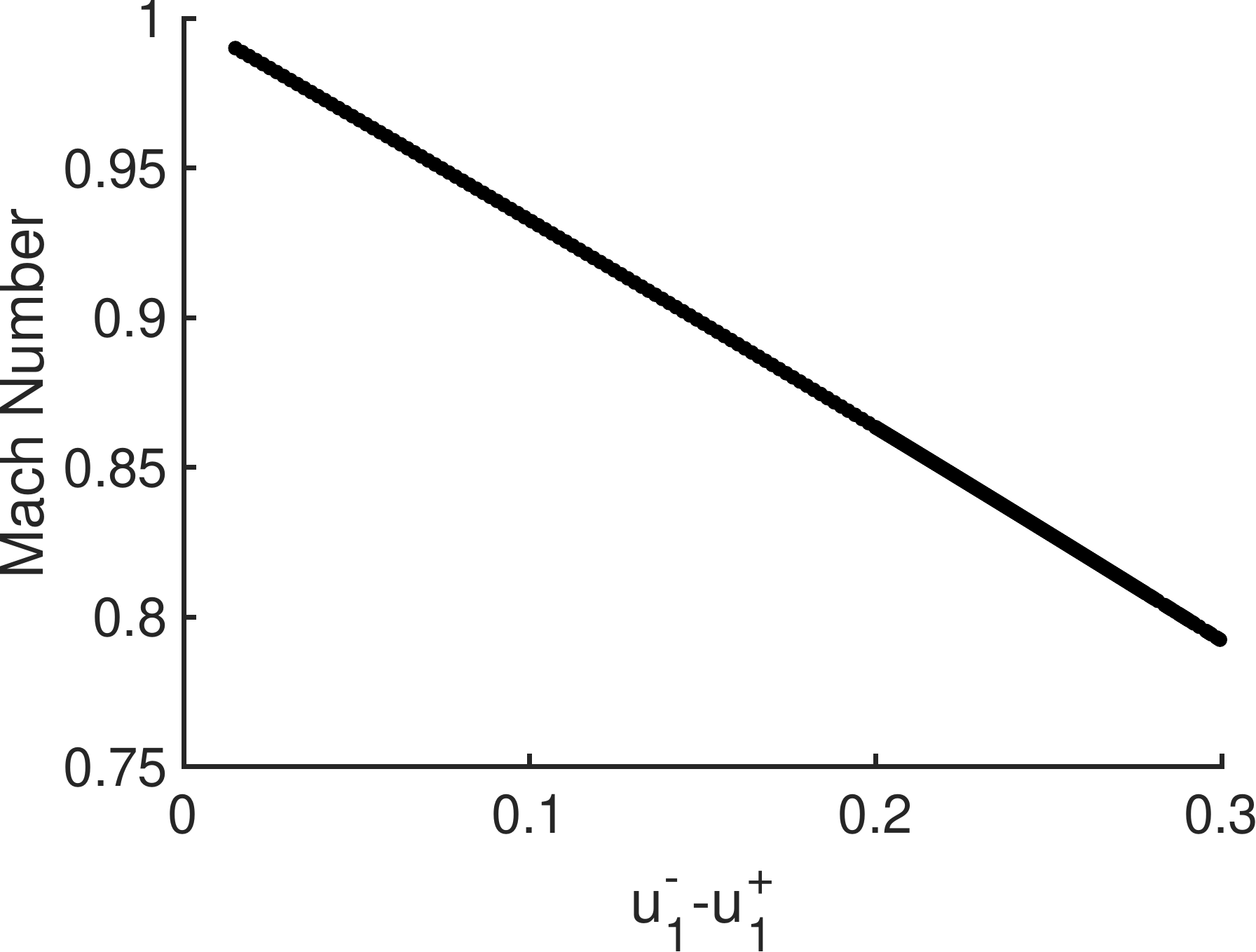}
\end{array}
$
\end{center}
\caption{(a)Plot of $\sigma$ against $u_1^--u_1^+$. (b) Plot of interpolation (red curve) of points $\xi_n$ for which $\lambda(\xi_n) = 0$. (c) A plot of the Mach number against $u_1^--u_1^+$ at the neutral stability curve. }
\label{fig273}
\end{figure}

\subsubsection{Finer points: analyticity and glancing}\label{s:finer}
The implicit function computation $\sigma= -\tilde D_\rho/\tilde D_{\lambda_0}$
of the second-order coefficient in $\lambda(\xi)=\sigma \xi^2 + \dots$, or, equivalently
$$
\lambda_0(\xi)= \sigma \rho + \dots,
$$
defined by $\tilde D(\rho, \lambda_0(\rho), 1)=0$, presupposes analyticity of $\tilde D$
in $(\rho, \lambda_0)$.
As pointed out in \cite{ZS,Z1}, analyticity holds away from ``glancing points'', defined as frequencies $\lambda_0$
for which $A_1^{-1}(\lambda_0 A_0 + i\xi_0 A_2)$ has neutral (i.e., zero real part) eigenvalues possessing a nontrivial
Jordan block.

However, for the parallel MHD equations, \cite[Lemma 7.2(ii)]{MeZ2} specialized to the 2-D case considered here
yields for $(\lambda_0, \xi_0)=(0, 1)$ that there is {\it always} a Jordan block of dimension 2, hence an associated
square-root singularity in the initializing decaying eigenspaces at both $x\to \pm \infty$, inherited by the manifolds
of decaying solutions.  This is readily verified by direct computation of the zero-eigenspace of $iA_1^{-1}A_2$,
which may be seen to have geometric multiplicity 2 but algebraic multiplicity 3.
Thus, we cannot simply appeal to nonglancing to conclude analyticity: there is always glancing!
On the other hand, the fact that {\it both} decaying manifolds at $\pm \infty$ have a square-root singularity
at $\lambda=0$ implies that the Evans determinant obtained as their exterior product, by a monodromy argument,
or simply by composing the two square roots to obtain a linear factor, is analytic, despite the presence of glancing 
modes.
This justifies our computations above.

\section{Additional description of numerics}\label{additional_numerics}
To create Fig. \ref{figs237and238} (a), we used the method of moments described in \cite{Bronski} to determine the roots of the Evans function. The moments were computed on the contour $\partial ((\{z\in \C: \mathrm{Re}(z)\geq 0\}\cap \{z\in \C: |z|\leq 0.01\})/ \{z\in \C: |z|\geq 10^{-4}\})$. To create Fig. \ref{figs237and238} (b), we used a forward finite difference scheme to approximate the derivative using the data given in Fig. \ref{figs237and238} (a).

\appendix

\section{On winding number computations}\label{winding_number_appendix}

The study  of unstable modes reduces to studying an eigenvalue equation: at the inviscid level through Lopatinsky determinant,  $\Delta(\lambda_{\mbox{lop}},\xi)$; at the viscous level through an Evans function, $D(\lambda_{\mbox{ev}},\xi)$. Due to analyticity of these two objects in their parameters, the search for growing  modes corresponds to verifying if, for a fixed $\xi^*$ there exists a root $\lambda_{\mbox{lop}}$ of the Lopatinsky determinant and a root $\lambda_{\mbox{ev}}$ of the Evans function  in the half space $\mathrm{Re}(z) >0$ of the complex space $\mathbb{C}$. This computation relies then on winding number computations, based upon the argument principle. An example of these computations can be seen in Figures \ref{fig233} and \ref{fig234}.

\begin{figure}[htbp]
        \centering
        \begin{subfigure}[b]{0.45\textwidth}
		\includegraphics[scale=0.4]{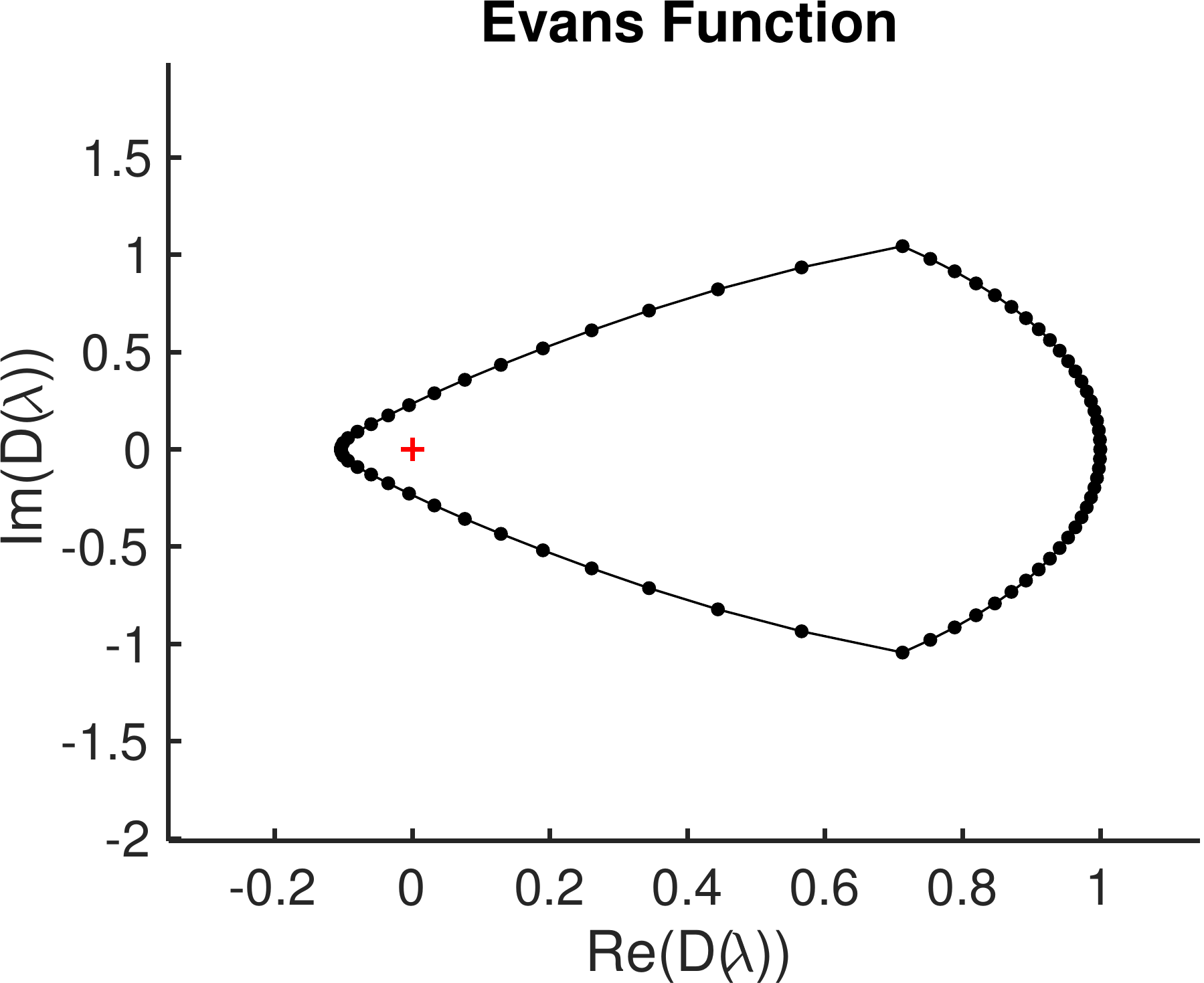}
		\caption{Plot of the image of the Evans function computed on a semi-circle contour, a contour with a vertical line centered at the origin and a half circle on the right connecting the end points of the vertical line, with radius 0.1 when $\gamma = 5/3$, $u_1^+ = 0.4$, and $h_1 = 2$.}
	\label{fig233}
        \end{subfigure}%
         \quad
        \begin{subfigure}[b]{0.45\textwidth}
		\includegraphics[scale=0.4]{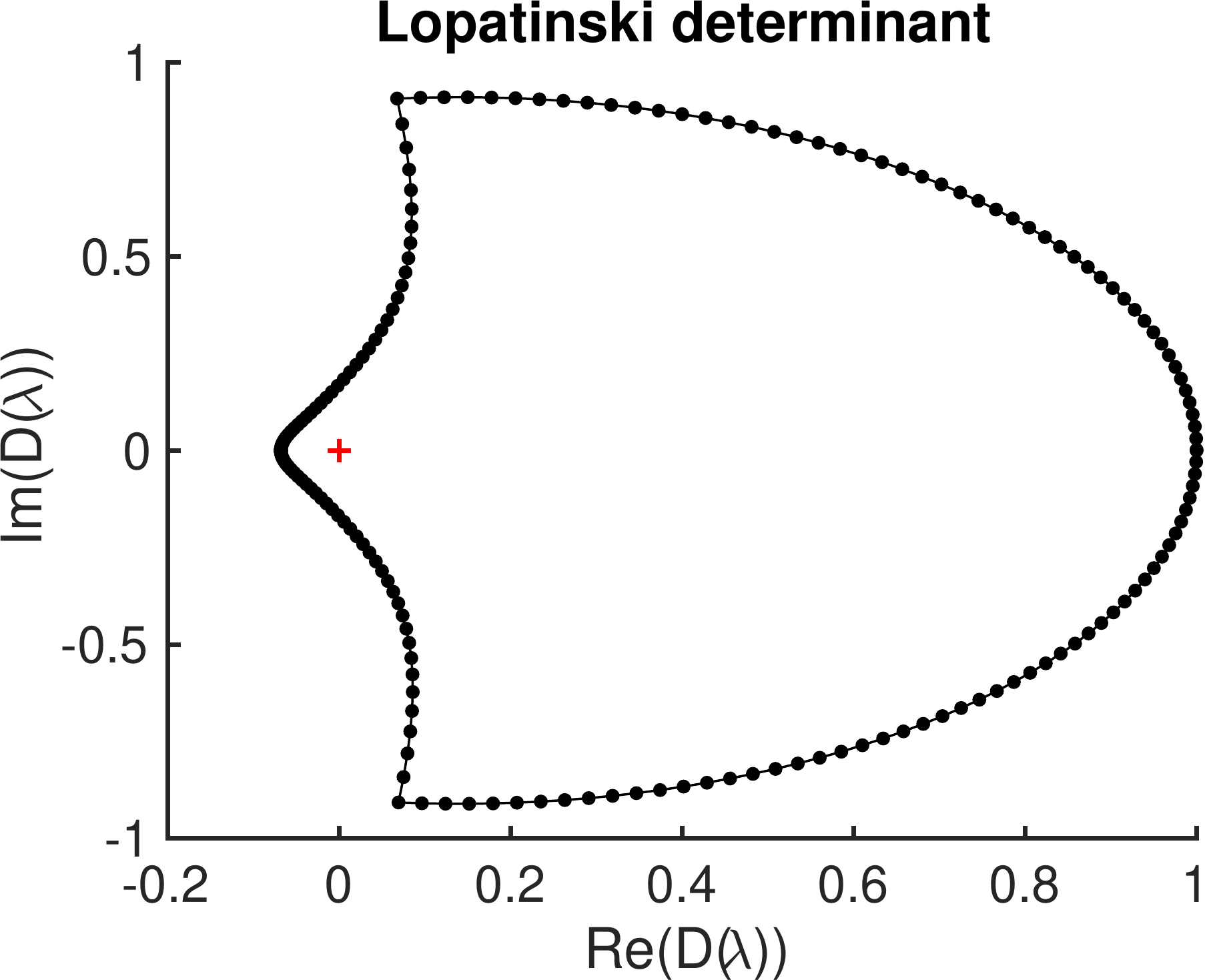}
		\caption{ Plot of the image of the Lopatinsky determinant computed on a semi-circle of radius 0.5 shifted right by 0.01 when $\gamma = 5/3$, $u_1^+ = 0.4$, and $h_1 = 2$. A red plus sign marks the origin.}
	\label{fig234}
	\end{subfigure}%
	\caption{Winding number computations for the viscous and inviscid system. A red plus sign marks the origin.}
\end{figure}

\section{Guide to \cite{FT} and \cite{BT1} \# 1: persistence of constraint condition, a second proof }\label{appendix_FT_proof}

We give here  a second proof of Proposition \ref{persistence} 
for the linearized MHD $\beta$-model,\footnote{Inherited from the nonlinear version, by linearity of the constraint.}
following \cite[Remark 3.2]{FT}. In what follows,  we write $\left( \mathcal{Q}_1\right)_{x_i}$ or $ \mathcal{Q}_{1,x_i}$ to denote  the partial derivative $\partial_{x_i}\left( \mathcal{Q}_1\right)$ of a quantity $\mathcal{Q}_1$ with respect to $x_i$. To begin with,  we  linearize  \eqref{eq1-c} about a shock profile $\overline{\mathcal{V}} = (\brh, \overline{u}_1, 0, \overline{h}_1,0)$ (thus, $\overline{u}_2 = 0$, $\overline{h}_2 = 0$), obtaining
\begin{subequations} 
\begin{eqnarray}
	(h_1)_t + (\overline{h}_1 u_2 - h_2 \overline{u}_1)_{x_2} + \beta (h_{1,x_1} + h_{2,x_2}) = 0; \label{eq38a}\\
 (h_2)_t  - (\overline{h}_1 u_2 - h_2 \overline{u}_1)_{x_1} = 0. \label{eq38b}
\end{eqnarray}
\end{subequations}
In this section we will only make use of the first of these equations.
\begin{equation}
	\mbox{(``+" case)} \quad \quad \quad (h_1^+)_t + (\overline{h}_1^+ u_2^+ - h_2^+\, \overline{u}_1^+)_{x_2} + \beta ( h_{1,x_1}^+ 
	+ h_{2,x_2}^+) = 0;  \label{eq39a}
\end{equation}
\begin{equation}
	\mbox{(``-" case)} \quad \quad \quad (h_1^-)_t + (\overline{h}_1^- u_2^- - h_2^- \,\overline{u}_1^-)_{x_2} + \beta ( h_{1,x_1}^- 
	+ h_{2,x_2}^-) = 0.  \label{eq40a}
\end{equation}
The linearized Rankine-Hugoniot conditions across a shock at $x_1 =0$ yield
\begin{subequations} 
\begin{eqnarray}
h_1^+ &=& h_1^-; \label{eq41-a}\\
\overline{h}_1 (u_2^+ - u_2^-) &=& \overline{u}_1^+\left(h_2^+ - \frac{\brh^+}{\brh^-}h_2^- \right) = \overline{u}_1^+\left(h_2^+ - R h_2^-\right). \label{eq41-b} 
\end{eqnarray}
\end{subequations}
We subtract equations \eqref{eq39a} and \eqref{eq40a},
\begin{equation}
 (h_1^+ - h_1^-)_t + \left(\overline{h}_1 (u_2^+ - u_2^-) - h_2^+\, \overline{u}_1^+ + h_2^- \,\overline{u}_1^-\right)_{x_2} + 
 \beta ( h_{1,x_1}^+ + h_{2,x_2}^+) -\beta( h_{1,x_1}^- + h_{2,x_2}^-) = 0,  \label{eq42a}
 \end{equation}
and we differentiate the equation \eqref{eq41-a} (respectively, \eqref{eq41-b})  with respect to $t$ (respectively, $x_2$), which is legitimate since we are not differentiating in any direction perpendicular to the shock front. We end up with
\begin{equation}
\left(\overline{u}_1^+(h_2^+ - R h_2^-) - h_2^+\, \overline{u}_1^+ + h_2^- \,\overline{u}_1^-\right)_{x_2} + 
\beta ( h_{1,x_1}^+ + h_{2,x_2}^+) -\beta( h_{1,x_1}^- + h_{2,x_2}^-) = 0 , \label{eq43a}
 \end{equation}
where $\left(\overline{u}_1^+(h_2^+ - R h_2^-) - h_2^+\, \overline{u}_1^+ + h_2^- \,\overline{u}_1^-\right)_{x_2} = 0$; thus, 
\begin{eqnarray} \label{boundary_condition}
\mathrm{div}(h^+)  = \mathrm{div}(h^-) \quad \mbox{at } \quad x_1 =0.
\end{eqnarray}
Now  we differentiate the equation \eqref{eq38a} (respectively, \eqref{eq38b}) with respect to $x_1$ (respectively, $x_2$), obtaining
$$\partial_t \mathrm{div}(h)  + \beta \partial_{x_1}\mathrm{div}(h) = 0.$$
Using that $\mathrm{div}(h)\Big|_{t=0} =0$ and the jump condition at $x_1 =0$ provided by the Rankine-Hugoniot condition \eqref{boundary_condition}, we see by the  characteristic method  that the only possible solution is $\mathrm{div}(h) = 0$ for $t>0$. 

\section{Guide to \cite{FT} and \cite{BT1} \# 2: a different approach towards the derivation of Rankine-Hugoniot conditions} \label{Rankine}

In this appendix we consider a different way to derive the Rankine-Hugoniot conditions upon linearization about a planar shock wave. There are several ways to do that, and in the context of MHD equations we refer to \cite{FT} and \cite{MeZ2}. As pointed out by Majda in \cite{Majda}, the study of stability of planar shocks reduces to a free boundary problem, in which a parameter measuring the deformation of the planar structure of the shock  is introduced. The main point of the analysis is to ``trade'' the latter unknown deformation by introducing a dynamic boundary condition at the linearized shock front. Now that we have given the rough idea we can put this heuristic on solid mathematical ground: assume initially that we are in a 2-D spatial domain. Throughout this section, the usual $l^2(\R^n)$ inner product is written  $\langle \cdot, \cdot\rangle$; a subindex $(\cdot)_a$  denotes the partial derivative with respect to the variable $a$, for $a \in \{t,x_1,x_2\}$. Let  $u(\cdot)$ be  a planar traveling wave with speed $s$ solving the following system of equations
\begin{eqnarray*} \label{dynamic_eq1}
 (f_0(u))_t + \sum_{i=1}^{2}[f_i(u)]_{x_i} = 0,
\end{eqnarray*}
where $u(x_1 -st) = u^{\pm}$ ($u^{\pm}$ are constants) whenever $x_1 -st \gtrless 0$ and $f_0$, $f_1$, $f_2 \in \mathscr{C}^{\infty}(\mathbb{R}; \mathbb{R}^n)$. The Rankine-Hugoniot conditions are given by
$$    -s [f_0(u)] + [f_1(u)] = 0, $$ where $[\cdot]$ denotes the jump across the shock.
We consider a perturbation of this system given by a function $u(\cdot) + v(\cdot)$, taking into account also perturbations in the shock front of the form $x_1 -\phi(x_2,t)=0$ for $\phi$ sufficiently smooth. In this case the Rankine-Hugoniot conditions are:
$$ - \phi_t[f_0(u+v)] - \phi_{x_2}[f_2(u+v)] + [f_1(u+v)]= 0.$$
As pointed out in Section \ref{Lopatinski_analysis}, we can take s=0. Linearizing the Rankine-Hugoniot conditions about the shock profile $u(\cdot)$ and the shock $x_1 -st=0 $  we have
\begin{equation} \label{dynamic_eq2}
 -\phi_t[f_0(u)] -\phi_{x_2}[f_2(u)] + [Df_1(u)v]= 0,
\end{equation}
where $Df_1(\cdot)$ denotes the Jacobian of the mapping $f_1$. We start with a trivial linear algebra result:
\begin{claim}\label{claim1}
We can choose  $n-2$  vectors $\{v_1, \ldots v_{n-2}\}$ in $\mathbb{R}^n$ so that $$\langle v_i , (-\phi_t[f_0(u)] -\phi_{x_2}[f_2(u)] )\rangle =0,$$ 
for all $i\in {1, \ldots , n-2}.$ Further, thanks to  \eqref{dynamic_eq2}, we must have that $\langle v_i , [Df_1(u)v] \rangle =0$.
\end{claim}
Each one of these vectors provide a ``static'' constraint. It turns out that we can actually find another vector $\widetilde{w}$ - independent of $v_1, \ldots, v_{n-2}$ - that is also orthogonal to $-\phi_t[f_0(u)] -\phi_{x_2}[f_2(u)]$. In this case, however, since $\phi_t$ and $\phi_{x_2}$ have a dynamic behavior (i.e., time dependence), we must expect the same for $\widetilde{w}$. The idea consists of looking for a vector of the form $\widetilde{w}= r \partial_t + s\partial_{x_2}$, where $r, s \in \mathbb{R}^n$ are unknowns still to be found. Recall that $[f_0(u)]$  and $[f_1(u)]$ are constant vectors, since $u$ is constant on both sides of the shock. Our aim is to satisfy $\langle r \partial_t + s\partial_{x_2} , -\phi_t[f_0(u)] -\phi_{x_2}[f_2(u)] \rangle =0$; expanding the latter, we obtain the equivalent expression
$$ \phi_{tt} \langle r , [f_0(u)]\rangle  + \phi_{x_2 x_2}\langle s, [f_2(u)] \rangle +\phi_{t x_2} \left( \langle s , [f_0(u)]\rangle +  \langle r ,[f_2(u)] \rangle\right) =0.$$
In order to verify this formula, it suffices to  find $r$ and $s$ in $\mathbb{R}^n$ such that 
\begin{equation} \label{dynamic_eq3}
\begin{split}
 \langle s , [f_0(u)]\rangle +  \langle r ,[f_2(u)] \rangle &= 0; \\
 \langle r , [f_0(u)]\rangle  &= 0;  \\
 \langle s, [f_2(u)] \rangle  &= 0.
\end{split}
\end{equation}
Our next step consists of defining a projection $u \mapsto  \mathbb{P}_v(u)$ that projects the vector $u\in \R^n$ in the space spanned by the space $v\in \R^n$. The following properties of this operator, whose proofs we omit,  are standard results in linear algebra:
\begin{claim}\label{projectors}
 Given a vector $v \in \mathbb{R}^n$, $v \neq 0$ and defining the operator $\mathbb{R}^n \ni u \mapsto \mathbb{P}_v(u)$, then the following properties are satisfied:
 \begin{enumerate}[label=(\roman*), ref=\thetheorem(\roman*)]
\hangindent\leftmargin
  \item[i.] $\mathbb{P}_v(0) = 0$;
  \item[ii.] $\mathbb{P}_v(v) = v$;
  \item[iii.]\label{projectors:c} $\mathbb{P}_v(w) = 0 \quad \Leftrightarrow \quad v \perp w \quad \Leftrightarrow \quad \mathbb{P}_w(v) = 0.$
 \end{enumerate}
\end{claim}
Let $[f_0]^{\perp}\in \mathbb{R}^n$ (respectively, $[f_1]^{\perp}\in \mathbb{R}^n )$  be defined so that $\{[f_0]^{\perp},[f_0],  v_1, \ldots, v_{n-2}\}$  (respectively,  $\{[f_1]^{\perp},[f_1],  v_1, \ldots, v_{n-2}\}$) spans $\mathbb{R}^n$.
Define  
\begin{align}\label{setting_mu}
r =  \mathbb{P}_{[f_0]^{\perp}}([f_2]) \quad  \mbox{and} \quad  s =  \widetilde{\mu} \mathbb{P}_{[f_2]^{\perp}}([f_0]),
\end{align}
where $\widetilde{\mu} \in \mathbb{R}$ will be defined later. Without loss of generality, we assume that  $\mathbb{P}_{[f_0]^{\perp}}([f_2])\neq 0$ and $\mathbb{P}_{[f_2]^{\perp}}([f_0])\neq 0$; indeed, using Claim \ref{projectors}(iii) and  the fact that $\{v_1, \ldots v_{n-2}\} \perp \mathrm{span}\{[f_0], [f_1] \}$,  if  one of these projections is vanishes then the $[f_0]$ and $[f_1]$ are linearly dependent, thus the choice of a $(n-1)^{th}$ orthogonal vector is reduced to a trivial problem.  On the other hand, it is easy to see that the last two equations in \eqref{dynamic_eq3} are satisfied. So we proceed as follows: we plug $r$ and $s$ as defined in \eqref{setting_mu} in the first equation of \eqref{dynamic_eq3} to find that $\widetilde{\mu}$ should be
$$ \widetilde{\mu} =  - \frac{\langle \mathbb{P}_{[f_0]^{\perp}}([f_2]) ,[f_2(u)]\rangle }{\langle \mathbb{P}_{[f_2]^{\perp}}([f_0]) , [f_0(u)]\rangle}.  $$ 
Applying the result to \eqref{dynamic_eq2} we see that $\displaystyle{\langle r \partial_t + s\partial_{x_2} , -\phi_t[f_0(u)] -\phi_{x_2}[f_2(u)] \rangle =  0}$, which implies that
\begin{align} \label{dynamic_eq4}
\langle r \partial_t + s\partial_{x_2} ,[Df_1(u)v] ]  \rangle =0.
\end{align}
 The latter equation is called  a \textit{dynamic Rankine-Hugoniot condition}; notice that it is independent of the perturbation $\phi$ of the shock front.   

 \br A generalization of the results in this appendix to cases with a higher number of vectors is a bit tricky. Indeed,
  one needs to find more  vectors $r_{j,k},    1\leq j\leq d-1,  1\leq k \leq d-1,$ corresponding to $d-1$ dynamic conditions  (this in the generic case that $[f_0], [f_2],...,[f_d] $ are independent, so that there are $n- d$ static conditions and in total there will be the needed $n-1$ total conditions for extreme shock). These must have similar properties to those in \eqref{dynamic_eq3}, that is: a)  $ r_k(\lambda, \eta):=\lambda r_{1,i} + i\eta r_{2,k} +$ ... $i\eta_d r_{d,k} ,    1\leq k\leq d-1$ be independent for each $(\lambda,\eta)\neq 0$,  and b)  $\langle r_k(\lambda,\eta), \lambda[f_0]+ \eta_2[f_2] +...\rangle=0.$ Now, the case $d=2$ is easy, because we only need find one of these vectors, and there are $d + d(d-1)/2$ homogeneous constraints, so we can always find a nontrivial solution;  to see that it is non-vanishing for all $(\lambda, \eta)$ one just looks and sees a contradiction if one entry but not the other is vanishing. However, when $d\geq 3$, this is not a simple task. Furthermore,  there are degenerate cases where $[f_0], [f_2],...,[f_d]$ are not independent, which need to be treated slightly differently also it seems.
In conclusion: the method works for the current purposes. In spite of its limitations, it illustrates a case where one can find  Rankine-Hugoniot conditions explicitly, even when the shock front is an unknown.
 \er
 
\subsection{The dynamic jump condition of \cite{FT}}

Recall  that $ u_{2}^{\pm} = 0$, $ h_2^{\pm} = 0.$  We make use of equations \eqref{eq1} in order to derive the vectors $f_0$, $f_1$ and $f_2$. 
\begin{eqnarray*}
 f_0 = \left(\begin{array}{c}
        \rho  \\
        \rho u_1 \\
        \rho u_2\\
        h_1 \\
        h_2
       \end{array}\right), \quad f_1 = \left(\begin{array}{c}
                                \rho u_1 \\
                                \rho u_1^2 - \frac{h_1^2}{2} + \frac{h_2^2}{2} + a\rho^{\gamma} \\
                                \rho u_1u_2  - h_1 h_2 \\
                                0 \\
                                u_1 h_2 - h_1 u_2
                               \end{array} \right) , \quad f_2 = \left(\begin{array}{c}
                                                                               \rho u_2 \\
										 \rho u_1 u_2  - h_1 h_2  \\
										 \rho u_2^2  + a\rho^{\gamma} + \frac{h_1^2}{2} -\frac{h_2^2}{2} \\
										  u_2 h_1 - h_2 u_1 \\
										  0
                                                                              \end{array}\right).
                            \end{eqnarray*}
The jumps across the shock are:
\begin{eqnarray*}
 [f_0] = \left(\begin{array}{c}
       \brh^+ - \brh^-\\
       0\\
       0\\
       0\\
       0\\
       \end{array}\right),\quad [f_2] = \left(\begin{array}{c}
                                                                                                 0\\
                                                                                                 0\\
                                                                                                 a \left\{(\brh^+)^{\gamma} - (\brh^-)^{\gamma}\right\} \\
                                                                                                 0\\
                                                                                                 0\\
                                                                              \end{array}\right).
                            \end{eqnarray*}
It is easy to find the vectors mentioned in claim \ref{claim1}: $v_1 = e_2$, $v_2 =e_4$ and $v_3 = e_5$. To calculate the static Rankine-Hugoniot conditions we need the Jacobian of $f_1$,  $Df_1((\rho, u_1, 0, h_1, 0))$:
\begin{eqnarray*}
 Df_1(\brh, \bu_1, 0, \bh_1, 0) = \left(\begin{array}{ccccc}
       \bu_1 & \brh  & 0 & 0& 0 \\
       \bu_1^2 + a\gamma\brh^{\gamma-1} & 2\brh \, \bu_1 & 0 & -\bh_1 & 0\\
       0  & 0 & \brh \, \bu_1& 0 & -\bh_1 \\
       0 & 0 & 0&0 &0 \\
       0 & 0 & -\bh_1& 0& \bu_1\\
       \end{array}\right).
       \end{eqnarray*}
Using the notation in \cite[\S 3.2]{FT}, the result in \ref{claim1} provides the static Rankine-Hugoniot conditions. The persistence of the divergence free condition (Section \ref{mathematical_settings} and Appendix \ref{appendix_FT_proof}) implies that
 \begin{eqnarray}\label{RH_eq0}
\left[\overline{h}_1\right]=0.
 \end{eqnarray}
Let $\mathcal{V} =(\rho, u_1,u_2 , h_1, h_2)$. We use $\displaystyle{\langle v_1, [Df_1(\brh, \bu_1, 0, \overline{h}_1, 0) \mathcal{V})]\rangle =0 }$ to get 
 \begin{multline} \label{RH_eq1}
\bu_1^+\left(1 + \frac{1}{M^2}\right)\rho^+ - \bu_1^+R^2\left(1 + \frac{1}{M^2R^{\gamma +1}} \right) \rho^- + 2\brh^+(u_1^+ -u_1^-) - \frac{\overline{h}_1}{(\bu_1^+)}(h_1^+ - h_1^-) =0.  \end{multline}
The condition
\begin{eqnarray}\label{RH_eq2}
\langle v_2, [Df_1(\brh, \bu_1, 0, \overline{h}_1, 0) \mathcal{V})]\rangle =0
\end{eqnarray}
does not provide anything relevant (since the $4^{th}$ row of $Df_1$ is zero), whereas the last static condition,  
 $
\langle v_3, [Df_1(\brh, \bu_1, 0, \overline{h}_1, 0) \mathcal{V})]\rangle =0 $ gives
\begin{eqnarray}\label{RH_eq3}
-\overline{h}_1(u_2^+ -u_2^-) + \bu_1^+h_2^+ - \bu_1^-h_2^- =0. 
 \end{eqnarray}
\br
If we use the preserved constraint and appropriate normalization (see \cite[\S 3.1]{FT}), equation \eqref{RH_eq1} yields the first Rankine-Hugoniot condition in \cite[Equation (44)]{FT}, while equation \ref{RH_eq3} corresponds to the last equation in \cite[Equation (44)]{FT}. 
\er
Now we derive the dynamic Rankine-Hugoniot condition. We begin by calculating the projections mentioned at claim \ref{projectors}:
\begin{eqnarray*}
 \mathbb{P}_{[f_0]{\perp}}([f_2]) = [f_2] - \frac{\langle[f_0],[f_2] \rangle }{\langle[f_2],[f_2] \rangle }  \overbrace{=}^{[f_0]\perp[f_2]} [f_2], \qquad  \mathbb{P}_{[f_2]{\perp}}([f_0]) = [f_0] - \frac{\langle[f_2],[f_0] \rangle }{\langle[f_0],[f_0] \rangle }  \overbrace{=}^{[f_0]\perp[f_2]} [f_0].
\end{eqnarray*}
Set $r = [f_2]$  and $s = \lambda [f_0]$, where $\lambda$ is defined by $\displaystyle{\lambda  = -\frac{\langle[f_2],[f_2] \rangle }{\langle[f_0],[f_0] \rangle }}.$
The dynamic Rankine-Hugoniot condition will be given by equation \eqref{dynamic_eq4}, $\displaystyle{\langle [f_2]\partial_t + \lambda[f_0]\partial_{x_2},[Df_1(\brh, \bu_1, 0, \overline{h}_1, 0) \mathcal{V})]\rangle =0}$, i.e., %
\begin{eqnarray*}
\langle [f_2], [Df_1(\brh, \bu_1, 0, \overline{h}_1, 0)\partial_t \mathcal{V}] \rangle +  \lambda\langle[f_0],[Df_1(\brh, \bu_1, 0, \overline{h}_1, 0) \partial_{x_2}\mathcal{V})]\rangle= 0. 
\end{eqnarray*}
Define the scalars $\mathcal{A}_1$ and $\mathcal{A}_3$ so that  $[f_0] = \mathcal{A}_1 e_1$, $[f_2] = \mathcal{A}_3 e_3$; it follows that  $\mathcal{A}_1 = \brh^+(\frac{R-1}{R})$ and, by equation \eqref{RH_shock}, $\mathcal{A}_3 = -\brh^+(\bu_1^+)^2 (1-R)$. Clearly, $\lambda = -(\mathcal{A}_3/\mathcal{A}_1)^2$. After some computations, the above equation is reduced to 
\begin{eqnarray*}
 \mathcal{A}_1 \left\{(u_2^+ - u_2^-) -\frac{q}{M^2}(h_2^+ - h_2^-) \right\}_t  - \frac{\mathcal{A}_3}{\bu_1^+}\left\{ (\rho^+  - R\rho^-) +(u_1^+ -\frac{u_1^-}{R})\right\}_{x_2}  =0.
\end{eqnarray*}
Following \cite[\S 3.2]{FT}, we use the normalization  $\widetilde{t} \simeq t\bu_1^+$ to rewrite  the previous equation  as
\begin{eqnarray*}
 \mathcal{A}_1 \left\{(u_2^+ - u_2^-) -\frac{q}{M^2}(h_2^+ - h_2^-) \right\}_{\widetilde{t}}  - \frac{\mathcal{A}_3}{(\bu_1^+)^2}\left\{ (\rho^+  - R\rho^-) +(u_1^+ -\frac{u_1^-}{R})\right\}_{x_2}  =0, \end{eqnarray*}
 or, equivalently,
 \begin{align}\label{appendix_temp1}
  \frac{(R-1)}{R} \left\{(u_2^+ - u_2^-) -\frac{q}{M^2}(h_2^+ - h_2^-) \right\}_{\widetilde{t}}  -(R-1)\left\{ (\rho^+  - R\rho^-) +(u_1^+ -\frac{u_1^-}{R})\right\}_{x_2}  =0.
 \end{align}
At this point we make some substitutions. First, we solve the Rankine-Hugoniot condition \eqref{RH_eq3} 
for  $(u_2^+ - u_2^-)$ and substitute the result in the first  term of \eqref{appendix_temp1}. Then  we solve \eqref{RH_eq1} for ${u}_1^+$,  giving
\begin{equation*}
 u_1^+ = u_1^- - \frac{\bu_1^+}{2\brh^+}\left(1 + \frac{1}{M^2}\right)\rho^+ +\frac{\bu_1^+}{2\brh^+}R^2\left(1 + \frac{1}{M^2R^{\gamma +1}} \right) \rho^-+ \mathcal{G}(M, R, \gamma,  \bu_1^+,\brh^{\pm},h_1^{\pm}),
\end{equation*}
and plug the result in the second term of \eqref{appendix_temp1}; note that $\mathcal{G}$ is independent of  $u_1^{\pm}$ and $\rho^{\pm}$. This gives

%
\begin{eqnarray*}
 \left\{\frac{1}{q}(h_2^+ - Rh_2^-) - \frac{q}{M^2}(h_2^+ - h_2^-) \right\}_{\widetilde{t}} - \left\{ R(1- b_1)\rho^+ + R(b_2 -R)\rho^- + (R-1)u_1^-\right\}_{x_2} =0,
 \end{eqnarray*}
 where $\displaystyle{ b_1 = \frac{\bu_1^+}{2\brh^+}\left(1 + \frac{1}{M^2}\right)}$ and $\displaystyle{b_2 =\frac{\bu_1^+}{2\brh^+}R^2\left(1 + \frac{1}{M^2R^{\gamma +1}} \right)} $. Multiplying by $\displaystyle{\frac{M^2}{M^2  -q^2 }}$, we obtain the \textit{dynamic Rankine-Hugoniot condition}
%
 %
 %
 \begin{equation}\label{dynamic_RH_1}
 \left\{\frac{1}{q}(h_2^+) + \frac{q^2 - RM^2}{q(M^2 - q^2)}( h_2^-) \right\}_{\widetilde{t}}  
 + \left\{ b_3\rho^+ + b_4\rho^- + b_5 u_1^-\right\}_{x_2} =0,
 \end{equation}
 where  $\displaystyle{ b_3 = -R(\frac{( M^2 -1 )}{(M^2 -q^2)}\rho^+}$, $\displaystyle{b_4 = -\frac{RM^2}{(M^2 -q^2)}(b_2 -R)}$, and $\displaystyle{b_5 = -\frac{M^2(R-1)}{M^2 -q^2}}$. 
 
 This final equation  does not agree with the \textit{dynamic Rankine-Hugoniot condition} \cite[Equation (44)(ii)]{FT} of Freist\"uhler \& Trakhinin. For, according to their calculations,
\begin{equation*}
 \left\{u_2^+  - u_2^- \right\}_{\widetilde{t}}  
 + \left\{ b_3\rho^+ + b_4\rho^- + b_5 u_1^-\right\}_{x_2} =0 \quad \mbox{and} \quad h_2^+ - Rh_2^- - q u_2^+  +qu_2^- =0  \quad \mbox{at} \quad x_1 =0
 \end{equation*}
 which implies that 
 \begin{equation}\label{dynamic_RH_2}
 \left\{\frac{1}{q}(h_2^+ - Rh_2^-) \right\}_{\widetilde{t}}  
 + \left\{ b_3\rho^+ + b_4\rho^- + b_5 u_1^-\right\}_{x_2} =0 \quad \mbox{at} \quad x_1 =0.
 \end{equation}
A more careful analysis  shows that Equation  \eqref{dynamic_RH_2} is not a linear combination of \eqref{dynamic_RH_1} and the other jump conditions in \eqref{RH_eq0}-\eqref{RH_eq3} and the interior equations evaluated on $x_1 =0$. Indeed, if that were the case then $\partial_{\widetilde{t}} \{h_2^-\} =0$. However there is no way of obtaining $h_2^-$ as a linear combination of  the jump conditions (because it would introduce a linear term in $h_2^+$) or interior equations evaluated on $x_1 =0$ (because it would introduce variables that are spatial derivatives in $x_1$). One can conclude that the the dynamic Rankine-Hugoniot condition in \cite[Equation (44)]{FT} is incorrect. 

This small  error  did not do much harm to the main calculations of Freist\"uhler \& Trakhinin. In particular they obtained they obtained the correct order for the root of the Lopatinsky determinant (see \cite[Equation (61)]{FT}), since in the large magnetic field we have $q \to \infty$ and the coefficients of \eqref{dynamic_RH_1} and \eqref{dynamic_RH_2} only 
differ by an $\mathcal{O}(\eps^3)$ term.

\br
In the $\beta$-model the derivation of the Rankine-Hugoniot conditions using the reasoning presented in this section follows similar lines; the main differences are that
\begin{eqnarray*}
 f_0 = \left(\begin{array}{c}
        \rho  \\
        \rho u_1 \\
        \rho u_2\\
        h_1 \\
        h_2
       \end{array}\right), \quad f_1 = \left(\begin{array}{c}
                                \rho u_1 \\
                                \rho u_1^2 - \frac{h_1^2}{2} + \frac{h_2^2}{2} + a\rho^{\gamma} \\
                                \rho u_1u_2  - h_1 h_2 \\
                                \beta h_1\\
                                u_1 h_2 - h_1 u_2
                               \end{array} \right) , \quad f_2 = \left(\begin{array}{c}
                                                                               \rho u_2 \\
										 \rho u_1 u_2  - h_1 h_2  \\
										 \rho u_2^2  + a\rho^{\gamma} + \frac{h_1^2}{2} -\frac{h_2^2}{2} \\
										  u_2 h_1 - h_2 u_1 + \beta h_2\\
										  0
                                                                              \end{array}\right).
                            \end{eqnarray*}
The jump across the shock are:
\begin{eqnarray*}
 [f_0] = \left(\begin{array}{c}
       \brh^+ - \brh^-\\
       0\\
       0\\
       0\\
       0\\
       \end{array}\right),\quad [f_2] = \left(\begin{array}{c}
                                                                                                 0\\
                                                                                                 0\\
                                                                                                 a \left\{(\brh^+)^{\gamma} - (\brh^-)^{\gamma}\right\} \\
                                                                                                 0\\
                                                                                                 0\\
                                                                              \end{array}\right),
                            \end{eqnarray*}
and finally,
\begin{eqnarray*}
 Df_1((\brh, \bu_1, 0, \bh_1, 0)) = \left(\begin{array}{ccccc}
       \bu_1 & \brh  & 0 & 0& 0 \\
       \bu_1^2 + a\gamma\brh^{\gamma-1} &2\brh \, \bu_1 & 0 & -\overline{h}_1 & 0\\
       0  & 0 & \brh \, \bu_1& 0 & -\overline{h}_1 \\
       0 & 0 & 0&\beta &0 \\
       0 & 0 & -\overline{h}_1& 0& \bu_1\\
       \end{array}\right).
       \end{eqnarray*}
       Furthermore, unlike the case presented before, in equation \ref{RH_eq2} we would obtain $h_1^+ = h_1^-$, which we already know for MHD due to persistence of the constraint \eqref{constraint} (see Section \ref{persistence_constraint}).                            
\er
%


\bibliographystyle{plain}
\bibliography{refs}
\end{document}